\documentclass{amsart}
\usepackage{hyperref}
\usepackage{times}
\usepackage{verbatim}
\usepackage[T1]{fontenc}
\usepackage[utf8]{inputenc}
\usepackage[english]{babel}
\usepackage{amssymb}
\usepackage{amsthm}
\usepackage{amsmath}
\usepackage{amstext}
\usepackage{multirow}
\usepackage{graphicx}
\usepackage[all]{xy}
\usepackage{mathrsfs}
\usepackage{xcolor}
\usepackage{fixfoot}
\usepackage{stmaryrd}
\usepackage{tensor}
\usepackage{mdwlist}
\usepackage{enumitem}
\usepackage{MnSymbol}
\usepackage{mathscinet}
\usepackage{combelow}
\usepackage{tikz-cd}
\usepackage{xspace}
\newtheorem{thm}{Theorem}[section]

\newtheorem{lemma}[thm]{Lemma}
\newtheorem{prop}[thm]{Proposition}

\newtheorem{cor}[thm]{Corollary}

\newtheorem{THM}{Theorem}

\theoremstyle{definition}
\newtheorem{dfn}[thm]{Definition}

\newtheorem{question}{Question}

\theoremstyle{remark}
\newtheorem{remark}[thm]{Remark}

\newcommand{\mb}{\mathbb}
\newcommand{\mc}{\mathcal}
\newcommand{\R}{\mb R}
\newcommand{\C}{\mb C}

\newcommand{\Q}{\mb Q}

\newcommand{\sS}{\mc S}
\newcommand{\Hj}{\mb H}
\newcommand{\scrH}{\mathscr{H}}

\newcommand{\F}{\mc F}
\newcommand{\G}{\mc G}
\newcommand{\cC}{\mc C}
\newcommand{\cD}{\mathscr{D}}
\makeatletter
\newcommand{\eqnum}{\leavevmode\hfill\refstepcounter{equation}\textup{\tagform@{\theequation}}}
\makeatother
\DeclareMathOperator{\Orth}{O}

\DeclareMathOperator{\Isom}{Isom}
\DeclareMathOperator{\Lieisom}{\mathfrak{isom}}
\newcommand{\To}{\longrightarrow}
\DeclareMathOperator{\Aff}{Aff}
\DeclareMathOperator{\Lie}{Lie}
\DeclareMathOperator{\Ker}{Ker}

\newcommand{\barF}{\overline{\mathcal{F}}}
\newcommand{\wtX}{\widetilde{X}}
\newcommand{\wtsX}{\widetilde{X^\sharp}}
\newcommand{\wtF}{\widetilde{\mathcal{F}}}
\newcommand{\wtG}{\widetilde{\mathcal{G}}}
\newcommand{\wtbarF}{\widetilde{\barF}}
\renewcommand{\Im}{\operatorname{Im}}
\newcommand{\Aut}{\mathrm{Aut}}
\newcommand{\Autfix}{\mathrm{Aut}^\mathrm{Fix}}
\newcommand{\knr}{\ensuremath{\mathscr{K}_{\mathrm{nR}}}\xspace}
\DeclareMathOperator{\codim}{codim}

\DeclareMathOperator{\nd}{nd}
\newcommand{\frg}[1]{\mathfrak{g}\left(#1\right)}

\newcommand\restr[2]{{
  \left.\kern-\nulldelimiterspace % automatically resize the bar with \right
  #1 % the function
  \vphantom{\big|} % pretend it's a little taller at normal size
  \right|_{#2} % this is the delimiter
  }}
\usepackage{todonotes}

\makeatletter
\newcommand{\xdashrightarrow}[2][]{\ext@arrow 0359\rightarrowfill@@{#1}{#2}}
\def\rightarrowfill@@{\arrowfill@@\relax\relbar\rightarrow}
\def\arrowfill@@#1#2#3#4{%
  $\m@th\thickmuskip0mu\medmuskip\thickmuskip\thinmuskip\thickmuskip
   \relax#4#1
   \xleaders\hbox{$#4#2$}\hfill
   #3$%
}
\makeatother

\DeclareMathOperator{\rg}{rk}
\DeclareMathOperator{\SO}{SO}

\DeclareMathOperator{\Chi}{\mathfrak{X}}

\DeclareMathOperator{\PSL}{PSL}
\DeclareFixedFootnote{\repnote}{ These notations are precisely explained in Subsection \ref{sec:sheavesautomorphisms} but should be meaningful when looking at Picture \ref{pic:SecNorm}.}

\usepackage{mathtools}
\usepackage{ifmtarg}
\addtocounter{section}{0}             % Start with section 1
\numberwithin{equation}{section}       % Number formulas within sections


\title[K\"ahler foliations]{Structure of  K\"ahler foliations with negative transverse Ricci curvature}

\author[B. Claudon]{Beno\^it Claudon}
\address{Univ Rennes, CNRS, IRMAR - UMR 6625, F-35000 Rennes, France}
\email{benoit.claudon@univ-rennes.fr}

\author[F. Touzet]{Fr\'ed\'eric Touzet}
\address{Univ Rennes, CNRS, IRMAR - UMR 6625, F-35000 Rennes, France}
\email{frederic.touzet@univ-rennes.fr}

\thanks{The first named author would like to thank the Institut Universitaire de France for providing excellent working conditions. The authors benefits from the support of the French government “Investissements d’Avenir” program integrated to France 2030, bearing the following reference ANR-11-LABX-0020-01.}

\begin{document}
\begin{abstract} 
We investigate the structure of transversely K\" ahler foliations with quasi-negative tranverse Ricci curvature. In particular, we prove a de Rham type decomposition theorem on the leaf space where we characterize each factor. This can be seen as a foliated analog of Nadel and Frankel's uniformization theorem for canonically polarized manifold. This is also related to the works of the second named author on codimension one singular foliations with transverse hyperbolic structure. Further properties are established when the ambient manifold is compact K\"ahler and the foliation is holomorphic. 
\end{abstract}
\maketitle

\sloppy
\tableofcontents

\section{Introduction}\label{S:results}

\subsection*{Preliminary warning} Unless otherwise stated, the objects considered here such as manifolds, foliations, functions, tensors$\dots$ are supposed to be smooth. For the sake of notational simplicity, we will denote by the same symbol (typically $\F$) a foliation/distribution and its tangent bundle. 
\medskip
%\subsection{Frankel-Nadel's uniformization Theorem}
%{\color{red} Let $X$ a canonically polarized compact complex manifold...} \footnote{{\color{red} Recall that this means that the canonical bundle is ample, in particular $X$ is projective and admits a K\" ahler-Einstein metric.}}
\subsection{Statement of the main results}\label{SS:main-results} Let $\F$ be a transversely K\"ahler foliation of complex codimension $n$ on $X$ a compact (differentiable) manifold; it is worth mentioning here that $X$ is not necessarily endowed with a complex structure. We denote by $J_\F$ the holomorphic transverse structure (see \S\ref{SS:transverse Kähler}) and by $\overline{g}$ the transverse K\"ahler metric. In some holomorphic coordinates $(z_1,\ldots,z_{ n})$ parameterizing the local space of leaves, it reads  $$ \overline{g}= \sum_{i,{ j}}g_{i\bar{ j}}dz_i  d\bar{z_j}$$
where $g_{i\bar{ j}}$ depends only of the transverse variables $(z_1,\ldots,z_{n})$. The foliation is thus equipped with two basic closed $(1,1)$ forms. Namely the fundamental form $\omega$ of $\overline{g}$ and the (transverse) Ricci form $\gamma=\mathrm{ Ric}(\overline{g})$ respectively defined in the previous local transverse coordinates as 
\[\omega=\sqrt{-1}\sum_{i,j}g_{i\bar{ j}}dz_i\wedge  d\bar{z_j} \]
and
\[\gamma=\mathrm{ Ric}(\overline{g})= -\sqrt{ -1}\partial\bar{\partial}\log \left(\frac{\omega^{ n}}{{|dz_1\wedge\cdots \wedge dz_{ n}|}^2}\right)=-\sqrt{ -1}\partial\bar{\partial}\log \left(\textrm{det}(g_{i\bar{j}})\right). \]

In this paper, we will make use of the following assumptions.
\begin{enumerate}[label=(A\arabic*)]
	\item \label{A:Riccineg} The tranverse Ricci form $\mathrm{ Ric}(\overline{g})$ is \textit{quasi-negative}: $\mathrm{ Ric}(\overline{g})\leq 0$ and  $\mathrm{ Ric}(\overline{g})<0$ in the transverse direction (\emph{i.e.} has maximal rank $n$) somewhere.
	\item \label{A:homor} $\F$ is \textit{homologically orientable}, \emph{i.e.} the top basic cohomology group $H^{2n}(X/\F)$ is non-trivial, and then generated over $\R$ by the class of the transverse volume $\omega^{n}$ (see Theorem \ref{th:poincareduality}).
\end{enumerate}
Motivated by Nadel/Frankel's uniformization results (see the paragraph below Theorem~\ref{T:CdeRhamdec}) and by the codimension one situation (see \S~\ref{SS:codim 1}), we can then prove the following statement.

\begin{THM}\label{T:A}
Let us assume that the pair $(X,\F)$ satisfies \ref{A:Riccineg} and \ref{A:homor} above.
\begin{enumerate}
	\item There exist on $X$ two regular foliations $\G$ and $\overline{\F}$ containing $\F$ and such that  $\F=\G\cap \overline\F$.
	\item The foliations $\G$ and $\overline\F$ are holomorphic with respect to the complex structure $J_\F$ on the normal bundle $N\F$ . Moreover, $\G/\F$ and $\overline{\F}/\F$ are orthogonal and parallel with respect to $\overline{g}$ (as subbundles of $N\F$).
	\item The leaves of $\overline\F$ are the topological closure  of the leaves of $\F$. In particular, they are closed. The leaf space $X/\overline\F$ is a compact K\"ahler orbifold with quasi-negative Ricci curvature, in particular of general type.
	\end{enumerate}
\end{THM}

We denote by  $\widetilde{\bullet}$ the lift on the universal cover $\widetilde{X}$ of any object previously defined. To get a more in-depth description of the foliations $\G$ and $\overline{\F}$ we first formulate the following result.
\begin{THM}\label{L:BSpaceleaves}
The leaves of $\wtF$ are closed and the leaf space $\wtX/\wtF$ is a complete K\"ahler orbifold (with respect to the metric $\tilde{\overline{g}}$).
\end{THM}

Let $\mathcal{G}_1$ and $\mathcal{G}_2$ be the foliations induced on $\wtX/\wtF$ by $\wtG$ and $\wtbarF$. Note that they induce an infinitesimal splitting
(in the orbifold category):
$$T( \wtX/\wtF)=\mathcal{G}_1\oplus \mathcal{G}_2.$$
We have then the following geometric description of these foliations.
\begin{THM}\label{T:CdeRhamdec}
\leavevmode
\begin{enumerate}
	\item The leaves of $\G_2$ are all isometric to a Hermitian  symmetric space $\mathscr{H}$ of  non-compact type.
	
	\item The leaves of $\G_1$ are all isometric to a  K\"ahler complete  orbifold $\mathscr{K}$ with quasi-negative Ricci curvature.
	
	\item The aforementioned infinitesimal splitting gives rise to a global decomposition (de Rham decomposition)
\[\wtX/\wtF=\mathscr{K}\times \mathscr{H}.\]

	\item \label{I:denserepresentation}  Let us consider the natural diagonal and isometric action of $\pi_1(X)$ (with respect to the decomposition above). Then this action is minimal (\emph{i.e.} dense) on the second factor and discrete cocompact on the first factor $\mathscr{K}$. 
\end{enumerate}	
\end{THM}

The most interesting situation occurs when the Hermitian symmetric factor, equivalently the bounded symmetric domain $\mathscr{H}$ is not reduced to a point in the decomposition of $\wtX/\wtF$. Let $\Isom(\mathscr{H})$ be the Lie group of isometries of $\mathscr{H}$ with respect to its Bergman metric. The elements of the identity component $G=\Isom^0(\mathscr{H})$, a semi-simple real algebraic group without compact factors,  are holomorphic transformation maps of $\mathscr{H}$ and $\mathscr{H}$ is identified with the coset space $G/K$ where $K$ is a maximal compact subgroup of $G$.  The leaves of $\G_1$ coincide with the fibers of a submersion $f: \wtX\to \mathscr{H}$. Possibly after replacing $X$ by a finite étale cover, $f$ is equivariant with respect to some representation $\rho: \pi_1(X)\to G$ with dense image and which is in addition transversely holomorphic (see \S\ref{SS:extension}, as well as Theorems~\ref{T:harmonic-kahler} and~\ref{T:harmonic-holomorphic}). Actually, as it will be highlighted in the sequel,  this map $f$ is the unique $\rho$-equivariant map with respect to (the lift of) some suitable bundle-like metric on $X$.

These statements shall be seen as foliated analogues of results obtained by Nadel \cite{Nadel} and Frankel \cite{Frankel-annals}. Both authors have studied the geometry of the universal covering $\wtX$ of $X$ a complex projective manifold with ample canonical bundle.\footnote{Let us recall that it is equivalent to saying that $X$ admits a K\"ahler--Einstein metric with negative Ricci curvature.} Nadel proved that $\Aut^0(\wtX)$ is a semi-simple Lie group having no compact factors and he conjectured the existence of a splitting $\wtX\simeq \mathscr{K}\times \mathscr{H}$ with $\Aut^0(\mathscr{K})=\{\mathrm{Id}\}$ and $\mathscr{H}$ being a bounded symmetric domain (equivalently a Hermitian symmetric space of non-compact type).  This was first confirmed in the case of surfaces \cite[Theorem~0.2]{Nadel} and then in full generality \cite[Theorem~0.1]{Frankel-annals}. Theorem~\ref{T:CdeRhamdec} can thus be seen as a foliated version of the above-mentioned splitting where, in our setting, the Hermitian factor encapsulates the ``non-trivial part'' of the dynamic.

\subsection{Examples}\label{SS:examples}

To illustrate the above-mentioned statements, we give examples with increasing complexity.
\begin{enumerate}[label=(\alph*)]
\item The most basic construction consists in considering the foliation defined by a submersion $X\to M$ with connected fibers of a compact manifold onto a compact K\" ahler manifold  $(M,g_M)$ with quasi-negative Ricci curvature. Note that $M$ is of the general type (\emph{i.e.}, the canonical bundle $K_M$ is big) by virtue of Riemenschneider's Theorem \cite{Riemen73}.
\item Another family of examples is given by the so-called suspension process. Let $N$ be a compact manifold, $\widetilde{N}$ its universal cover and let us consider a representation $\rho: \pi_1(N)\to \textrm{Aut}( M)$ of the fundamental group of $N$ into the holomorphic transformation group of $M$ with quasi-negative Ricci curvature as above. Recall that the latter group is finite. Up to averaging $g_M$, one can suppose that $M$ is equipped with a $\textrm{Aut}(M)$-invariant K\"ahler metric. Let $X$ be the quotient manifold $(\tilde{ N}\times M)/\pi_1(N)$ defined by the natural diagonal action of  $\pi_1(N)$. This latter  has a structure of fiber bundle $M\rightarrow X\rightarrow N$ over $N$ and carries a natural foliation $\F_H$ transverse to the fibers, namely the projection of the horizontal foliation on $N\times M$.  This foliation has the sough properties but this construction does not gives rise to an interesting example from the dynamical viewpoint: the leaf space $X/\F_H$ identifies with the K\"ahler orbifold $M/ \textrm{Im}\ \rho$. 
\item Now, suppose given on a compact manifold $N$ a  \textit{minimal  foliation} (\emph{i.e.} with dense leaves) $\F_N$ whose transverse geometry is  locally modeled on a Hermitian symmetric space of non-compact type $\mathscr H= G/K$. The foliation carries a natural transverse invariant metric $\overline{g_N}$ induced by the $G$-invariant K\" ahler--Einstein metric on $\mathscr H$. Let $(M,g_M)$ be a compact K\"ahler manifold with quasi-negative Ricci curvature. The product $X=M\times N$ is equipped with an induced  foliation $\F$ of the same rank which restricts to each vertical fiber $\{m\}\times N\simeq N$ to $\F_N$.  This foliation is transversely K\"ahler with respect to the transverse metric $g_M\oplus \overline{g_N}$ and satisfies moreover hypothesis \ref{A:Riccineg} and \ref{A:homor}. Here, the leaves closure are precisely the vertical fibers. Moreover, if we fix a Hermitian symmetric space of non-compact type $\mathscr H$, it is possible to exhibit a compact foliated manifold $(N,\F_N)$ as above. Indeed, let $\mathscr H={ \mathscr H}_1\times\cdots\times { \mathscr H}_p$ be the decomposition of $\mathscr H$ into irreducible symmetric factors. It is well known since Borel's work \cite{Borelcompact63} (see also  \cite[Section~IX.4.7, Theorem~C]{Margulis91} and \cite[Corollary~18.7.4]{WitteMorris2015}) that for every $i$, there exists a discrete torsion free subgroup of holomorphic isometries $\Gamma_i$ of  
$\mathscr H_i \times \mathscr H_i$  acting cocompactly and diagonally and such that $\Gamma_i$ acts densely on each factor. It is then sufficient to consider the projective manifold $N= ( \mathscr H\times\mathscr H)/\prod_i \Gamma_i$ and to take as 
$\F_N$ the holomorphic foliation which lift to the horizontal (or vertical) one on the universal cover $\mathscr H \times \mathscr H$.  One can also mix with the suspension construction by considering  the foliation $\F$ lying on ${N\times M}$ obtained as the intersection $\F= \F_H\cap {\pi^*} \F_N$. With the notations of Theorem~\ref{T:A},  $\F_H$ and  ${\pi^*} \F_N$ correspond respectively to  $\overline{\F}$ and $\G$.  
\end{enumerate}

Apart this use of irreducible uniform lattices, we are not aware of other examples of somewhat different nature.  In this setting, it is worth mentionning  that, under special circumstances, Zimmer has  shown that the ``holonomy group'' of a minimal  Riemannian foliation on a compact manifold with non-compact semi-simple structural Lie algebra is of ``arithmetic nature'' \cite{Zimmer88}. 
As pointed out in the following subsection, when the ambient manifold is algebraic/K\"ahler, we have at our disposal a powerful machinery that enable to connect the representation $\rho$ (see te paragraph after Theorem~\ref{T:CdeRhamdec}) to arithmetic lattices in semi-simple Lie group.

\subsection{The case where $X$ is K\"ahler}\label{S:kahlercase}
When $X$ is a compact complex manifold and $\F$ is a holomorphic foliation satisfying the conditions \ref{A:Riccineg}(with respect to the natural holomorphic transverse structure) and \ref{A:homor} above.  The second item of Theorem \ref{T:A} implies that $\F_1$ and $\F_2$ are also holomorphic (the tranverse complex structure being the one induced by the complex structure of the ambient manifold $X$). Moreover, if $X$ is K\" ahler, the homological orientability is automatically fullfilled. In this case, the $\rho$-equivariant map $f$ is holomorphic. Examples are provided by foliations defined by a subline bundle $L\subset \Omega_X^p$ such that $c_1(L)$ is semi-positive with the maximal possible rank $=p$ everywhere, see Proposition~\ref{P:HNRclass}. 

Assume now for simplicity that $X$ is projective and that $G=H_\mathbb R$ consists of the real points of a simple algebraic group $H$.
As $f$ has maximal rank, the representation $\rho$ tends to be rigid in as an element of $\mathrm{Hom}(\pi_1 (X), H_{\mathbb C})$. The foliation $\mathcal{G}_1$ is defined as the kernel of the  differential $d^{1,0}f$, which can be regarded as the Higgs field attached to the representation $\pi_1(X)\to G_\C= \mathrm{Aut}^0 ({ \mathfrak g}_\C)$.
The minimality of $\mathcal{G}_1$ prevents from the existence of a  morphism $\varphi:X\to Y$, $\dim_{ \mathbb C }(Y)\leq \mathrm{rank}_{ \mathbb C } {H}$ through which $\rho$ factors. With this at hands, together with similar results valid in the context of Zariski dense representations of the fundamental group $\pi_1(X)$ to $p$-adic simple group, we can infer that the local system defined by $\rho$ is a direct factor of some $\mathbb Z$-variation of Hodge structures. In particular, thanks to the metric properties  of period domains, this enables to show that $X$ cannot support a Zariski dense entire curve, see \S~\ref{SS:hyperbolicity prop}.

\begin{THM}\label{TH:hyperbolicity}
	Let $(X,\F)$ be a foliated compact K\"ahler manifold ($\F$ being holomorphic).  Assume that $\F$ is transversely K\" ahler with quasi-negative Ricci curvature and that $\F\subsetneq\overline{\F}$. Then every entire curve on $X$ is contained in a proper analytic subset of $X$. 
\end{THM}
\subsection{The case where $\F$ has complex codimension one}\label{SS:codim 1}
In this situation Theorems~\ref{T:A} and~\ref{T:CdeRhamdec} read as follows.\footnote{This can be proved with simpler arguments than those used in the rest of this paper.}
\begin{enumerate}
	\item Either the leaves of $\F$ are closed and the leaf space $X/\F$ is a compact Riemann surface hyperbolic in the orbifold sense.
	\item Either $\F$ is minimal and transversely hyperbolic: the Hermitian symmetric space $\mathscr{H}$  involved in the statement of Theorem~\ref{T:CdeRhamdec} is the upper half-plane $\Hj^2$ ($\mathscr{K}$ is reduced to a point).
\end{enumerate}
Maybe the simplest instance of such minimal foliation is provided by transversely hyperbolic holomorphic foliations on canonically polarized projective surfaces $S$ which appear in Brunella's classification \cite{Brunella97}. In  \textit{loc.cit.} Brunella raised the following question: is  $S$ necessarily a quotient $\Hj^2/\Gamma$ of a bidisk by an irreducible cocompact lattice? Up to our knowledge, this is still an open problem.\footnote{Even if it is known that the monodromy representation of this transverse hyperbolic structure takes values in an arithmetic group (see Theorem~\ref{TH:codim1}).}
\medskip

It's important to note here that the latter dichotomy holds in a more general framework. Namely, $\F$ is allowed to be singular in \cite{Touzetconormal} (see also Theorem \ref{TH:codim1}) and its conormal bundle is merely assumed to be pseudoeffective.

\subsection{The case where $\F$ has real dimension one}\label{SS:dimR=1}

Let $X$ a $n$-dimensional compact manifold equipped with a 1-dimensional Riemannian foliation $\F$ transversely modelled on $(M, \Isom(M))$ where $M$ is a simply connected complete Riemannian manifold together with its group $\Isom(M)$ of isometries. Let $\rho:\pi_1(X)\to \Isom(M)$ the holonomy representation and denote by $H_0$ the identity component of the closure of the image of $\rho$. It is easy to see that the following assertions are  equivalent:
\begin{itemize}
	\item Every leaf of $\F$ is closed.
	\item The group $H_0$ is trivial.
\end{itemize}
According to a result of Thurston when $M={\mathbb H}^{n-1}$ and  generalized afterwards by Carrière \cite[Appendix~A, Theorem~1.1]{Molino-livre}, the group $H_0$ is \textit{abelian}. This implies that, if $\F$ satisfies properties~\ref{A:Riccineg} and~\ref{A:homor}, it cannot fit into this category unless $H_0$ is trivial. Indeed, suppose by contradiction that $\F$ possesses a non compact leaf $\mathcal L$. Then, according to Theorems~\ref{T:A}, \ref{L:BSpaceleaves}, and~\ref{T:CdeRhamdec}, the topological closure $\bar{\mathcal L}$ is a submanifold of $M$ and the restriction ${\F}_{\bar{\mathcal L}}$ is transversely modelled on a symmetric space $\mathscr{H}$. Moreover, the group $H_0$ attached in this case  to the representation $\pi_1(\bar{\mathcal L})\to \Isom(\mathscr{H})$ is nothing but the identity component of the target, a semi-simple Lie group.

Important examples of 1-dimensional transversely K\"ahler foliations are provided by the orbits of the so-called Reeb vector field on Sasakian manifold. These foliations are known to satisfy property~\ref{A:homor} (this is a general feature for foliations defined by a killing vector field) so that the leaves are automatically closed whenever they satisfy \ref{A:Riccineg}. On the opposite side, Epstein has given in \cite{Epstein82} examples of $1$-dimensional transversely hyperbolic foliation on closed 3-dimensional manifolds $M$ such that the leaf closure is 2-dimensional torus and the holonomy representation $\pi_1(M)\to \Isom(\mathbb H^2)$ takes values in the affine subgroup. In fact, these foliations are not transversely homologically orientable and their dynamical and arithmetic behavior present some strong similarities with that of holomorphic foliations described in item \eqref{R:homorientnecessary} of  the following subsection.

\subsection{Remarks and counterexamples}\label{SS:counterexamples}
We give now some basic examples of transversely K\"ahler foliations which exhibit different behaviors when the assumptions \ref{A:Riccineg} or \ref{A:homor} are dropped.
\begin{enumerate} 
	\item Let $\F$ be a Riemannian foliation on a compact manifold~$X$. In general the leaves of the lifted foliation on the universal cover $\wtX$ are not necessarily closed.
	\item The conclusions of Theorem \ref{T:A} do not necessarily hold if we only require the Ricci form to be only semi-negative. Actually, it may happen that the topological closure of the leaves of a linear foliation on a complex torus are real hypersurfaces. In this setting, the natural transverse metric $\overline g$ is flat and the Ricci form $\gamma$ vanishes identically. 
	\item For general transversely K\" ahler holomorphic foliations on compact K\" ahler manifolds, the dimension of the topological closure of the leaves is likely to vary. A simple instance of this phenomenon is the Riccati foliation constructed on a ruled surface $S$ over a curve $C$ of genus $g\geq 1$ by the datum of a dense representation
\[\pi_1(C)\To { \mathbb S }^ 1\subset \textrm{Aut}(\mathbb P^1) .\]
In this situation, there exists exactly two closed leaves and the closure of the other leaves are Levi-Flat hypersurfaces.The natural transverse metric is induced by the Fubiny-study metric on $\mathbb P^1$ and thus coincides with its Ricci form.  We don't know if this equidimensionality defect can occur when the Ricci form is semi-negative (but not quasi-negative).
\item\label{R:homorientnecessary}
We cannot drop the homological orientability assumption, even if $X$ is complex and $\F$ is holomorphic.  

In order to justify this assertion, let us consider the  examples of non-{K}ähler compact complex manifolds  associated to number fields as constructed in \cite{OT} generalizing some examples of Inoue surfaces \cite{Ino}. We retain the presentation given in \cite{OT} and we refer to \textit{loc.cit.} for details. This example has been previously considered in \cite[\S5.5]{lobiancoetal2022} in relation with the study of the transverse action of the automorphisms group of a foliation (see the next subsection).
 
Let $K$ be a number field, let $\sigma_1,\dots,\sigma_s$ be its real embeddings and $\sigma_{s+1},\dots,\sigma_{s+2t}$ its complex embeddings ($\sigma_{s+t+i}=\overline{\sigma_{s+i}}$).  Let us assume that $s,t>0$. Let $\Hj$ be the Poincar\'e upper half-plane. Let $a \in \mathcal{O}_K$ acting on $\Hj^s\times \C^t$ as a translation by the vector $(\sigma_1(a),\dots,\sigma_{s+t}(a))$. Let $u \in \mathcal{O}_K^{*,+}$ be a totally positive unit (\emph{i.e.} $\sigma_i(u)>0$ for all real places). Then $u$ acts on $\Hj^s\times \C^t$ by $u\cdot(z_1,\dots,z_{s+t})=(\sigma_1(u)z_1,\dots, \sigma_{s+t}(u)z_{s+t}).$ Moreover, as $t>0$, the set $\{ (\sigma_1(a),\dots,\sigma_{s}(a))\mid\ a\in \mathcal{O}_K \}$ is dense in ${ \mathbb R}^s$. For any subgroup $U$ of totally positive units, the semi-direct product $U \ltimes \mathcal{O}_K$ acts freely on $\Hj^s\times \C^t$. This subgroup $U$ is called admissible if the quotient space $X(K,U)$ is a compact complex manifold. In particular, admissible groups must have rank $s$. We can always find such admissible subgroups. 

The admissible group $U$ being given, the corresponding compact complex manifold support a transversely Hermitian symmetric foliation of complex codimension $s$  transversely locally modeled on $\Hj^s$ and which lifts to the vertical foliation on the universal cover $\Hj^s\times \C^t$. The  leaves closure are thus codimension $s$ real submanifolds of $X$, namely  $(2t+s)$-dimensional real tori fibering over the real $s$-dimensional torus ${ \mathbb T}^ s$. In any cases, these submanifolds fail to be holomorphic.

Here, the representation $\rho_\F:\pi_1 (X(K,U))\to { \PSL(2,\R)}^s$  associated to the transverse hyperbolic structure takes values in the product of affine subgroups ${ \Aff(2,\R)}^ s$ and its linear part $\rho_\F^1:\pi_1 (X(K,U))\to ({\R}_{>0}^s,\times)$ has non-trivial image.

The product of Poincar\'e metrics (defining the transverse hyperbolic metric) on  ${ \Hj}^ s$ is given (up to a multiplicative factor) as $\omega= dd^c \left(\sum_{i=1}^s\log (\Im (z_i))\right)$ and thus descends to $X$ as an exact two form, namely the differential of the $\pi_1(X)$-invariant one form $d^c\left(\sum_{i=1}^s \log (\Im (z_i))\right)$. Hence, this foliation is not homologically orientable (albeit satisfying \ref{A:Riccineg}).

\end{enumerate}

\subsection{Transverse action of the group of automorphisms}
Consider now a foliated compact complex manifold $(X,\F)$ and
let $\Aut(X,\F)$ be the group of biholomorphisms  of $X$ preserving the foliation $\F$.  This group contains the normal subgroup $\Autfix(X,\F)$ of biholomorphisms $f$ preserving the foliation leafwise. That is, $f\in \Autfix(X,\F)$ iff for every $x\in X$, ${ \mathcal L}_x= { \mathcal L}_{f(x)}$ where ${ \mathcal L}_x$ denotes the leaf through $x$.
\begin{dfn}
	We will say that the transverse action of $\Aut(X,\F)$ is finite whenever  the quotient $\Aut(X,\F)/\Autfix(X,\F)$ is finite.
\end{dfn}
In other words, the action is transersely finite if the set theoretic action $\Aut(X,\F)\times X/\F\to X/\F$ on the leaf space $X/\F$ is \textit{finite} (\emph{i.e.} has finite image in the symmetric group of $X/\F$).

Assume now that $\F$ is transversely K\" ahler and satisfies the assumption~\ref{A:Riccineg}. In the classical (unfoliated) situation, the quasi-negativity of the Ricci curvature implies that $X$ is of general type~\cite{Riemen73}. As a byproduct, the group $\Aut(X)$ of biholomorphisms of $X$ is nessessarily finite. In our foliated setting, it is therefore natural to inquire whether the transverse action is finite. Actually, we cannot expect it to hold in full generality. This is illustrated by the example given in \S\ref{SS:counterexamples}-(\ref{R:homorientnecessary}) above. Indeed, retaining the same notations, we know from Dirichlet's units theorem that $\mathcal{O}_K^*$ is a group of rank $s+t-1$. Elements of $\mathcal{O}_K^{*,+}/U$ induce automorphisms of $X(K,U)$. Therefore as soon as $t>1$, we obtain automorphisms with infinite transverse order. 

We now turn our attention to the particular case where $X$ is itself K\" ahler.  As we can see, the situation is much better.
\begin{THM}\label{T:transversefinite}
	Let $(X,\F)$ be a foliated compact K\"ahler manifold ($\F$ being holomorphic).  Assume that $\F$ is transversely K\" ahler with quasi-negative Ricci curvature.  Then the transverse action of $\Aut(X,\F)$ is finite.
\end{THM}
\begin{remark}
We are not aware of any example satisfying \ref{A:Riccineg} and \ref{A:homor} with $X$ complex non-K\"ahler such that the transverse action of $\Aut(X,\F)$ is infinite.
\end{remark}

\section{Outline of the proof}
Let us now describe the strategy that we employ to establish the main Theorems~\ref{T:A} and~\ref{T:CdeRhamdec}. As already mentioned in the introduction, the techniques are widely inspired by previous results of Nadel \cite{Nadel} and Frankel \cite{Frankel-annals} about the structure of the universal cover of a canonically polarized manifold. In particular, we make use of ``foliated twisted harmonic maps'', available in our context. 
 
From a previous result by Touzet \cite{Touzet-toulouse} (\emph{cf.} Theorem~\ref{T:Touzet-semisimple}), we know that the commuting sheaf $\mathcal C$ of the foliation (in Molino's theory terminology, see \S\ref{SS:commutingsheaf}) is semi-simple without compact factors. Roughly speaking, this commuting sheaf is a locally constant sheaf of Lie algebra (with typical fiber denoted by $\mathfrak g$) of basic Killing vector fields which encodes the dynamic of the foliation and which  somehow represent the infinitesimal part of the holonomy pseudo-group (see \S  \ref{SS:asstrfrbundle}, this is a common feature of Riemannian foliations). In our setting, we can derive from semi-simplicity of $\mathfrak g$ the existence of a representation (the \textit{monodromy representation of $\F$} attached to $\mathcal{C}$):
\[\rho:\pi_1(X)\To \mathrm{Aut}^0(\mathfrak{g})\]
with dense image (up to replacing $X$ with a finite \'etale cover). We can also derive from the semi-simplicity of $\mathfrak{g}$ that the leaves of $\wtF$ (the lifted foliation $\wtF$ on the universal cover $\tilde X$) are closed. The latter being a complete Riemannian foliation, this implies that its space of  leaves $\wtX/\wtF$ is a complete K\"ahler orbifold as described in Theorem~\ref{L:BSpaceleaves}.  The constant sheaf $\widetilde{\mathcal{C}}$ defined on $\wtX$ as the lift of $\mathcal C$ can be then identified to a Lie subalgebra of the Lie algebra of Killing fields on $\wtX/\wtF$.
  
On the other hand, it is well known that $\Aut^0(\mathfrak{g})$ is identified, via the adjoint action, to the unique center-free Lie (and algebraic) group $G$ having $\mathfrak{g}$ as Lie algebra.
  
Let $K$ be maximal compact subgroup of $G$. According to a theorem of Corlette \cite{Corlette-jdg}, there exists a unique $\rho$-equivariant harmonic map $f:\wtX\to G/K$ (depending of a course on a given Riemannian structure $g$ on $X$). It is then natural to investigate the existence of such $f$ by requiring $f$ to be constant on the leaves  of $\wtF$. Such a property will be called $\F$-\textit{invariance}. It turns out that this can be realized provided the metric $g$ is suitably chosen. Namely, $g$ is bundle-like and $\F$ is \textit{taut}: the leaves of $\F$ are minimal submanifolds with respect to $g$. Actually, the assumption ``homologically orientable'' is equivalent to the existence of this kind of metric as stated in Theorem \ref{TH:Masatautcrit}.  Basically, one way to prove the existence of an $\F$-invariant harmonic map is first to construct a smooth  $\rho$-equivariant and $\F$-invariant map $f_0$ and to deform it to a harmonic one via the usual evolution equation. By the result of Section~\ref{S:evolution harmonic map} (which we hope has an interest in its own right), the solutions $f_t$ will remain $\F$-invariant and so will be the sough harmonic map which is obtained by taking the limit when $t\to\infty$. Actually, we do not proceed exactly in this way because it seems rather delicate to construct directly $f_0$ from the original manifold $X$. To circumvent this problem, we work on the transverse orthonormal frame bundle $X^\sharp$ equipped with its natural structure of $\Orth(n)$-principal bundle over $X$.\footnote{There is a reduction of the structural group to $\mathrm{U}(n)$ when $\F$ is transversely Kähler.} According to Molino's theory, the foliation $\F$ lifts to $X^\sharp$ as a transversely parallelizable foliation $\F^\sharp$ of the same rank and the commuting sheaf of $\F^\sharp$ descends on $X$ as the commuting sheaf $\mathcal C$. In particular, the monodromy representation of $\F^\sharp$ is exactly given by $\rho$. Now, we can exploit the structure of transversely parallelizable foliations to construct on the universal covering $\widetilde{X^\sharp}$ (see~\cite[p.~162]{Molino-livre})  a $\rho$-equivariant map $F_0$ with values in $G/K$. We can then deform $F_0$  following the evolution equation 
\[\frac{\partial F_t}{\partial t}=-d_\nabla^* d(F_t)\]
preserving the $\rho$-equivariance of the map $F_t$.

Using that the representation $\rho$ has dense image and is in particular reductive, we get, taking the limit, a $\rho$-equivariant and $\widetilde{\F^\sharp}$-invariant harmonic map $F_\infty$ which is invariant by the isometric action of the structural group $\mathrm{U}(n)$ of $\widetilde{X^\sharp}$ (as a consequence of the uniqueness of the harmonic map in Corlette's Theorem). This actually implies that $F_\infty$ descends on $\wtX$ as a $\wtF$-invariant harmonic map $f_\infty$ (with respect to the original bundle-like metric $\widetilde{g}$).
  
Alternatively, this enables us to consider $f_\infty$ as a $\rho$-equivariant harmonic map
\[f_\infty:\wtX/\wtF\To G/K\]
which can be easily seen to be a surjective submersion with connected fibers. In addition we can prove the following points.
 \begin{enumerate}
 	\item The Lie algebra $\widetilde {\mathcal{C}}$ projects via $f_\infty$ to $\mathfrak{g}=\Lieisom(G/K)$ (the Lie algebra of the isometry group $\Isom(G/K)$ of the symmetric space $G/K$). 
 	\item For every fiber $F$ of $f_\infty$, for every $x\in F$, there does not exist any $V\in\widetilde{\mathcal{C}}$ such that $V(x)\not=0$ and $V(x)\in T_x F/ {\wtF}_x$.\label{MA}
 	\item The map $f_\infty$ is $H$-equivariant with respect to a subgroup  of the isometry group ${\Isom}(\wtX/\wtF)$ which integrates $\widetilde{\mathcal{C}}$.
 \end{enumerate}

If we combine this with the rigidity properties of harmonic maps as proved by Carlson--Toledo and Jost--Yau, we obtain that $\mathscr{H}:= G/K$ is a Hermitian symmetric space and that $f_\infty$ is indeed holomorphic (up to switching the complex structure of $\mathscr{H}$ to its conjugate). The proof is widely inspired from Frankel's article \cite{Frankel-annals}  and is somehow simpler, at least as far the $H$-equivariance property is concerned. On the other hand,  the item~\eqref{MA} is a byproduct of the existence of solutions of a foliated Monge--Amp\`ere equation in the spirit of what is done in \cite{Elkacimi90} and which requires a fairly more technical analysis.

Again relying on Frankel's argumentation \cite{Frankel-acta,Frankel-annals}, we can deduce the existence of a holomorphic splitting 
$$T( \wtX/\wtF)=\mathcal{G}_1\oplus \mathcal{G}_2$$
where $\G_1$ is the foliation induced by the orbits of $\tilde{\mathcal C}$ and $\G_2$ is the fibration defined by $f_\infty$. Keeping in mind the dynamical meaning of $\mathcal C$, we obtain Theorems~\ref{T:A} and~\ref{T:CdeRhamdec}.

Theorem \ref{T:transversefinite} turns to be a consequence of the latter statements together with results concerning linear representations of K\"ahler groups that we will recall subsequently (see Section~\ref{S:transversefinite}).

\subsection*{Organization of the article}

Let us describe briefly the content of this article. Section~\ref{S:recoll} gathers standard results on Riemannian foliations and the basic objects attached to them. Their structure (Molino's theory) is then presented in Section~\ref{S:structure riemannian foliation}. The (foliated) harmonic flow is studied in Section~\ref{S:evolution harmonic map} and more precisely how the tangential energy behaves with respect to the heat operator. This fundamental estimate is used to produce harmonic maps that are constant along the leaves of the foliation in Section~\ref{S:equivariant-foliated-harmoni-maps}.  The transverse Kähler case is studied in the remaining sections. In Section~\ref{S:DRdecomposition}, we show the existence of a supplementary foliation and prove Theorems~\ref{T:A}, \ref{L:BSpaceleaves} and~\ref{T:CdeRhamdec}. The proof of Theorem~\ref{T:transversefinite} is given in Section~\ref{S:transversefinite} and Theorem~\ref{TH:hyperbolicity} is given in Section~\ref{S:questions} together with additional questions and remarks. Finally Appendix~\ref{S:proofofmaintechnlemma} is devoted to the proof of a doubly foliated version of Yau's theorem

\section{Riemannian foliations: a recollection}\label{S:recoll}

Let $(X,\F)$ be a foliated manifold. One will denote repectively by $m$ and $n$ the \textit{rank} and the \textit{codimension} of $\F$.\footnote{For the time being, we deal with real foliations. However, from Section~\ref{S:DRdecomposition} for the sake of notational simplicity, we will also denote by $n$  the complex codimension of transversely K\"ahler foliations (hence of real codimension $2n$). Actually this is the notation used in the introductory part.} We follow notation from \cite[Chapter~2]{Molino-livre}: $\Chi(X)$ (resp. $\Chi(\F)$) stands for the Lie algebra of vector fields on $X$ (resp. tangent to $\F$). Let us denote by $L(X,\F)\subset \Chi(\F) $ the Lie algebra of \emph{foliated} vector fields:
$$L(X,\F)= \left\{v\in\Chi (X) \mid [\Chi (\F),v]\subset \Chi (\F)\right\}.$$
Remark that $L(X,\F)$ is a module over the ring $\Omega^0 (X/\F)$ of basic functions (\emph{i.e.} functions constant on the leaves). Let us finally consider $\ell(X,\F)$ the Lie algebra of \textit{basic} vector fields defined by the quotient
$$\ell(X,\F):=L(X,\F)/\Chi (\F).$$
It is also a module over $\Omega^0 (X/\F)$.
\begin{dfn}[\emph{cf.} \protect{\cite[Chapter 1, \S1.1 and~1.4]{Molino-livre}}]\label{D:simple}
\leavevmode
\begin{enumerate}
	\item An open subset $U\subset X$ is said to be \emph{ distinguished} (with respect to the foliation $\F$) if there exists a diffeomorphism $\phi=(x_1,\ldots,x_m,y_1,\ldots,y_n):U\to \Omega$ onto a domain $\Omega$ of $\R^{m+n}=\R^{m}\times \R^n$ such that the restriction of $\F$ to $U$ is given by $\{dy_1=\cdots=dy_n=0\}$. 	
	\item The foliation $\F$ is said to be \emph{simple} if its leaves can be defined as the fibers of some surjective submersion $f:X\to Y$. In particular the space of leaves $X/\F$ (with the quotient topology) is homomeorphic to $Y$ and thus inherits a natural structure of manifold such that the projection map $X\to X/\F$ is submersive.
\end{enumerate}
\end{dfn}
Of course, we can cover $X$ by open subsets $U$ which are both distinguished and simple (with respect to the projection $(x,y)\mapsto y$) for the restricted foliation $\restr{ \F}{U}$. In the sequel, when we will consider \textit{local} foliated/basic objects or \textit{local} space of leaves, it should be understood that this means ``in restriction to these peculiar neighborhoods''. The pair $(U,\phi)$ as above will be refered as a \textit{foliated chart}.

\subsection{Basic tensor fields/basic cohomology}
Let $N\F=TX/\F$ denote the \textit{normal} bundle to $\F$. 
The respective duals $\F^ *$, $N^*\F$ are called the \textit{ cotangent} and \textit{conormal} bundle of $\F$. The (local) sections of the latter are precisely (local) forms of degree one whose restriction to the leaves vanish identically. More generally, a transverse $(p,q)$ tensor field is a section of $\Gamma_\F^{p,q}:=N\F^{\otimes p}\otimes N^*\F^{ \otimes q}$.

We can alternatively define a basic vector field as a section of $N\F$ which is flat with respect to the partial \textit{Bott connection} 
$$\nabla^\F: N\F\to N\F\otimes \F^ *$$
defined for every local sections $X$ and $Y$ (respectively sections of $\F$ and $N\F$) by $ \nabla_X^\F Y= [X,Y]$ (well defined as a section of $N\F$ by integrability).

We can more generally define a \textit{basic} $(p,q)$ transverse tensor field as a section $s$ of $\Gamma_\F^{p,q}$ such that $\nabla^\F s=0$ (here, we denote by the same symbol the extension of the Bott connection to $\Gamma_\F^{p,q})$. In local foliated transverse coordinates $x=( x_1,\ldots,x_n)$, this amounts to saying that $s$ can be written as a sum of simple tensors of the form $$
f(x)\frac{\partial}{\partial x_{i_1}}\otimes\cdots\otimes \frac{\partial}{\partial x_{i_p}}\otimes {d x_{j_1}}\otimes\cdots\otimes {d x_{j_q}}.$$

\begin{dfn}[Riemannian foliation]\label{def:riemannian foliation}
A foliation $\F$ is said to be (transversely) \emph{Riemannian} if $M$ can be quipped with a basic transverse metric $\overline{g}$, that is a basic section $\overline{g}$ of the symmetric power $\mathrm{Sym}^2 N^*\F\subset N^*\F^{ \otimes 2} $ which is positive definite on $N\F$.
\end{dfn}

%Note that any automorphism of $X$ preserving the foliation naturally induces an automorphism $f_*$ of the vector space of basic transverse tensor fields.
%j'ai commenté la phrase qui ne me semblait pas incontournable.

\medskip
We can easily check that a $q$-form $\theta\in \Gamma(\bigwedge^q T^*X)$ is basic (then in particular is a section of $N^*\F^{ \otimes q}$) if and only if we have $i_v \theta=i_v(d\theta)=0$ for every vector field $v$ tangent to $\F$. The algebra $\Omega^\bullet (X/\F)$ of basic differential forms is then a subcomplex of $\Omega(X)$ whose cohomology $H^*(X/\F)$ is the so-called \textit{basic (de Rham) cohomology}. 

In the case where $\F$ is (a codimension $n$) \emph{Riemannian} foliation on a \emph{compact} manifold $X$, the basic cohomology turns out to be finite dimensional, according to \cite{Kacimihector86}. Moreover, the same authors prove in the same setting the following result.
\begin{thm}[\emph{cf.} \protect{\cite[Th\'eor\`eme de dualit\'e 4.10]{Kacimihector86}}] \label{th:poincareduality}
	For a Riemannian foliation $\F$ on a compact manifold $X$ we have the following alternative:
	\begin{enumerate}
		\item either $H^n(X/\F)=0,$
		\item or $H^n(X/\F)\simeq \R$, in which case $\F$ is transversely orientable,   $H^n(X/\F)$ is generated by the class of the basic volume form and $H^*(X/\F)$ satisfies the Poincar\'e--Hodge duality.
	\end{enumerate}
\end{thm}

\subsection{Bundle-like metrics}\label{SS:bundlelike}
Let $(X,\F)$ be a foliated manifold where in addition $\F$ is assumed to be Riemannian. Let us denote by $\overline{g}$ the transverse invariant metric defined on the normal bundle $N\F$ and by $\nu_\F$ the associated transverse volume form, assuming that $N\F$ is oriented.  A metric $g$ on $X$ is said to be \textit{bundle-like} (with respect to the given Riemannian transverse structure) if the metric induced on $\F^\perp\simeq N_\F$ is precisely  $\overline{g}$. This amounts to saying that $\F$ can be locally described by the fibers of a Riemannian submersion where the basis of the fibration is equipped with the metric $\overline{g}$.  It is well known (see \cite{Reinhart59}) that bundle-like metrics always exist. The triple $(X,\F,g)$ will be referred as a \textit{Riemannian foliated manifold}. The pair $(\F,g)$ stands for a Riemannian foliation together with a bundle-like metric $g$ on the ambient manifold.

\subsection{The transversely Kähler case}\label{SS:transverse Kähler}

Consider now the situation where $(X,\F)$ is a foliated manifold such that $\F$ is transversely holomorphic of complex codimension $n$. Recall that this means that we can cover $X$ by foliated charts $U_i$ of the form $f_i: U_i \to V_i\subset \R^m\times \C^{n}$ such that for every $i,j$ such that $U_i\cap U_j\not=\emptyset$, the local transformation map $f_{ij}=f_j\circ f_i^{-1}$ of  $\R^m\times \C^{n}$  takes the form
\[f_{ij}(x,z)= \left(g_{ij} (x,z), h_{ij}(z)\right)\]
where  $h_{ij}$ is \textit{holomorphic}.

As usual this is equivalent to  the datum of a basic endomorphism  $J_\F:N\F\to N\F$ satisfying $J_\F^2=-\mathrm{Id}$ (a  transverse basic complex structure) fulfilling the Newlander--Nirenberg integrability property:
\begin{center}
{\it for all local \emph{basic} sections $v,w$ of ${ N\F}^{0,1}$, $[v,w]$ is still a section of  ${ N\F}^{0,1}$.}
\end{center}
Here,  ${ N\F}^{0,1}$ is the second summand in the splitting 
\[{ N\F}\otimes \C={ N\F}^{1,0}\oplus{ N\F}^{0,1}\]
determined by the $\pm \sqrt{-1}$ eigenspaces of $J_\F$ (more generally, if $E$ is any $J_\F$-stable subbundle of $N\F$, we can analogously consider $E^{1,0}$ and $E^{0,1}$).

\begin{dfn}[Transversely Kähler foliation]\label{def:transverse Kähler}
The foliation $\F$ is said to \emph{(transversely) Kähler} if $M$ admits a basic transverse complex structure $J_\F$ together with a  basic transverse Hermitian metric $\overline{g}$ whose imaginary part is $d$-closed.
\end{dfn}

Let $f:X\to Y$ be a smooth \textit{basic} function taking values in a complex manifold $Y$. Let $J$ be the complex structure on $Y$. Then $f$ is said to be \textit{transversely}-holomorphic or merely \textit{$J_\F$-holomorphic} if $f_*\circ J_\F=J\circ f_*$. This exactly means that the local factorization $\bar f: U/\F\to Y$ of $f$ through the local leaves spaces is \textit{holomorphic}.

Similarly, we will say that an extension (see Definition~\ref{def:extension} below) $\G$ of $\F$ is $J_\F$-\textit{holomorphic} if the foliation induced by $\G$ on $U/\F$ is \textit{holomorphic}.  Alternatively, holomorphicity of $\G$ can be characterized by the following properties.

\begin{itemize}
\item The vector bundle $\G/\F\subset N\F$ is $J_\F$-stable.
\item  $ { ( \G/\F)}^{1,0}\subset {N\F}^{1,0} $ is locally spanned by basic \textit{holomorphic} vector fields (\emph{i.e. } local sections of ${N\F}^{1,0}$ that project to holomorphic vector fields on $U/\F$). 
\end{itemize}

\subsection{Riemannian extensions}\label{SS:extension}

\begin{dfn}[Extension of a foliation]\label{def:extension}
Let $(X,\F)$ be a foliated manifold. A foliation $\mathcal G$ on $X$ is said to be an \emph{extension} of $\F$ if $\F\subset \G$. 
\end{dfn}
Suppose in addition that there exists on $X$ a basic \footnote{Here and hereafter, we mean basic with respect to $\F$.} symmetric form  $\sigma\in\mathrm{Sym}^2 N^*\F$  such that the restriction of $\sigma$ to $\G/\F$ is positive. 
This implies that $\F$ is Riemannian in restriction to the leaves of $\G$ with respect to $\sigma':=\restr{\sigma}{\G/\F}$.We will say that $\sigma'$ is a \textit{$\F$-basic $\G$-leafwise metric}. For any (local) basic section $v$ of $\G/\F $, it is then meaningful to consider the $\F$-basic $\mathcal G$-leafwise divergence $\mathrm{div}_\G (v)$ well defined as a (local) basic function of $\F$. Similarly, we can consider the $\F$-basic $\mathcal G$-leafwise gradient $\nabla_\G (f)$ (with respect to $\restr{ \overline g}{\G/\F}$) of a (local) basic function as a (local) basic function of $\F$. 
We can thus consider the associated \textit{$\F$-basic $\G$-leafwise Laplacian} $\Delta_\G$ acting on local basic functions by 
$$\Delta_\G (f)=-\textrm{div}_\G(\nabla_\G (f)).$$
It is well defined on $X$ as a \textit{basic} differential operator of order $2$ (see \cite{Elkacimi90}). It will play a prominent role in our work (see \S\ref{SS:prelproofoflemmaE}).

\subsection{Characteristic and mean curvature forms}\label{SS:characteristicform}
Let $(X,\F)$ be a foliated manifold. We also assume that $\F$ is oriented. Let $g$ be a Riemannian metric on $X$. The characteristic form $\chi_\F$ is the $m$-form (where $m=\rg(\F)$) defined by the following properties:
\begin{enumerate}
	\item the restriction of $\chi_\F$ to the leaves is the volume form associated to $g_\F$ (the leafwise metric induced by $g$);
	\item for all $v\in\Gamma(\F^\perp)$, $i_v (\chi_\F)=0$.
\end{enumerate}

Let $\tau_g\in\Gamma ({ \F}^\perp)$ be the mean curvature vector of the leaves with respect to $g$. We can define the \textit{mean curvature form} $\kappa_g$ by setting
\[
\kappa_g (s)=g(\tau_g,s) \quad \mathrm{for}\ s\in\Gamma(TX) .
\]
The following fundamental result was proven by Dominguez \cite{Dominguez98}.

 \begin{thm}\label{TH:Dominguezth}
 	Let $(X,\F)$ be a foliated compact manifold. Assume that $\F$ is Riemannian. Then there exists a bundle-like metric $g$ such that $\kappa_g$ is basic (or equivalently, $\tau_g$ is foliated).
 \end{thm}

\begin{remark}
As noticed in \cite[Chapter 7]{Tondeur}, when $(X,\F,g)$ is a compact Riemannian foliated manifold,  this basic mean curvature form $\kappa_g$ is indeed closed and its cohomology class $[\kappa] :=[\kappa_g]\in H^1(X,\R)$ does not depend on the choice of the bundle-like metric $g$. Moreover, $[\kappa] $ vanishes if and only if there exists a bundle-like metric $g$ such that $\kappa_g$ vanishes identically, \emph{i.e.} the leaves are immersed \textit{minimal} submanifolds. 
\end{remark}

\begin{dfn}
We will say that a Riemannian foliation $(\F,g)$ is \emph{tense} if $\kappa_g$ is basic and $(\F,g)$ is \emph{taut} if $\kappa_g$ vanishes identically.\footnote{This slightly differs from the usual terminology, according to which $\F$ is tense/taut if there exists  bundle-like metrics $g$ fullfilling these properties, but which are not necessarily the one we are working with. For this peculiar metrics, tense/taut (in our sense) corresponds to isoparametric/minimal.} 
\end{dfn}

\begin{thm}[\emph{cf.} \cite{Masa}]\label{TH:Masatautcrit}
Let $(X,\F)$ be a compact foliated manifold such that $\F$ is transversely Riemannian. Suppose that $\F$ is transversely orientable. Then the following two properties are equivalent:
\begin{itemize}
\item[---] There exists a bundle-like metric $g$ such that $(\F,g)$ is taut.
\item[---] The foliation $\F$ is homologically orientable.	
\end{itemize}
\end{thm}

\subsection{Relationship between the characteristic and mean curvature forms}

It is given by the so-called Rummler's formula \cite[Corollary~4.26]{Tondeur}.  We assume here that $(X,\F)$ is a foliated manifold, that $\F$ is oriented and we fix a Riemannian metric $g$ on $X$. Set as before $m=\textrm{rk}(\F)$, then:
\begin{equation}\label{Rummlerformula}
d\chi_\F+\kappa_g\wedge\chi_\F\in F^2 A^{m+1}
\end{equation}
where $F^2 A^{m+1}$ denotes the space of $(m+1)$-forms $\varphi$ on $X$ such that $i_V\varphi=0$ for any $m$-multivector field tangent to $\F$ (\emph{i.e.} $V\in \Gamma \left(\bigwedge^m \F\right)$). In particular, if $\eta \in \Omega^{n-1}(X/\F)$ is a basic $(n-1)$-form where $n$ is the codimension of $\F$, then $\eta\wedge d\chi_\F=-\eta\wedge \kappa_g\wedge\chi_\F$.
  
As a straightforward but fundamental consequence of Stokes' Theorem, we have the following result that we will used repeatedly.
  
\begin{lemma}\label{L:Rummler/Stokes}
Let $(X,\F)$ be an oriented compact foliated manifold and $g$ be a Riemannian metric on $X$. Assume that $\F$ is oriented and  that $\kappa_g$ vanishes identically. If $n= \codim(\F)$ and $\eta \in \Omega^{n-1}(X/\F)$, we then have:
$$ \int_X d\eta \wedge \chi_\F =0.$$
\end{lemma}

\subsection{Some adapted orthonormal frame bundle}
Let $(X,\F,g)$ be a Riemannian foliated manifold (see \S\ref{SS:bundlelike}). Set $m+n=\dim(X)$ where $m=\rg(\F)$. Denote by $\nabla^g$ the Levi-Civita connection associated to $g$. Recall the two basic fundamental properties of $\nabla _g$, namely the torsion freeness and metric compatibility:
\begin{align}
\forall  u,v\in \Gamma(TX),\quad \nabla_u^g v- \nabla_v^g u&=[u,v] ,\label{E:Torsionfreeg}\\
\forall  u,v,w\in \Gamma(TX),\quad u\cdot\langle v,w\rangle_g&= \langle\nabla_u^g v,w\rangle_g +\langle v,\nabla_u^g w\rangle_g.\label{E:metriccompg}
\end{align}

\begin{dfn}
Let ${(e_i)}_{i=1,\ldots,m}$ be a local orthonormal frame of $\F$ near $x\in X$. The family $(e_i)$ is said to be tangentially geodesic (with respect to $\F$) at the point $x\in X$ if for all $v\in T_xX$, we have: 
\[\nabla_{v}^g e_i(x)\in { ( \F_x)}^\perp.\]	
\end{dfn}

If $(e_i)$ is tangentially geodesic, note that the Lie bracket $[e_i,e_j]$ \textit{vanishes} at the point $x$ as an immediate consequence of the torsion freeness of $\nabla^g$ and the involutivity of $\F$:
\[
\forall\,u\in \F_x,\,\,\langle[e_i,e_j],u\rangle_g=\langle\nabla^g_{e_i}e_j,u\rangle_g-\langle\nabla^g_{e_j}e_i,u\rangle_g=0.
\]

\begin{lemma}\label{L:tangentiallygeodexistence}
For every $x\in X$, there exists  a  tangentially geodesic frame of $\F$ at the point~$x$. 
\end{lemma}
\begin{proof}
	Let $\mathcal L$ be the leaf of $\F$ through $x$  and $g_\mathcal L$ be the metric induced by $g$ on $\mathcal L$. Take $( f_i)$ an  orthonormal frame with respect to $g_\mathcal L$ defined in some neighborhood $V\subset\mathcal L $ of $x$ and geodesic in the usual sense at $x$ and extend it as a orthonormal frame $(\xi_i)$ of $\F$ in some neighborhood $U\subset X$ of the ambient manifold. For all $v\in { \F_x}$, we have thus $\nabla_{v}^g \xi_i(x)\in {  \F_x}^\perp$. We want to extend this property for all $v\in T_x X$ by modifying suitably the $\xi_i$'s. To this aim, pick a basis $(v_k)$, $k=m+1,\ldots,m+n$ of ${  \F_x}^\perp$. For every $(i,j,k)\in {\llbracket 1,m\rrbracket}^2 \times \llbracket m+1,m+n\rrbracket$, there exist $a_{ijk}\in\R$, $v_{ik}\in { (T_x \F)}^\perp$ such that 
\[\nabla_{v_k}^g \xi_i(x)=\sum_{i,j}a_{ijk}\xi_j (x) +v_{ik}.\]

From \eqref{E:metriccompg}, we can infer the skew-symmetry property $a_{ijk}=-a_{jik}$ so that there certainly  exists a family $(A_{ij})$ of smooth function $A_{ij}:U\to\R$ (shrinking $U$ if necessary) fulfilling the following properties:
	\begin{itemize}
		\item $A_{ij}$ vanishes along the leaf $\mathcal L$,
		\item $v_k( A_{ij})(x)=a_{ijk}$,
	\item $A_ {ij}=-A_{ji}$.
	\end{itemize}
For $i=1,\ldots,m$, set $\epsilon_i=\xi_i-\sum_j A_{ij}\xi_j$. By construction, the family $(\epsilon_i)$ forms a local frame of $\F$ around $x$ and satisfies 
$$\nabla_{v}^g \epsilon_i(x)\in {  \F_x}^\perp,\ \textrm{for every}\ v\in T_x X.$$

Let $(e_i)$ be the local orthonormal frame of $\F$ produced by Gram--Schmidt orthonormalization process applied to the family $(\epsilon_i)$.
 The properties of the $A_{ij}$'s listed above guarantee that for $i\not=j$, $\langle \epsilon_i,\epsilon_j\rangle_g$ vanishes at $x$ at order at least $2$. This easily implies that the jets of $e_i$ and $\epsilon_i$ at the point $x$ coincide up to order $1$, thus proving that $(e_i)$ is tangentially geodesic at $x$.
\end{proof}
The following lemma will be useful in the proof of the forthcoming Proposition~\ref{P:heatoperatorexpress} to be found in \S\ref{SS:proofoftangenergysubsol}.
\begin{lemma}\label{L:bracketidgeod}
	Let ${(e_i)}_{i=1,\ldots,m+n}$
 be a local orthonormal frame of $TX$ near $x\in X$ such that ${(e_i)}_{i=1,\ldots,m}$ is tangentially geodesic at $x$ and $e_i$ is foliated for $i>m$. Then for any indices $i,l\in \llbracket 1,m\rrbracket$ and $k\in \llbracket m+1,m+n\rrbracket$, we have
\[
\langle[e_i,e_k],e_l\rangle_g (x)= \langle[e_l,e_k],e_i\rangle_g (x)
\]
\end{lemma}
\begin{proof}
This a straightforward consequence of \eqref{E:Torsionfreeg}, \eqref{E:metriccompg} and the fact that $\F$ is stable under Lie bracket.
\end{proof}

Concerning the transverse behavior of the Levi-Civita connection, we have the following statement (see \cite[Lemma~1]{Oneill66}).
\begin{lemma}\label{L:covariantfoliated}
Let $u$, $v$ be local foliated vector fields orthogonal to $\F$. Then, the orthogonal projection of $\nabla_u^g v$ on ${ \F}^\perp$ is still foliated. 
\end{lemma}

\section{Structure of Riemannian foliations}\label{S:structure riemannian foliation}

The references for this section are the books \cite{Molino-livre} and~\cite{Moerdijkbook2003}.

\subsection{The structure of transversely parallelizable}\label{SS:TPrecol}

Let $(X,\F)$ be a foliated manifold with $\dim(X)=m+n$ and $\rg( \F)=m$. Recall (\emph{cf.} Section~\ref{S:recoll}) that $\Chi(X)$, $\Chi(\F)$ and $\ell(X,\F)$ denote respectively the Lie algebras of vector fields, tangent and basic vector fields.

\begin{dfn}[\emph{cf.} \protect{\cite[Chapter 4]{Molino-livre}}]
The foliation $\F$ is said to be \emph{transversely parallellizable} or to admit \emph{a transverse parallelism} if there exists $n=\mathrm{codim}(\F)$ basic vector fields $Y_1,\ldots,Y_n$ that form a global  frame of the normal bundle of $\F$.
 \end{dfn}
 
Remark that such foliations can be equipped with a holonomy invariant metric $\overline{g}$. Indeed, it suffices to consider the metric on $N\F$ with respect to which $Y_1,\ldots,Y_n$ is an orthonormal frame. Thus, transversely parallellizable foliations turn out to be particular cases of Riemannian foliations. We shall see in the sequel that in some sense the structure of Riemannian foliations can be studied through the transversely parallelizable case (see \S\ref{SS:asstrfrbundle}).
 
\begin{dfn}
A transversely parallelizable foliation is said to be \emph{complete} (TC for short) if it possesses a transverse frame $Y_1,\ldots,Y_n$ such that each $Y_i$ can be represented by a complete vector field $X_i\in\Chi (X)$.\footnote{The definition is a little bit more restrictive than the one given in \cite[\S4.5]{Molino-livre} but sufficient for our purposes.} Note that this automatically holds whenever $X$ is compact.
\end{dfn}

\begin{remark}
The lift of a TC foliation to any cover is still TC.
\end{remark} 

The structure of TC foliations is given by the following theorem by Molino.
\begin{thm}[Molino \protect{\cite[Theorem~4.2]{Molino-livre}}]\label{th:structure TC}
Let $(X,\F)$ be a foliated manifold such that $\F$ is TC. The topological closure of the leaves of $\F$ are the fibers of a  locally trivial fibration $\pi_\F: X\to W$,  the so-called \emph{basic fibration}. The restriction of $\F$ to any fiber $X_w:=\pi_\F^{-1}(w)$ is such that $\ell(X_w,\restr{\F}{X_w})$ is finite dimensional with
\[\dim\left(\ell(X_w,\restr{\F}{X_w})\right)=n-\dim(W).
\]
The basis $W$ of the fibration $\pi_\F$ is called the \textit{basic manifold}. 
\end{thm}

\begin{remark}
In particular, the closedness of a single leaf implies the closedness of the others and in that case the foliation $\F$ is \textit{simple} (\emph{cf.} Definition~\ref{D:simple}). 
\end{remark}

\begin{dfn}\label{def:structural Lie algebra}
With the notation of Theorem~\ref{th:structure TC}, the Lie algebra $\frg{\F}:=\ell(X_w,\F_{X_w})$ is independent of $w\in W$ (up to isomorphism) and is called the \emph{structural Lie algebra} of the TC foliation~$\F$.
\end{dfn}

\subsection{The commuting sheaf}\label{SS:commutingsheaf} 

We refer to \cite[\S4.4 and~4.5]{Molino-livre} for details about this paragraph. Let $\F$ be a TC foliation on $X$. Let $U$ be an open set of $M$ and denote by $\mathcal C (U,\F)$ the Lie algebra formed by the basic vector fields of $\F_{|U}$ that commute with each element of $\ell(X,\F)$.  The collection of $\mathcal C (U,\F)$ defines a presheaf. We will denote by $\mathcal C_\F$ the corresponding sheaf. The main properties of $\mathcal C_\F$ are summarized in the following proposition.
 
\begin{prop}[\emph{cf.} \protect{\cite[Proposition~4.4]{Molino-livre}}]
The sheaf $\mathcal C_\F$ is a locally constant sheaf of Lie algebra with typical fiber $\frg{X,\F}$. Moreover, any local section $v$ of $\mathcal C_\F$ is tangent to the basic fibration.
\end{prop}

The sheaf $\mathcal{C}_\F$ being locally constant, it gives rise to a monodromy representation:
\[\rho_\F\colon\pi_1(X)\To\Aut(\frg{X,\F}).\]
Let $G_\mathfrak{g}$ be the connected and simply connected Lie group integrating $\mathfrak{g}:=\frg{X,\F}$. Theorem~\ref{th:structure TC} can be made even more precise by saying that the foliation $\restr{\F}{X_w}$ is of \emph{Lie type} (see \cite{Fedida-cras} or \cite[\S4.2]{Molino-livre}), where $X_w$ is a fiber of the basic fibration. In particular we have a commutative diagram
\begin{equation}\label{diag:monodromy}
\begin{tikzcd}
\pi_1(X_w)\ar[r,"\rho_{\F|_{X_w}}"] \ar[d, "i_*"'] & G_\mathfrak{g}\ar[d,"\mathrm{Ad}_{G_\mathfrak{g}}"] \\
\pi_1(X)\ar[r,"\rho_\F"] & \Aut(\mathfrak{g}).
\end{tikzcd}
\end{equation}

\begin{prop}\label{prop:dense image}
The top horizontal arrow of the diagram~\eqref{diag:monodromy} has dense image. In particular, if $\mathfrak{g}$ is semi-simple, the group $\Im(\rho_\F)\cap\Aut^0(\mathfrak{g})$ is dense in $\Aut^0(\mathfrak{g})$.
\end{prop}

\begin{proof}
The proof of the first part can be found in~\cite{Fedida-cras} and the last part is a consequence of classical results about semi-simple Lie algebras (\emph{cf.}~\cite[Chapter~II]{Helgasonbook}).
\end{proof}

\begin{dfn}\label{D:central-cover}
Let $\rho_\F:\pi_1(X)\to \Aut(\frg{X,\F})$ be the monodromy representation attached to the locally constant sheaf  $\mathcal C_\F$. The covering space $\wtX_{\rho_\F}$ corresponding to $\mathrm{Ker}(\rho_\F)$ is called the \emph{central cover} of $X$. 
\end{dfn}

\subsection{A developability criterion}\label{SS:developability}

\begin{dfn}
A foliation $\F$ on $X$ is said to be \emph{developable} if its lift $\wtF$ to the universal cover $\wtX$ is simple.
\end{dfn}

It is established in \cite[Proposition~4.6]{Molino-livre} that if $p\colon\wtX_{\rho_\F}\to X$ is the central cover of $X$ (\emph{cf. ~supra}), we then have
\[\frg{\wtX_{\rho_\F},p^*\F}\subset Z\left(\frg{X,\F}\right)\]
(with the obvious identification) and in particular we get the following criterion.

\begin{lemma}\label{L:devcentercriterion}
If $\F$ is a TC foliation with centerless structural Lie algebra $\frg{X,\F}$, its lift $\F$ on the central cover is a simple foliation. In particular, if $\F$ is a TC foliation whose structural Lie algebra $\frg{X,\F}$ is semi-simple, then $\F$ is developable.
\end{lemma}
\begin{proof}
This is \cite[Proposition~4.6]{Molino-livre}.
\end{proof}

\subsection{The transverse frame bundle of a Riemannian foliation}\label{SS:asstrfrbundle}
 
As it was alluded to above, the structure of Riemannian foliations can be understood from the viewpoint of parallelizable ones.
 
Let $(X,\F,g)$ be a compact Riemannian foliated manifold (see \S\ref{SS:bundlelike} for the definition). As usual, we set $m=\rg(\F)$ and $m+n=\dim(X)$. Denote by $\F^\sharp$ the foliation constructed as the lift of $\F$ on the direct orthonormal transverse frame bundle $X^\sharp$ (see \cite[\S2.5]{Molino-livre}).  Both foliations have the same rank, $\F^\sharp$ projects onto $\F$ via the natural projection map 
\[p^\sharp:X^\sharp\longrightarrow X\] so that the differential $p^\sharp_*$ induces a surjective morphism $\overline{p^\sharp_*}$ between the normal bundles $N{\F^\sharp}$ and $(p^\sharp)^*N\F$.

The space $X^\sharp$ is naturally endowed with a structure of $\SO(n)$-principal bundle. As in the classical setting, the transverse Levi-Civita connection associated to $\F$ is defined on $X^\sharp$ by a horizontal distribution $\mathcal H$ on $X^\sharp$. By construction, the foliation $\F^\sharp$ is tangent to $\mathcal H$, so that $\overline { \mathcal H}:=\mathcal H/\F^\sharp$ is a subbundle of $N\F^\sharp$,  and both $\mathcal H$ and $\F^\sharp$ are invariant under the right action of  $\SO(n)$. Moreover, the foliated manifold $( X^\sharp,\F^\sharp)$  is equipped with a canonical transverse  paralellism. More precisely let us first fix a basis $(\lambda_1,\ldots,\lambda_k)$ of $\mathfrak{so}(n)$ (with $k=\frac{n(n-1)}{2}$) identified with the Lie algebra of fundamental vector fields (with respect to the action of $\SO(n)$). On the other hand, pick a basis $(e_1,\ldots,e_n)$ of $\R^n$ and  denote by $\overline{ u_i}\in\Gamma( \overline { \mathcal H} )$ the horizontal transverse vector field on $X^\sharp$ such that for every $z\in X^\sharp$, the projection $\overline{ p_*} (\overline{ u_i}(z))$ has coordinate vector $e_i$ in the transverse frame $z$. Let $\overline{\lambda_i}$ be the projection of $\lambda_i$ on $N\F^\sharp$. It turns out that the transverse vector fields $\overline{ \lambda_1},\ldots,\overline{ \lambda_k}, \overline{ u_1},\ldots,\overline{ u_n}$ are \textit{basic} and thus define a transverse parallelism for the lifted foliation $\F^\sharp$ (\emph{cf.} \cite[\S3.3]{Molino-livre}). We can thus apply Theorem~\ref{th:structure TC} to get a quite precise description of the foliation $\F^\sharp$.

\begin{dfn}
The structural Lie algebra $\frg{X,\F}$ of $\F$ is by definition the structural Lie algebra $\frg{X^\sharp,\F^\sharp}$ of the transversely parallelizable foliation $\F^\sharp$.
\end{dfn}

The transverse invariant metric $\overline{g}$ on $N\F$ induces a canonical transverse invariant metric ${ \overline{g}}^\sharp=(p^\sharp)^* \overline{g}\oplus_{ \overline { \mathcal H} } \vartheta$ on $N\F^{\sharp}$. The latter is obtained as the orthogonal sum (with respect to the splitting $N\F^\sharp= \overline { \mathcal H}  \oplus \textrm{Ker}(dp^\sharp)$) of the lifting of $\overline{g}$ on $\mathcal H/\F^\sharp $ and the metric $\vartheta$ on vertical fibers induced by the unique bi-invariant metric of volume $1$ on $\SO(n)$. Over the local Riemannian leaf space, $( U/\F,\overline{g})$, $((p^\sharp)^{-1}(U)/\F^\sharp,{ \overline{g}}^\sharp)$ is nothing but the orthonormal frame bundle equipped with its canonical metric. In particular the vertical fibers are totally geodesic (see for instance \cite[\S5, p.~466]{Oneill66}).

The bundle-like metric $g$ induces a canonical $\SO(n)$-invariant bundle-like metric $g^\sharp=(p^\sharp)^* {g}\oplus_{  { \mathcal H} } \vartheta $ for $\F^\sharp$ on $X^\sharp$. Indeed, $g^\sharp$ is constructed as the orthogonal sum of the lifting of $g$ on the horizontal distribution $T_\mathcal H$ and $\vartheta$. In particular $p^\sharp:(X^\sharp, g^\sharp)\to (X, g)$ is a Riemannian submersion with totally geodesic fibers.  According to \cite[Proof of Lemma~7]{Nozawarigid2010}, the mean curvature forms (see \S\ref{SS:characteristicform}) of $(X,\F,g)$ and $(X^\sharp,\F^\sharp,g^\sharp)$ are simply related by
\begin{equation}\label{E:basicMCrelated}
\kappa_{g^\sharp}= (p^\sharp)^*(\kappa_g).
\end{equation}
In particular, if $(\F,g)$ is \textit{ tense} (resp. \textit{ taut}), then $(\F^\sharp,g^\sharp)$ is \textit{ tense} (resp. \textit{ taut}).

Let $\mathfrak{g}:=\frg{X^\sharp,\F^\sharp}$ and $\cC_{\F^\sharp}$ be respectively  the structural Lie algebra and the commuting sheaf attached to $\F^\sharp$. Note that $\cC_{\F^\sharp}$ is invariant under the right action of $\SO(n)$, so that we can define $\cC_\F:=(p^\sharp)_* \cC_{\F^\sharp}$, the \textit{ commuting sheaf} associated to $\F$. According to \cite[\S5.3]{Molino-livre}, $\cC_\F$ is a locally constant sheaf 
of Lie algebras with typical fiber $\frg{X^\sharp,\F^\sharp}$ formed by local basic Killing vector fields whose local flows describe the leaves closure.\footnote{In particular, if $\F$ was already transversely parallelizable, we recover the former definitions of structural Lie algebra and commuting sheaf.} This commuting sheaf  can be also alternatively defined as the Lie algebra of the closure of the holonomy pseudo-group which turns out to be a Lie pseudo-group \cite[Appendix~D by \'E.~Salem]{Molino-livre}. 

The locally constant sheaf $\cC_\F$ gives rise to a representation of the fundamental group
\begin{equation}\label{E:repautg}
\rho_\F: \pi_1(X)\longrightarrow \Aut(\mathfrak g)
\end{equation}
such that $\rho_{\F^\sharp}=\rho_\F\circ (p^\sharp)_*$ where the representation 
\[\rho_{\F^\sharp}\colon \pi_1(X^\sharp)\longrightarrow \mathrm{Aut}(\mathfrak g)\]
is the one associated to $\cC_{ \F^ \sharp}$.

\subsection{The semi-simple without compact factors case}\label{SS:SScase}

We explain in this paragraph how the various structure results for TC foliations (Theorem~\ref{th:structure TC} and Lemma~\ref{L:devcentercriterion}) can be used when the structural Lie algebra is semi-simple without compact factors. In this setting, we can produce a basic map on the universal cover of the ambient manifold that is equivariant with respect to the monodromy representation and that takes values in a symmetric space.

Let $\F$ be a TC foliation on a manifold $X$ and let us assume that the structural Lie algebra $\mathfrak{g}:=\frg{X,\F}$ is semi-simple without compact factors.  From Lemma~\ref{L:devcentercriterion}, we know that $\F$ is developable and the leaves of the foliation $\wtF$ on $\wtX$ are thus given by the fibers of a submersion with connected fibers
\[\pi_{\wtF}\colon \wtX\To P:=\wtX/\wtF.\]
We have a commutative diagram of basic fibrations:
\[\begin{tikzcd}
\wtX \ar[r,"\pi_{\wtF}"]\ar[d,"r"'] & P\ar[d,"\alpha"]\\
X\ar[r,"\pi_\F"] & W.
\end{tikzcd}
\]

The map $\alpha$ is a submersion and it has the structure of a principal bundle over $W$. Let us now explain where this structure comes from. The Lie algebra $\ell(X,\F)$ is naturally a Lie subalgebra of $\Chi(W)$ and $\ell(\wtX,\wtF)$ can be identified (through $\pi_{\wtF}$) with $\Chi(P)$. Using $r^*$ we can thus consider $\ell(X,\F)$ as a subalgebra of $\Chi(P)$.

Let $G$ denote the group of vertical diffeomorphisms of $\alpha\colon P\to W$ (\emph{i.e.} acting in the fibers of $\alpha$) that fix each element of $\ell(X,\F)$ (seen as vector fields on $P$).

\begin{remark}
The group $G$ is a Lie group that can have infinitely many connected components.
\end{remark}

The manifold $P$ is then endowed with an action of $G$ and it is not hard to see that it makes $\alpha\colon G\to W$ a principal $G$-bundle.

By its very definition, it is straightforward to check that $\Lie(G)\simeq \mathfrak{g}$. In particular, its adjoint representation gives rise to a morphism
\[\mathrm{Ad}_G\colon G\To G_\mathfrak{g}:=\Aut(\mathfrak{g}).\]
The latter group is a real algebraic group and has thus finitely many connected components. We can consider $K\subset G_\mathfrak{g}$ a maximal compact subgroup: the homogeneous space $\mathcal{S}_\mathfrak{g}=G_\mathfrak{g}/K$ is thus a Riemannian symmetric space of the non compact type (\emph{cf.} for instance \cite[Chapter~VI]{Helgasonbook} and \cite[Chapter IV]{BorelBook98} ). This is a contractible space (diffeomorphic to a Euclidean space) and if $H$ denotes the subgroup $H:=\mathrm{Ad}_G^{-1}(K)$, the space $G/H$ is diffeomorphic to $\mathcal{S}_\mathfrak{g}$ \cite [Chapter IV, Proposition 4.10, Chapter VII, Theorem 3.7]{BorelBook98}.

It is now a classical fact in the theory of fiber bundles: the homogenous space $G/H$ being contractible, we can reduce the structure group of $P$ to $H$. Equivalently, it means that we can find
\[\Phi\colon P\To G/H\]
such that
\begin{equation} \label{E:equivariance}
\forall\, x\in P,\ \Phi(x\cdot g)= g^{-1} \Phi(x).
\end{equation}

We gather here some useful observations.

\begin{enumerate}
\item The group $\pi_1(X)$ of deck transformations of $\wtX$ acts vertically on $P$ with respect to $\alpha$ and preserves $\pi^*\ell(X,\F)$ pointwise, thus defining a morphism
\[\rho_\F: \pi_1(X)\To G .\]
This morphism maps $\pi_1(X)$ onto a \textit{dense} subgroup of the Lie group $G$. To see this, pick a point $w\in W$ and let $F_w^0$ be  a connected component of the fiber $\alpha^{-1} (w)$. Let $G^0$ be the component of the identity in $G$. As $\pi_1(X)$ acts transitively on the set of connected component of $\alpha^{-1} (w)$, it is sufficient to prove that $\pi_1(X)\cap G^0$ acts densely on $F_w^0$. But this last point immediately results from the fact that $\F$ is minimal (any leaf is dense) in restriction to the fibers of the basic fibration $\pi_\F:X\to W$, according to Proposition~\ref{prop:dense image}.
\item Let $\Phi:P\to G/H$ be the map constructed above (it depends on the choice of a section of the map $P\times_G G/H\to W$). Let us consider
\[\Psi=\Phi\circ \pi_{\wtF}\colon \wtX\To G/H.\]
By construction, $\Psi$ is an equivariant locally trivial fibration where equivariant means
\begin{equation} \label{E: Equiv2} 
\forall\,\gamma\in \pi_1(X),\ \forall\,x\in\wtX,\ \Psi (\gamma(x))= \rho_\F(\gamma)\cdot\Psi(x).
\end{equation}
\item Conversely any $\Psi\in\mathcal C^\infty (\wtX, G/H)$ constant 
 on the leaves of $\wtF$ and satisfying the equivariance property \eqref{E: Equiv2} can be obtained in this way. Indeed, let $\Phi\in \mathcal C^\infty (P,G/H)$ be the map factorizing $\Psi$. Then $\Phi$ satisfies the equivariance property \eqref{E:equivariance} for a dense subgroup of $G$ (namely $\pi_1(X)$), hence for all $g\in G$. It is then straightforward to see that $\Phi$ is given by a reduction from $G$ to $H$.
\end{enumerate}

Let us sum up the discussion above with the following statement.

\begin{thm}\label{TH:equivariantsym}
Let $(X,\F)$ be a foliated manifold such that $\F$ is TC and such that the structural Lie algebra $\mathfrak g:=\frg{X,\F}$ is semi-simple without compact factors. Let $\mathcal S_\mathfrak{g}$ be the associate Riemannian symmetric space. The following holds true.
\begin{enumerate*}
\item The foliation $\F$ is developable and the monodromy representation of $\mathcal{C}_\F$
\[\rho_\F\colon \pi_1(X)\To G\]
take values into a Lie group such that $\Lie(G)=\mathfrak{g}$.

\item \label{I:density} The closure of the image of $\rho:=\mathrm{Ad}_G\circ\rho_\F$  contains $\Isom^0(\mathcal S_\mathfrak{g})$.

\item There exists on $\wtX$ a basic (with respect to $\wtF$ ) smooth map $\Psi:\wtX\to \mathcal S_\mathfrak{g}$ such that
\begin{equation}\label{E:equivariancetrsymmetric}
\forall\,x\in\wtX,\, \forall\,\gamma\in\pi_1(x),\ \Psi (\gamma(x))= \rho(\gamma) (\Psi (x)).
\end{equation}
\item  \label{I:eqimplieslocaltrivial} Any smooth basic map $\Psi:\wtX\to \mathcal S$ satisfying the equivariance condition \eqref{E:equivariancetrsymmetric} is a locally trivial fibration  with connected fibers. Moreover such a map $\Psi$ decomposes as $\Phi\circ \pi_{\wtF}$ with $\Phi\colon P\to\mathcal S_\mathfrak{g}$ and we have
\[
\Phi\left(g\cdot(\pi_{\wtF}(x))\right)=g\cdot\Psi (x)
\]
for every $g\in G$ and for all $x\in\wtX$.
\end{enumerate*}
\end{thm}
\begin{remark} \label{R:defofG}
Under the assumptions/notation of the previous theorem, the foliation defined by the submersion $\Psi$ descends to $X$ as a minimal foliation (all leaves are dense) $\G$ whose codimension is the dimension of $\mathcal S$; this foliation $\G$ is an extension of $\F$ (in the sense of \S\ref{SS:extension}) and is transversely Riemannian homogeneous.\footnote{We refer to \cite[Chapter~III, \S3]{Godbillonbook91} for the general definition of a transversely homogeneous foliation.}
\end{remark}

We will also make use of the following observation relating the equivariant maps and the transverse frame bundle (see Subsection~\ref{SS:asstrfrbundle}).
This is a  consequence of the definition of $\cC_\F$ in term of $\cC_{\F^\sharp}$, the identification between $\Aut(\mathfrak g)$ and $\Isom(\mathcal S_\mathfrak{g})$, and the fact that $\rho_{\F^\sharp}$ factors through $\rho_\F$.
\begin{lemma}\label{L:equivariantlift}
Suppose that the structural Lie algebra $\mathfrak g$ is semi-simple without compact factors and let $\mathcal S_\mathfrak{g}$ be the associated symmetric space. Assume that there exists a $\wtF$-invariant smooth $\rho_\F$-equivariant  map $\Psi: \wtX \to \mathcal S$.  Then the composed map $\Psi^\sharp:=\Psi\circ q: \widetilde{X^\sharp}\to \mathcal S_\mathfrak{g} $ satisfies the equivariance property \eqref{E:equivariancetrsymmetric} of Theorem~\ref{TH:equivariantsym}:
\[\forall\,x\in \widetilde{X^\sharp},\forall\, \gamma \in \pi_1(X^\sharp),\
\Psi^\sharp  (\gamma(x))= \rho_{\F^\sharp}(\gamma) (\Psi (x)).\]
\end{lemma} 

\section{Evolution equation for harmonic maps}\label{S:evolution harmonic map}

The main goal of this section is to prove that the hamonic flow used to produce harmonic maps preserves the property of being basic (with respect to a given Riemannian foliation). To do so, we introduce the notion of tangential energy (\S\ref{SS:tangenergy} below) and study its behavior along the harmonic flow.

\subsection{Tangential energy}\label{SS:tangenergy}

Let $(X,\F,g)$ be a foliated Riemannian manifold and $f:X\to Y$ be a smooth function with values in a Riemannian manifold $(Y,h)$. Let us denote by $f_*^T$ the restriction of the differential $f_*$ to $\F$ and let us define the \textit{tangential energy density} (with respect to $\F$) as the function
$$e_T(f):\left\{\begin{array}{ccl}
X  &\To   & \R^+ \\
x	&\longmapsto  &\frac{1}{2} { \|{f}_*^T\|}_2^2=\frac{1}{2}\mathrm{Trace}\left(  { ( f_*^T)}^\star{ f_*^T}\right)=  \frac{1}{2} \sum_i {{ \vert f_* (e_i)\vert}_h^2}
\end{array}\right.$$
where $(e_i)$ is any orthonormal basis of $\F_x$. Here, the star exponent stands for the adjoint.

\subsection{Tension field and basic maps}\label{SS:tensionfbasicmaps}

We maintain notation/assumptions form \S\ref{SS:tangenergy}. In the sequel, we will denote indifferently the differential of $f$ by $df$ or$f_*$.  Recall that the tension field is the section of $f^{*} ( TY)$ defined by 
$$\tau(f)= \mathrm{Trace}\left(\nabla(df)\right)=\mathrm{div}(df)=-d_\nabla^* df.$$
Here, $\nabla$ is the pull-back by $f$ of the Levi-Civita connection $\nabla^h$ on $Y$ and $d_\nabla^*$ is the adjoint operator of the differentiation
$$d_\nabla: A^k(X, T^*X\otimes f^{*}(TY))\To A^{ k+1}(X, T^*X\otimes f^{*}(TY))$$
of forms  on $ X$ valued in ${f}^{*}(T {  Y})$.

In more layman terms, once we have fixed a local orthonormal frame $(e_i)$ of $TX$,
\begin{equation}\label{E:tensionfieldlocalexp} 
\tau(f)= \sum_i \nabla_{ {e_i}}f_*{e_i} - f_*(\sum_i { \nabla}_{ {e_i}}^{ { g}}  { e_i}).
\end{equation}

Let us now focus on the case  where $f$ is \textbf{basic} (\emph{i.e.} leafwise constant). Consider the curvature form of the connection $\nabla$:
$$\mathscr R (v,w)= \nabla_v \nabla_w-\nabla_w\nabla_v-\nabla_{[v,w]}$$ where $v,w\in \Gamma(TX)$. If $v\in\Gamma (\F)$, $f_*v=0$ and from the definition of $\nabla$, we can infer that 
\begin{equation}\label{E:flatFcurvature}
i_v  \mathscr R =0
\end{equation}
for every $v\in\Gamma(\F)$. Let us fully justify this vanishing property. Pick $x\in X$ and any local section $s$ of $TY$ defined near $f(x)$. Set $s_f=s\circ f$. This is well defined as a local section of $f^{*}(TY)$ near $x$. By multilinearity of the curvature tensor, it is then sufficient to check that $\mathscr R (v,w)(s_f)=0$ for any local foliated vector field $w$. This last point is just a consequence of the fact that both $s_f$ and $\nabla_w (s_f)$ are constant along the leaves and that $[v,w]$ is tangent to $\F$, the vector field $w$ being foliated. This yields
\[\nabla_v (s_f)= \nabla_v\nabla_w (s_f)=\nabla_{[v,w]}(s_f)=0.\] 

In the terminology of \cite[Definition~2.33]{Kambertondeur75} (see also \cite[\S2.6]{Molino-livre}) $f^{*}(TY)$ is a \textit{foliated vector bundle} and  $\nabla$ is a \textit{basic connection} on it. Equivalently, the pair $f^{*} (TY)$ belongs to the category of $\F$-vector bundles as defined in \cite[Section~2.2]{Elkacimi90}. In our context,  a section $s$ of $f^{*}(TY)$ is said to be \textit{basic} if $\nabla_v s=0$ for every $v\in \Gamma(\F)$. Local basic sections form a free module over the ring of local basic functions whose rank coincides with that of 
 $T_Y$.

More generally, denote by $A_b^k\subset A^k:=A^k(X, T^*X\otimes f^{*}(TY))$ the subspace of twisted basic forms of degree $k$, that is
$\alpha\in A_b^k$ iff for every $v\in\Gamma (\F)$, we have $i_v(\alpha)=0$ and $i_v ( d_\nabla \alpha)=0$. More explicitly, $\alpha\in A^k$ is basic if and only if it can be written locally as a finite sum of simple tensors of the form $s\otimes \xi$ where $s$ and $\xi$ are respectively (local) basic sections of $f^{*}TY$ and $\wedge^k N^*\F\subset\wedge^k T^*X$. Thanks to~\eqref{E:flatFcurvature}, the differential $d_\nabla$ induces a differential on the graded algebra $A_b^\bullet$.

Remark that for every local foliated vector field $v$, $f_* v$ is constant along the leaves, so that $\nabla_w f_*v=0$ for every $w$ tangent to $\F$. In particular the tension field of $f$ takes the simplified expression:
\begin{equation}\label{E:tensionfieldbasic}
\tau(f)= \sum_i \nabla_{ {e_i}}f_*{e_i} - f_*\left(\sum_i { \nabla}_{ {e_i}}^{ { g}}  { e_i}\right)-f_* (\tau_g).
\end{equation}
Here, $(e_i)$ is any local orthonormal frame of ${ \F^\perp}$. 
Suppose in addition that $\F$ is \textit{tense} with respect to $g$ (\emph{i.e.} $\tau_g$ is foliated). By choosing the $e_i$'s to be foliated and thanks to Lemma~\ref{L:covariantfoliated} and \eqref{E:flatFcurvature},  we immediately check that $\tau(f)$ is \textit{basic} in the previous sense,  that is
\[\nabla_w \tau(f)=0\]
for every $w\in\Gamma(\F)$.

As noticed before we can restrict the operator $d_\nabla$ to the graded algebra $A_b$ of basic forms. Moreover, when $\F$ is transversely oriented, the basic star operator ${\overline \star}$ defined at the level of local basic forms (\emph{cf.} \cite[Chapter~7]{Tondeur}) extends\emph{via} the metric structure to $f^{*} ( TY)$ as an operator that we denote by the same symbol 
${\overline\star}:A_b^\bullet\to A_b^{n- \bullet}$
 so that we can consider the adjoint $d_\nabla^{\overline\star}:A_b^\bullet \to A_b^{ \bullet -1}.$ As $f\in A_b^0$, we can thus define its \textit{basic tension field}
\[ \tau_b (f)=-d_\nabla^{\overline\star} d f.\]
 More explicitly, let us consider $U$ a sufficiently small distinguished simple open set of $X$ and $\overline f:U/\F\to Y$ the map that factorizes $f$ through the projection $\pi:U\to U/\F$. Note that $\pi$ induces a bundle map $\pi_*:f^{*} ( TY)\to { \overline f}^{*} ( TY)$. The (local) space of leaves $U/\F$ is equipped with the transverse metric $\overline g$ and we have 
 \[\pi_*\tau_b (f)=\tau(\overline{f})\circ\pi .\]
If $(e_i)$ is a local orthonormal frame of ${ \F^\perp}$ and $(\overline{e_i})$ the corresponding orthononormal frame with respect to $\overline{g}$ (regarded as vector fields on $U/\F$), we have
 
\[ \tau(\overline{f})= \sum_i \overline{\nabla}_{\overline{e_i}}\overline{f}_*{\overline{e_i}} - \overline{f}_*\left(\sum_i { \overline{\nabla}}_{\overline{e_i}}^{ {\overline{g}}}  { \overline{e_i}}\right)\]
where $ \overline{\nabla}=\overline{f}^* (\nabla^h)$.
In view of \eqref{E:tensionfieldbasic} and in the particular case where $\tau_g$ vanishes identically,  this somehow means that the notion of harmonicity (with respect to $g$) and transverse harmonicity (with respect to $\overline{g})$ coincide.

We gather the previous observations in the following.

\begin{prop}\label{P:relationtensionfieldbasictf}
Let $(X,\F,g)$ be a Riemannian foliated manifold and $f:X\to Y$ be a leafwise constant smooth map to a Riemannian manifold $(Y,h)$. Then the ordinary tension field $\tau (f)$ and the basic tension field $\tau_b(f)$ are related by 
	\[\tau(f)= \tau_b(f)-f_* (\tau_g).\]
In particular $\tau(f)$ is basic if $(\F,g)$ is tense and $\tau (f)=\tau_b (f)$ whenever $(\F,g)$ is taut.
\end{prop}

\begin{remark}\label{R:basiclaplacian}
A special occurence of Proposition~\ref{P:relationtensionfieldbasictf} is given in \cite[Chapter~7]{Tondeur}, when $Y=\R$ equipped with the standard euclidean metric. In this context, $-\tau$ and $-\tau_b$ are nothing but the usual and basic Laplacian.
\end{remark}

\subsection{The equivariant setting}

Let $(X,\F)$ be a $(m+n)$-dimensional compact foliated manifold  where $\F$ is Riemannian and has rank~$m$.  Let us equip $\F$ with a bundle-like metric $g$ such that the mean curvature vector field of the  the leaves is \textit{foliated}. This can be always achieved thanks to Theorem~\ref{TH:Dominguezth}.

Let $(Y,h) $ be a Riemannian manifold together with a representation
\[\rho:\pi_1(X)\To \Isom(Y).\] 
of the fundamental group of $X$, seen as the group of deck transformations of the universal cover $\wtX$, into the isometry group of $(Y,h)$.

Let $f_0:\wtX\to Y$ be a smooth $\rho$-equivariant mapping and let us also consider a smooth variation $(f_t)_{t\in I}$ of $f_0$ (with $I=[0,\,t_0]$), that is a mapping $F\in{ \mathcal C }^\infty (I\times \wtX),\ f_t:=F(t,\cdot)$ such that 
\begin{enumerate}
	\item $f_t$ is $\rho$-equivariant,
	\item The family $(f_t)$ satisfies the evolution equation
		\begin{equation}\label{E:evolution}
		\frac{\partial f_t }{\partial t}=\tau(f_t)
	 \end{equation}
\end{enumerate}
where $\tau(f_t)\in\Gamma (f^{*} TY)$ is the \textit{tension field} of $f_t$ (see \S\ref{SS:tensionfbasicmaps}).

Finally, let us denote by $e_T: \wtX\times I\to \R_{\geq 0}$ the function $(x,t)\mapsto e_T(f_t)(x)$ where the tangential energy density (introduced in \S\ref{SS:tangenergy}) is taken with respect to the foliation and the complete metric on $\wtX$ obtained by respective pull-backs of $\F$ and $g$. Obviously, $e_T$ descends on $X\times I$ as a function denoted in the same way. 

\begin{thm}\label{Th:tangentialenergy}\label{TH:Tangentialbehavior}
The tangential energy density $e_T$ is a subsolution of the heat operator, \emph{i.e.} there exists a positive constant $C$ such that 
\[\left( \dfrac{\partial}{\partial t}+ { \Delta}^g\right) (e_T)\leq Ce_T.\]
\end{thm} 

Let us introduce the following terminology.
\begin{dfn}
A smooth $\rho$-equivariant mapping	$f:\wtX\to Y$ is said to be $\F$-invariant whenever $f$ is basic, \emph{i.e.} leafwise constant (with respect to the lifted foliation $\wtF$).
\end{dfn}

\begin{cor}\label{C:remainsleafwiseconstant}
If the initial datum $f_0$ is $\F$-invariant, then the maps $f_t$ are $\F$-invariant as well for all time $t\in I$.
\end{cor}

Here and hereafter $\Delta^g$ (resp. $\nabla^g$) denotes indifferently the Laplace--Beltrami operator $\Delta^g (u)=-\textrm{div}\left(\textrm{grad}(u)\right)$ (resp. the Levi-Civita connection) with respect to the metric $g$ or the lifted metric $\widetilde g$, depending on whether we work on $X$ or $\wtX$.
\begin{proof}
This is a completely standard application of the maximum principle in the presence of subsolutions of the heat operator $\dfrac{\partial}{\partial t}+ { \Delta}^g$ (Moser--Harnack's inequality \cite{Moser64}): we write $a(x,t)=  e_T(x,t) e^{-Ct}$ and observe that it satisfies 
\[ ( \dfrac{\partial}{\partial t}+ { \Delta}^g) (a)\leq 0\]
 so that the maximum principle yields $a(x,t)\leq \mathrm{sup}_{x\in X}a(x,0)$ or equivalently 
\[ e_T(x,t)\leq e^{Ct}\mathrm{sup}_{x\in X}e_T(x,0).\]
	 Now, it is obvious that the map $f_t$ is $\F$-invariant if and only if the function $e_T(\cdot,t)$ vanishes identically and the corollary follows.
\end{proof}

\subsection{Proof of Theorem \ref{TH:Tangentialbehavior}}\label{SS:proofoftangenergysubsol} Here and henceforth, we will identify $\F$, $\F^\perp$ and their dual to subbundle of $TX$ and $T^*X$ thanks to the ortogonal decomposition.

Let $f:\wtX\to Y$ be a smooth function and $f_*^ T$ its tangential differential in the direction of the foliation. We can regard $f_*^ T$ as a section of $f^{*} ( TY) \otimes {{\wtF}^*}\subset f^{*} ( TY) \otimes {T^*\wtX}$. Let
\[
\widetilde{\nabla}:\Gamma( f^{*} ( TY) \otimes {T^*\wtX})\longrightarrow \Gamma ( f^{*} ( TY) \otimes {T^*\wtX}\otimes  {T^*\wtX})
\]
be the connection induced on the tensor product by $\nabla^g$ and $\nabla:=f^*(\nabla^h)$. It is defined by the following rule
$${ \widetilde{\nabla}}_w (s\otimes t)=\nabla_w s \otimes t + s\otimes \nabla_w^g t.$$ 
The connection $\widetilde{\nabla}$ splits as $\widetilde{\nabla}= \nabla^{\F,\F}+ \nabla^{\F^\perp,\F^\perp}+\nabla ^{\F,\F^\perp}+\nabla^{\F^\perp,\F}$ according to the canonical decomposition 

\[
{ T^*\wtX}\otimes   {T^*\wtX}= \left(  {\wtF}^*\otimes {\wtF}^{*}  \right)\oplus \left( \wtF^{ \perp *} \otimes \wtF^{ \perp *}  \right)\oplus  
\left(  \wtF^{  *}\otimes  {\wtF^{ \perp *}}\right) \oplus \left( {\wtF^{ *}}\otimes  {\wtF}^*\right)
\]
We will also consider $\nabla ^{\F,\F^\perp}$ as an operator
$$f^{*} ( TY) \otimes {T^*\wtX}\To f^{*} ( TY) \otimes \wtF^{ \perp *}\otimes  {\wtF}^*$$
via the canonical isomorphism $\wtF^{ \perp *}\otimes  {\wtF}^*\simeq \wtF^*\otimes \wtF^{ \perp *}$.
With this at hands, Theorem \ref{TH:Tangentialbehavior} is an easy consequence of the following result.

\begin{prop}\label{P:heatoperatorexpress}
When applied to the function $e_T$ the heat operator has the following expression:
\begin{equation}\label{E:heateqexpression}
\begin{split}
\left( \dfrac{\partial}{\partial t}+ { \Delta}^g\right) (e_T)=& -\mathrm{Ric}_{ \wtX} ({f_t}_*^T,{f_t}_*^T)+ \mathrm{Riem}_Y ({f_t}_*^T,{f_t}_*,{f_t}_*,{f_t}_*^T) \\
&-{ \left\arrowvert\nabla ^{\F,\F^\perp}{ f_t}_*^T+\nabla^{\F^\perp,\F}{ f_t}_*^T\right\arrowvert}^2 -{\left\arrowvert\nabla^{\F,\F}{ f_t}_*^T\right\arrowvert}^2 \\
&+{ \left\arrowvert\nabla^{\F^\perp,\F^{\perp}}{ f_t}_*^T\right\arrowvert}^2 +2{ \left\arrowvert\nabla^{\F^\perp,\F}{ f_t}_*^T\right\arrowvert}^2.
\end{split}
\end{equation}

\end{prop}
Before entering into the details of the proof, let us firstly explain the meaning of the two first terms in the right-hand side  which involve the Ricci and Riemann curvature tensors and why it implies Theorem \ref{Th:tangentialenergy}. Both terms $\mathrm{Ric}_{\wtX}$ and $\mathrm{Riem}_Y$ in \eqref{E:heateqexpression} are defined respectively from the Ricci tensor of $g$ on $X$ and the full curvature tensor of $h$ on $Y$ and extended in a natural way to tensors in $f^{*} ( TY) \otimes {T^*\wtX}$. We have chosen to keep notation $\mathrm{Ric}_{\wtX}$ and $\mathrm{Riem}_Y$ to avoid cumbersome expressions. In the sequel we use the notation $\langle-,-\rangle$ for the scalar product induced by $g$ (or $\widetilde g$) and $\langle-,-\rangle_h$ for the scalar product induced by $h$.   We will denote $(e_1,\ldots,e_{m+n})$ any local orthonormal frame of $T \wtX$ such that $ (e_1,\ldots,e_m )$ is a local othonormal frame of $\wtF$ and $(e_1^*,\ldots,e_{m+n}^*)$ the dual coframe.

If $u_i=s_i\otimes t_i$ (for $i=1,\ldots,4$) are decomposable tensors in $f^{*} ( TY) \otimes {T^*\wtX}$, we can set $\mathrm{Ric}_{ \wtX} (u_1,u_2):=\langle s_1,s_2\rangle_h \mathrm{Ric}_{ \wtX}(t_1,t_2)$  and extend to general tensors by bilinearity. In particular, in terms of local frame as above, we can easily check that 
\begin{equation}\label{E:Ricciexp}
\begin{split}
&\mathrm{Ric}_{\wtX} ({f_t}_*^T,{f_t}_*^T)(x)=\sum_{i,k\leq m}\langle  \mathrm{Ric}_{ \wtX}(e_k,e_i){f_t}_*e_k ,{f_t}_* e_i\rangle_h \\
&=\sum_{\substack{i,\,k\leq m \\ l\leq m+n}} \left(\langle\nabla_{e_l}^g \nabla_{e_k}^g e_i,e_l\rangle -\langle\nabla_{e_k}^g \nabla_{e_l}^g e_i,e_l\rangle -\langle\nabla_{[ e_l,e_k]}^g  e_i,e_l\rangle\right) \langle { f_t}_* e_k,{ f_t}_* e_i\rangle_h\\
&=\sum_{\substack{i,\,k\leq m \\ l\leq m+n}} \left(\langle\nabla_{e_i}^g \nabla_{e_l}^g e_l,e_k\rangle -\langle\nabla_{e_l}^g \nabla_{e_i}^g e_l,e_k\rangle -\langle\nabla_{[ e_i,e_l]}^g  e_l,e_k\rangle\right) \langle { f_t}_* e_k,{ f_t}_* e_i\rangle_h
\end{split}
\end{equation}
where the last equality comes directly from the symmetry property 
\[\langle R_{\wtX} (u,v)w,z\rangle=\langle R_{\wtX} (w,z)u,v\rangle\]
of the Riemannian curvature tensor on $(\wtX, \widetilde{g})$.
Let us now focus on the second term in the right hand side of \eqref{E:heateqexpression} and denote by $R_Y$ the Riemannian curvature tensor on the target manifold $Y$ .

For decomposable tensors of $f^{*} ( TY) \otimes {T^*\wtX}$ as above, we can set
\[R_Y(u_1,u_2)u_3:= \langle t_2,t_3\rangle_g R_Y(s_1,s_2)s_3\otimes t_1\in  f^{*} ( TY) \otimes {T^*\wtX}.\] This formula allows to extend by multilinearity this map to general tensors $u_i\in f^{*} ( TY) \otimes {T^*\wtX}$ and also define 
\[\mathrm{Riem}_Y (u_1,u_2,u_3,u_4):= \langle R_Y(u_1,u_2)u_3,u_4\rangle_{ h\otimes g}\] where ${ h\otimes g}$ stands for the metric induced by $h$ and $g$ on $f^{*} ( TY) \otimes {T^*\wtX}$. In particular, in the local frame $(e_1,\ldots,e_{m+n})$, we obtain the explicit formula:
\begin{equation}\label{E:Seccurvexp}
\begin{split}
&\mathrm{Riem}_Y ({f_t}_*^T,{f_t}_*,{f_t}_*,{f_t}_*^T)= \sum_{\substack{ i\leq m\\ k\leq m+n}}\langle R_Y({f_t}_* e_i,{f_t}_* e_k){f_t}_*e_k,{f_t}_* e_i\rangle_h\\
 &=\sum_{\substack{ i\leq m\\ k\leq m+n}} \langle \nabla_{e_i}\nabla_{e_k}{ f_t}_* e_k,{ f_t}_* e_i\rangle_h - \langle \nabla_{e_k}\nabla_{e_i}{ f_t}_* e_k,{ f_t}_* e_i\rangle_h- \langle \nabla_{[ e_i,e_k]}{ f_t}_* e_k,{ f_t}_* e_i\rangle_h.
\end{split}
\end{equation}

Moreover, we have
$$\nabla^{\F^\perp,\F^{\perp}}{ f_t}_*^T=\sum_{ \substack{i\leq m\\ j,k>m}}\langle\nabla_{e_k} e_i,e_j\rangle_g {  f_t}_* e_i\otimes e_j^*\otimes e_k^*$$
so that 
 \begin{equation}\label{E:normperperp}
 { \left\arrowvert\nabla^{\F^\perp,\F^{\perp}}{ f_t}_*^T\right\arrowvert}^2  =\sum_{ \substack{i,l\leq m\\ k>m}}   \langle \nabla_{e_k}^g {e_i},\nabla_{e_k}^g {e_l} \rangle\langle{ f_t}_* e_l,{ f_t}_*e_i\rangle_h.
\end{equation}
In the same vein, we can write
 \begin{equation}\label{E:nablafperpf}
 \nabla^{\F^\perp,\F}{ f_t}_*^T=\sum_{\substack{i,l\leq m\\ k> m}}  \langle\nabla_{e_i}^g e_l,e_k\rangle { f_t}_* e_l \otimes e_k^*\otimes e_i^*
 \end{equation}
and infer that
 \begin{equation}\label{E:normperp}
 { \left\arrowvert\nabla^{\F^\perp,\F}{ f_t}_*^T\right\arrowvert}^2= \sum_{i,l,k\leq m}   \langle \nabla_{e_k}^g {e_i},\nabla_{e_k}^g {e_l} \rangle\langle{ f_t}_* e_l,{ f_t}_*e_i\rangle_h.
 \end{equation}
 
 It is now easy (modulo the proof of Proposition~\ref{P:heatoperatorexpress}) to get the upper bound given in Theorem~\ref{Th:tangentialenergy}. Indeed, fix a compact fundamental domain $K\subset \wtX$ with respect to the action of $\pi_1(X)$.\footnote{Recall that the energy density is a well defined function on $X$. More generally, it is not difficult to see that each of the six terms appearing in the RHS of \eqref{E:heateqexpression} remains unchanged when replacing  $f_t$ by $\varphi\circ f_t$, where $\varphi$ is an isometry of $Y$. In particular, it descends to $X$.} It then follows from the compactness of $\bigcup_{t\in[0,t_0]}f_t(K)$ together with the multilinearity of $\mathrm{ Ric}_{\wtX}$,  $\mathrm{Riem}_Y$ and expressions \eqref{E:Ricciexp}--\eqref{E:normperp}.
 
\begin{remark}
 When $\F$ has codimension zero, the terms in the right-hand side involving $\F^\perp$ of the equation~\eqref{E:heateqexpression} do not appear and we recover  the classical Eells--Sampson's formula \cite[Chapter~II, \S8-A]{Eelssampson64} (see also \cite[Formula~11]{Donaldson87}).
\end{remark}

\begin {proof}[Proof of Proposition \ref{P:heatoperatorexpress}]
Let $(e_i)_{i\le m+n}$ be a  local orthonormal frame centered at $x\in \wtX$ such that $e_i$ is tangent to $\wtF$ for $i\leq m$ and $e_i$ is  foliated otherwise. In particular, $[e_i,e_k]$ remains tangent to $\F$ for every $k$ and every $i\leq m$. We first proceed like the classical calculus of the first  variation formula of the  energy (see for instance \cite[p.130]{Urakawa93}). To this end, let us consider the mapping $F:I\times\wtX \rightarrow Y$ defined by $F(t,x)  =f_t(x)$ for $|t|<\varepsilon$. Set $f:=f_0$. In \emph{loc. cit.}, it is proven that $$\frac{\partial}{\partial t}(\frac{1}{2}  { \arrowvert { f_t}_* (e_i)\arrowvert}_h^2)=\langle \nabla_{e_i}\tau(f_t),{ f_t}_* e_i\rangle_h $$ so that, in view of the local writing of the tension field $\tau(f_t)$ given in \eqref{E:tensionfieldlocalexp}: 
	\begin{equation}\label{E:derivativeenergy}
	 \dfrac{\partial}{\partial t} (e_T)= \sum_{\substack{ i\leq m\\ k\leq m+n}} \langle \nabla_{e_i}\nabla_{e_k}{ f_t}_* e_k,{ f_t}_* e_i\rangle_h - \langle \nabla_{e_i}(\langle v,e_k\rangle { f_t}_* e_k),{ f_t}_* e_i\rangle_h
    \end{equation}
where $v=\sum_{ l\leq m+n}\nabla_{e_l}^g e_l$ and $\nabla$ still denotes (without specifying the parameter $t$) the pull-back of the Levi-Civita connection on $Y$ by $f_t$. 

Set $v_1= \sum_{l\leq m}	 \nabla_{e_l}^g e_l $ and $v_2=\sum_{l> m}	 \nabla_{e_l}^g e_l $.
Since the mean curvature vector field  $\tau_g$ is assumed to be foliated, we derive that $e_i \cdot\langle  v_1,e_k\rangle$
\textbf{vanishes identically} for $i\leq m$ and $k>m$. This is indeed the \textbf{ key point} of the calculation. According to Lemma \ref{L:covariantfoliated}, the same  holds for $e_i \cdot\langle  v_2,e_k\rangle$, $i\leq m$, $k>m$.

We also recall that for any vector fiels $u,w$ on $X$, we obtain the identity 
 \begin{equation}\label{E:torsionfree}
  \nabla_{u} f_*(w)- \nabla_{w} f_*(u)=f_*\left([u,w]\right)
 \end{equation}
as a consequence of the torsion-freeness of the Levi-Civita connection on $Y$ (see for instance \cite[Lemma 1.16, p.129]{Urakawa93}). Combining this with the previous vanishing properties, we get
\begin{multline}\label{E:derivativeplusbasic}
\sum_{\substack{i\leq m \\ k\leq m+n}}\langle \nabla_{e_i}(\langle v,e_k\rangle { f_t}_* e_k),{ f_t}_* e_i\rangle_h=\sum_{\substack{i,k\leq m \\ l\leq m+n}} \langle \langle\nabla_{e_i}^g \nabla_{e_l}^g e_l,e_k\rangle { f_t}_* e_k,{ f_t}_* e_i\rangle_h\\+ \sum_{\substack{i,k\leq m }}\langle v,\nabla_{e_i}^g {e_k} \rangle \langle  { f_t}_* e_k,{ f_t}_* e_i\rangle_h  \\ + \sum_{\substack{i\leq m\\ l\leq m +n}} \langle v,e_l\rangle \big(\langle \nabla_{e_l}{ f_t}_* e_i,{ f_t}_* e_i\rangle_h+  \langle { f_t}_* [e_i,e_l],{ f_t}_* e_i\rangle_h\big).
\end{multline} 

Write $\nabla_{e_i}^g {e_k}= \sum_{ l\leq m+n} \langle \nabla_{e_i}^g {e_k},e_l\rangle e_l$. From torsion-freeness \eqref{E:Torsionfreeg}, metric compatibility~\eqref{E:metriccompg} and involutivity of $\F$, we get
\begin{align*}
 \sum_{\substack{i,k\leq m }}\langle v,\nabla_{e_i}^g {e_k} \rangle \langle  { f_t}_* e_k,{ f_t}_* e_i\rangle_h= &\sum_{l\leq m+n}\langle v,e_l\rangle\underbrace{ \sum_{\substack{i, k\leq m }} \langle\nabla_{e_l}^g e_k,e_i\rangle\langle  { f_t}_* e_k),{ f_t}_* e_i\rangle_h}_{=0\ \textrm{thanks to}\ \langle\nabla_{e_l}^g e_k,e_i\rangle=-\langle\nabla_{e_l}^g e_i,e_k\rangle}\\ 
 &-  \sum_{\substack{i\leq m\\ l\leq m +n}} \langle v,e_l\rangle   \langle { f_t}_* [e_i,e_l],{ f_t}_* e_i\rangle_h 
\end{align*}

According to \eqref{E:derivativeenergy} and \eqref{E:derivativeplusbasic}, this leads to the simplified writing:
\begin{align*}
\dfrac{\partial}{\partial t} (e_T)= &\sum_{\substack{i\leq m \\ k\leq m+n}} \langle \nabla_{e_i}\nabla_{e_k}{ f_t}_* e_k,{ f_t}_* e_i\rangle_h -\sum_{\substack{i,\,k\leq m \\ l\leq m+n}} \langle\nabla_{e_i}^g \nabla_{e_l}^g e_l,e_k\rangle \langle { f_t}_* e_k,{ f_t}_* e_i\rangle_h\\-
&\sum_{\substack{i\leq m \\ k\leq m +n}} \langle v,e_k\rangle \langle \nabla_{e_k}{ f_t}_* e_i,{ f_t}_* e_i\rangle_h.
\end{align*}

On the other hand, the Laplace operator reads  
\begin{align*}
 \Delta^g (e_T)= & \sum_{\substack{i\leq m \\ k\leq m +n}} \langle v,e_k\rangle \langle \nabla_{e_k}{ f_t}_* e_i,{ f_t}_* e_i\rangle_h\\
 &- \sum_{\substack{i\leq m\\ k\leq m+n}} \langle \nabla_{e_k}^2{ f_t}_* e_i\,{ f_t}_* e_i\rangle_h +  \arrowvert\nabla_{e_k}{ f_t}_* e_i\arrowvert_h^2.
\end{align*}
Still exploiting repeatedly identity \eqref{E:torsionfree}, we get
\[\nabla_{e_k}^2{ f_t}_* e_i=\nabla_{e_k}\nabla_{e_i}{ f_t}_*{e_k}-\nabla_{ [e_i,e_k]}{ f_t}_* e_k+{ f_t}_* ( [e_k,[e_k,e_i]]).\]
If we combine this with the last lines of~\eqref{E:Ricciexp} and~\eqref{E:Seccurvexp}, we can deduce that
\begin{equation}\label{E:expheateqriemric}
\begin{split}
&\left( \dfrac{\partial}{\partial t}+ { \Delta}^g\right) (e_T)= \mathrm{Riem}_Y ({f_t}_*^T,{f_t}_*,{f_t}_*,{f_t}_*^T)-\mathrm{Ric}_{\wtX} ({ f_t}_*^T,{ f_t}_*^T)\\
&-\sum_{\substack{i\leq m \\ k\leq m +n}} \arrowvert\nabla_{e_k}{ f_t}_* e_i\arrowvert_h^2 -2\langle\nabla_{ [e_i,e_k]}{ f_t}_* e_k,{ f_t}_* e_i\rangle_h +\langle{ f_t}_* ( [e_k,[e_k,e_i]]),{ f_t}_* e_i\rangle_h\\
&-\sum_{\substack{i,l\leq m \\ k\leq m +n}} \langle \nabla_{[e_i,e_k]}^g e_k, e_l\rangle\langle { f_t}_* e_i,{ f_t}_* e_l\rangle_h + \langle \nabla_{e_k}^g\nabla_{e_i}^g e_k,e_l\rangle\langle{ f_t}_* e_l,{ f_t}_*e_i\rangle_h.
\end{split}
\end{equation}
 
On the other hand, $ [e_k,[e_k,e_i]]$ is tangent to $\F$ for $i\leq m$, so that we can write $$[e_k,[e_k,e_i]]	=\sum_{l\leq m}\langle [e_k,[e_k,e_i]],e_l\rangle e_l.$$
Using \eqref{E:Torsionfreeg}, this yields
\begin{align*}
&\langle{ f_t}_* ( [e_k,[e_k,e_i]]),{ f_t}_* e_i\rangle_h=-\sum_{l\leq m}\langle \nabla_{e_k}^g\nabla_{e_i}^g e_k,e_l\rangle\langle{ f_t}_* e_l,{ f_t}_*e_i\rangle_h\\
&+\sum_{l\leq m}\langle \nabla_{[e_i,e_k]}^g e_k, e_l\rangle\langle f_* e_i,f_* e_l\rangle_h +\sum_{l\leq m}\langle { \nabla_{e_k}^{ g}}{ \nabla_{e_k}^{ g}} e_i, e_l\rangle\langle f_* e_i,f_* e_l\rangle_h
\end{align*}
Applying again \eqref{E:Torsionfreeg} together with \eqref{E:metriccompg}, we can express the last term as
\begin{align*}
\sum_{l\leq m}\langle { \nabla_{e_k}^{ g}}{ \nabla_{e_k}^{ g}} e_i, e_l\rangle\langle f_* e_i,f_* e_l\rangle_h=&-\sum_{l\leq m}\langle \nabla_{e_k}^g {e_i},\nabla_{e_k}^g {e_l} \rangle\langle{ f_t}_* e_l,{ f_t}_*e_i\rangle_h\\
&-e_k\cdot\sum_{l\leq m}\langle e_i, \nabla_{e_k}^ g e_l\rangle\langle f_* e_i,f_* e_l\rangle_h
\end{align*}
whence 
\begin{equation}\label{E:expliebracket}
\begin{split}
&\sum_{\substack{i\leq m \\ k\leq m +n}} \langle{ f_t}_* ( [e_k,[e_k,e_i]]),{ f_t}_* e_i\rangle_h=-\sum_{\substack{i,l\leq m \\ k\leq m +n}}\langle \nabla_{e_k}^g\nabla_{e_i}^g e_k,e_l\rangle\langle{ f_t}_* e_l,{ f_t}_*e_i\rangle_h\\
&+\sum_{\substack{i,l\leq m \\ k\leq m +n}}\langle \nabla_{[e_i,e_k]}^g e_k, e_l\rangle\langle f_* (e_i),f_* (e_l)\rangle_h
-\langle \nabla_{e_k}^g {e_i},\nabla_{e_k}^g {e_l} \rangle\langle{ f_t}_* e_l,{ f_t}_*e_i\rangle_h\\
&-\sum_{k\leq m+n}e_k\cdot\underbrace{ \sum_{i,l\leq m}\langle e_i, \nabla_{e_k}^ g e_l\rangle\langle f_* e_i,f_* e_l\rangle_h}_{=0\ \textrm{thanks to}\ \langle e_i, \nabla_{e_k}^ g e_l\rangle=-\langle e_l, \nabla_{e_k}^ g e_i\rangle.}.
\end{split}
\end{equation}
When re-injecting \eqref{E:expliebracket} into  \eqref{E:expheateqriemric}, we thus find
\begin{equation}\label{E:heattangenergyexp}
\begin{split}
&\left( \dfrac{\partial}{\partial t}+ { \Delta}^g\right) (e_T)=  \mathrm{Riem}_Y ({f_t}_*^T,{f_t}_*,{f_t}_*^T{f_t}_*)-\mathrm{Ric}_{ \tilde{ X}} ({ f_t}_*^T,{ f_t}_*^T) \\
&+\sum_{\substack{i\leq m \\ k\leq m +n}} 2\langle\nabla_{ [e_i,e_k]}{ f_t}_* e_k,{ f_t}_* e_i\rangle_h - \arrowvert\nabla_{e_k}{ f_t}_* e_i\arrowvert_h^2\\
&+\sum_{\substack{i,l\leq m \\ k\leq m +n}}  \left(\langle \nabla_{e_k}^g {e_i},\nabla_{e_k}^g {e_l} \rangle -2 \langle \nabla_{[e_i,e_k]}^g e_k, e_l\rangle\right)\langle f_* e_i,f_* e_l\rangle_h.
\end{split}
\end{equation}
 
In order to simplify our computations, we will suppose hereafter, without loss of generality, that the orthonormal frame ${(e_i)}_{i\leq m}$ of  $\F$ is \textit{ tangentially geodesic} at $x$ (see Lemma~\ref{L:tangentiallygeodexistence}). Because $[e_i,e_l](x)=0$ for $i,l\leq m$,  the formula \eqref{E:heattangenergyexp} becomes 
 \begin{equation}\label{E:heatoperatorexpression}
 \begin{split}
 &\left( \dfrac{\partial}{\partial t}+ { \Delta}^g\right) (e_T)= \mathrm{Riem}_Y ({f_t}_*^T,{f_t}_*,{f_t}_*^T{f_t}_*)-\mathrm{Ric}_X ({ f_t}_*^T,{ f_t}_*^T)\\
 &+\sum_{\substack{i\leq m\\ k> m }}2\langle\nabla_{ [e_i,e_k]}{ f_t}_* e_k,{ f_t}_* e_i\rangle_h -\sum_{\substack{i\leq m\\ k\leq m+n }} \arrowvert\nabla_{e_k}{ f_t}_* e_i\arrowvert_h^2\\
& +\sum_{\substack{i,l\leq m\\ k\leq m +n }}   \langle \nabla_{e_k}^g {e_i},\nabla_{e_k}^g {e_l} \rangle\langle{ f_t}_* e_l,{ f_t}_*e_i\rangle_h \\
&-2\sum_{\substack{i,l\leq m\\ k>m }} \langle \nabla_{[e_i,e_k]}^g e_k, e_l\rangle \langle f_* e_i,f_* e_l\rangle_h.
\end{split}
\end{equation}
Now, in view of proving the equality of Proposition~\ref{P:heatoperatorexpress}, let us compute
\begin{equation}\label{E:nablaffperpgeod}
\nabla ^{\F,\F^\perp}{ f_t}_*^T = \sum_{\substack{i\leq m\\ k> m}} \nabla_{e_k}{ f_t}_* e_i\otimes e_k^*\otimes e_i^* +  \sum_{\substack{l\leq m}} \underbrace{ \langle\nabla_{e_k}^g e_i,e_l\rangle}_{=0} { f_t}_* e_i\otimes e_k^*\otimes e_l^* .
\end{equation}
On the other hand, by noticing that 
\begin{align*}
 \sum _{l\leq m} \langle\nabla_{e_i}^g e_l,e_k\rangle{f_t}_* e_l  &=-\sum _{l\leq m}\langle\ e_l,\nabla_{e_i}^g e_k\rangle {f_t}_* e_l  \\
&=-\sum _{l\leq m} \langle\ e_l,[e_i,e_k]\rangle{f_t}_* e_l = -{ f_t}_* [ e_i,e_k] 
\end{align*}
and accordingly to \eqref{E:nablafperpf}, we get
\begin{equation}\label{E:nablafperpfgeod}
\nabla^{\F^\perp,\F}{ f_t}_*^T= -\sum_{\substack{i\leq m\\ k> m }}{ f_t}_* [ e_i,e_k]\otimes e_k^*\otimes e_i^*. 
\end{equation}
As a consequence of \eqref{E:nablaffperpgeod} and \eqref{E:nablafperpfgeod}, we have:
\begin{multline*}\left\arrowvert\nabla ^{\F,\F^\perp}{ f_t}_*^T+\nabla^{\F^\perp,\F}{ f_t}_*^T\right\arrowvert^2=\\ \sum_{\substack{i\leq m\\ k>m }}  { \arrowvert\nabla_{e_k}{ f_t}_* e_i\arrowvert}_h^2  +{ \arrowvert{ f_t}_* [ e_i,e_k] \arrowvert}_h^2  -2\langle\nabla_{e_k}{ f_t}_* e_i, { f_t}_* [ e_i,e_k]\rangle_h .
\end{multline*}
By virtue of \eqref{E:torsionfree}, we can rewrite 
\[\sum_{\substack{i\leq m\\ k>m }} \langle\nabla_{e_k}{ f_t}_* e_i, { f_t}_* [ e_i,e_k]\rangle_h = \sum_{\substack{i\leq m\\ k>m}} \langle\nabla_{e_i}{ f_t}_* e_k, { f_t}_* [ e_i,e_k]\rangle_h -\left\arrowvert{ f_t}_* [ e_i,e_k] \right\arrowvert_h^2 .
\]
Now, thanks to the expansion
\[[ e_i,e_k]=\sum_{l\leq m}\langle [ e_i,e_k],e_l\rangle e_l\]
and Lemma~\ref{L:bracketidgeod}, we conclude that
\[\sum_{\substack{i\leq m\\ k>m}} \langle\nabla_{e_i}{ f_t}_* e_k, { f_t}_* [ e_i,e_k]\rangle_h  = \sum_{\substack{i\leq m\\ k>m}}\langle\nabla_{[e_i,e_k]}{ f_t}_* e_k,{ f_t}_* e_i \rangle_h \]
so that
\begin{equation}\label{E:squarenormsum}
\begin{split}
&	\left\arrowvert\nabla ^{\F,\F^\perp}{ f_t}_*^T+\nabla^{\F^\perp,\F}{ f_t}_*^T\right\arrowvert^2 =\\ &\sum_{\substack{i\leq m\\ k>m }}  { \arrowvert\nabla_{e_k}{ f_t}_* e_i\arrowvert}_h^2  +3{ \arrowvert{ f_t}_* [ e_i,e_k] \arrowvert}_h^2  -2\langle\nabla_{[e_i,e_k]}{ f_t}_* e_k,{ f_t}_* e_i \rangle_h .
\end{split}
\end{equation}
On the other hand, comparison of \eqref{E:normperp} and \eqref{E:nablafperpfgeod} yields
\begin{equation}\label{E:squarenormfperpf}
\begin{split}
\left\arrowvert\nabla^{\F^\perp,\F}{ f_t}_*^T\right\arrowvert^2
&=\sum_{i,l,k\leq m}   \langle \nabla_{e_k}^g {e_i},\nabla_{e_k}^g {e_l} \rangle(x)\langle{ f_t}_* e_l,{ f_t}_*e_i\rangle_h \\
&=\sum_{ \substack{i\leq m\\k>m}} \left\arrowvert{ f_t}_* [ e_i,e_k] \right\arrowvert_h^2 .
\end{split}
\end{equation}
Recall also (\emph{cf.} \eqref{E:normperperp}) that
\begin{equation}\label{E:squarenormfperpfperp}
\begin{split}
 \left\arrowvert\nabla^{\F^\perp,\F^{\perp}}{ f_t}_*^T\right\arrowvert^2 =\sum_{ \substack{i,l\leq m\\ k>m}}   \langle \nabla_{e_k}^g {e_i},\nabla_{e_k}^g {e_l} \rangle(x)\langle{ f_t}_* e_l,{ f_t}_*e_i\rangle_h .
 \end{split}
\end{equation}
According to the equations~\eqref{E:squarenormsum}, \eqref{E:squarenormfperpf}, \eqref{E:squarenormfperpfperp}  and to the equality
$${\left\arrowvert\nabla^{\F,\F}{ f_t}_*^T\right\arrowvert}^2 =\sum_{i,k\leq m} \arrowvert\nabla_{e_k}{ f_t}_* e_i\arrowvert_h^2, $$
the equation~\eqref{E:heatoperatorexpression} can be rewritten as
\begin{equation}\label{E:heateqexpression-final}
\begin{split}
&\left( \dfrac{\partial}{\partial t}+ { \Delta}^g\right) (e_T)= -\mathrm{Ric}_{ \wtX} ({f_t}_*^T,{f_t}_*^T)+ \mathrm{Riem}_Y ({f_t}_*^T,{f_t}_*,{f_t}_*,{f_t}_*^T) \\
&-{ \left\arrowvert\nabla ^{\F,\F^\perp}{ f_t}_*^T+\nabla^{\F^\perp,\F}{ f_t}_*^T\right\arrowvert}^2  -{\left\arrowvert\nabla^{\F,\F}{ f_t}_*^T\right\arrowvert}^2  +{ \left\arrowvert\nabla^{\F^\perp,\F^{\perp}}{ f_t}_*^T\right\arrowvert}^2 \\ 
&-2\sum_{\substack{i,l\leq m\\k>m}}\langle \nabla_{[e_i,e_k]}^g e_k, e_l\rangle\langle f_* e_i,f_* e_l\rangle_h +4\left\arrowvert\nabla^{\F^\perp,\F}{ f_t}_*^T\right\arrowvert^2 .
\end{split}
\end{equation}
It remains to identify the quantity (evaluated at the point $x$)
$$S:=\sum_{\substack{i,l\leq m\\k>m}}\langle \nabla_{[e_i,e_k]}^g e_k, e_l\rangle\langle f_* e_i,f_* e_l\rangle_h.$$
To do so, let us first expand $[e_i,e_k]$ with respect to the basis $(e_j)_{j\leq m}$. This yields:
$$S=\sum_{ \substack{i,l,j\leq m\\k>m}}\langle [e_i,e_k], e_j\rangle\langle \nabla_{e_j}^g e_k, e_l\rangle\langle f_* e_i,f_* e_l\rangle_h.$$
Finally, let us remark that
$$\langle [e_i,e_k], e_j\rangle = -\langle e_k,\nabla_{e_j}^g e_i\rangle (x)\quad\mathrm{and}\quad \langle \nabla_{e_j}^g e_k, e_l\rangle (x)=-\langle  e_k, \nabla_{e_j}^ge_l\rangle (x).$$
According to \eqref{E:nablafperpf}, we get after summation
$$S=\left\arrowvert\nabla^{\F^\perp,\F}{ f_t}_*^T\right\arrowvert^2.$$
We thus obtain the sough formula of Proposition~\ref{P:heatoperatorexpress} by reporting this identification into the heat operator expression given in the equation~\eqref{E:heateqexpression-final}.
\end{proof}

\section{Construction of equivariant foliated harmonic and holomorphic maps}\label{S:equivariant-foliated-harmoni-maps}

In this section, we put together results from Sections~\ref{S:structure riemannian foliation} and~\ref{S:evolution harmonic map} to prove existence of harmonic or holomorphic maps invariant under some foliations.

\subsection{The Riemaniann case}

We first state a general existence result for some TC foliations over a compact base.

\begin{thm}\label{TH:existence+uniqharmonic}
Let $(X,\F)$ be a foliated manifold where $X$ is compact and $\F$ is TC with structural Lie algebra $\mathfrak{g}:=\frg{X,\F}$ semi-simple without compact factors. Let us consider $g$ a metric on $X$ which is bundle-like with respect to $\F$ and such that the mean curvature vector of the leaves is foliated, \emph{i.e.} $(\F,g)$ is tense (see Theorem~\ref{TH:Dominguezth}). Then there exists a unique $\Psi$ that is harmonic with respect to the lift $\widetilde{g}$ of the metric on $\wtX$ and to the Killing metric on $\mathcal S_\mathfrak{g}$ (the symmetric space associated with $\mathfrak{g}$). 
\end{thm}
\begin{proof}
Thanks to the item~\ref{I:density} of Theorem~\ref{TH:equivariantsym}, there exists a \textit{basic} (in other words $\F$-invariant) $\rho_\F$-equivariant map
\[f_0=\Psi:\wtX\To \mathcal S_\mathfrak{g}\]
where $\rho_\F\colon\pi_1(X)\to \Isom(\mathcal{S}_\mathfrak{g})$ is the monodromy of the commuting sheaf whose image is dense in the identity component according to Proposition~\ref{prop:dense image}.

We can now apply the existence criterion of twisted harmonic map given in \cite{Labourie91} or \cite{Corlette-jdg}. As the image of $\rho_\F$ does not fix any point on the boundary of $\mathcal S_\mathfrak{g}$, we can apply \cite[Th\' eor\`emes~0.1 et~0.2]{Labourie91} to infer the existence of a $\rho_\F$-equivariant \textit{harmonic} map $f_\infty: \wtX\to \mathcal S_\mathfrak{g}$ which is obtained as the limit of a subsequence of $(f_{t_n}),\ t_n\longmapsto +\infty$ where  $f_t$ is a solution of the evolution equation \eqref{E:evolution} at time $t$ with initial datum $f_{t_0}=f_0$. Moreover, $f_\infty$ remains $\F$-invariant according to Corollary~\ref{C:remainsleafwiseconstant}.

Concerning the uniqueness, we can first assume, up to passing to a finite \'etale cover of $X$ that the representation takes values in $\Isom^0 (\mathcal S_\mathfrak{g})$, the neutral component of $\Isom (\mathcal S_\mathfrak{g})$. Let $\mathcal S_\mathfrak{g}=\mathcal S_1\times\cdots\times \mathcal S_p$ the decomposition of $\mathcal S$ as a Riemannian product of irreducible symmetric spaces. The image of $\rho_\F$ acts diagonally and isometrically with respect to this decomposition.  By projection  to each factor, we inherit a $\rho_i$-equivariant harmonic map $\Psi_i: \wtX\to \mathcal S_i$  where $\rho_i:\pi_1(X)\to \Isom^0 (\mathcal S_i)$ is the corresponding representation. By density, the image of $\rho_i$ does not preserve any non-trivial subspace of ${ \mathfrak g }_i:= \Lie(\Isom^0 (\mathcal S_i))$. By virtue of Corlette's uniqueness Theorem \cite[Theorem~3.4]{Corlette-jdg}, the $\Psi_i$'s are the only ones $\rho_i$-equivariant harmonic maps. In particular $\Psi$ is the unique $\rho_\F$-equivariant harmonic map, as desired.
\end{proof}
 
Using the results from Subsection~\ref{SS:asstrfrbundle}, we can apply the above statement to the transverse frame bundle of a Riemannian foliation $\F$. We first need the following result by Nozawa.

\begin{thm}[\emph{cf.} \protect{\cite[Theorem~2]{Nozawarigid2010}}]\label{T:minimalizableSS}
Let $(X,\F)$ a compact Riemannian foliated manifold and let us assume in addition that the structural Lie algebra of $\F$ is semi-simple. Then $\F$ is taut.
\end{thm}

With this at hands, we can state the following general existence result for foliated harmonic maps.

\begin{thm}\label{T:harmonic-riemannian}
Let $\F$ be a transversely orientable Riemannian foliation on a compact manifold $X$ and let us assume that its structural Lie algebra $\mathfrak{g}:=\frg{X,\F}$ is semi-simple without compact factors (as above $\mathcal S_\mathfrak{g}$ will be the associated symmetric space). Let $g$ be a bundle-like metric with vanishing mean curvature vector field (it exists by virtue of Theorem~\ref{T:minimalizableSS}).
 	
Let
\[\rho_\F\colon \pi_1 (X)\To \Aut(\mathfrak g)\simeq \Isom(\mathcal S_\mathfrak{g})\]
be the monodromy representation attached to the commuting sheaf $\cC_\F$. Then there exists a surjective and submersive $\rho_\F$-equivariant \emph{harmonic} map
\[\Psi\colon \wtX\To \mathcal S_\mathfrak{g}\]
with connected fibers which is constant on the leaves of $\wtF$ (the lift of $\F$ to $\wtX$).
\end{thm}
\begin{proof}
Keep the notations of \ref{SS:asstrfrbundle}. Consider the lift $\F^\sharp$ of $\F$ on the orthonormal frame bundle $X^\sharp$ together with its bundle like metric $g^\sharp$. According o the equality~\eqref{E:basicMCrelated}, $(\F^\sharp,g^\sharp)$ is also \textit{taut}, so that the conclusion of Theorem~\ref{TH:existence+uniqharmonic} is valid: there exists a unique $\widetilde{\F^\sharp}$-invariant and $\rho_{\F^\sharp}$-equivariant harmonic map $\Psi_0^\sharp:\wtsX\to \mathcal S_\mathfrak{g}$. A priori, it is not completely obvious that ${ \Psi_0}^\sharp$ factors through a a $\rho_\F$-equivariant map defined on $\wtX$. Actually, this holds true according to the following trick which is inspired from \cite{Elkacimiharmonique96}.

  One  can lift the $SO(n)$-action on $X^\sharp$ to an isometric action  of the universal covering group $\widetilde{\SO(n)}$ on $\wtsX$ (with respect to the induced bundle-like metric) with the corresponding quotient map
\[q\colon\wtsX\To \wtX.\]
Consequently, if $\tau$ is the isometric transformation associated to an element of $\widetilde{SO(n)}$  the mapping $\Psi_0^\sharp\circ\tau$ is also $\rho^\sharp$-equivariant and harmonic. By the uniqueness part of Theorem~\ref{TH:existence+uniqharmonic}, $\Psi_0^\sharp$ is $\widetilde{\SO(n)}$ - invariant, hence factors through a $\rho_\F$-equivariant and $\wtF$-invariant  map  $\Psi_0: \wtX\to \mathcal S$. We can now apply the same deformation process as in the proof of Theorem~\ref{TH:existence+uniqharmonic} and thus obtain the expected (and necessarily unique) $\rho_\F$-equivariant and $\wtF$-invariant  harmonic map $\Psi$.\footnote{In fact, it is not difficult to observe, using that the fibers of $X^\sharp\to X$ are totally geodesic combined with Proposition \ref{P:relationtensionfieldbasictf}, that we can directly take $\Psi=\Psi_0$.} By  Lemma~\ref{L:equivariantlift}, $\Psi^\sharp=\Psi\circ q$ satisfies property \eqref{E:equivariancetrsymmetric} of Theorem~\ref{TH:equivariantsym}, so that $\Psi^\sharp$ is a surjective submersion with connected fibers by the item~\eqref{I:eqimplieslocaltrivial} of the same theorem. This immediately implies that $\Psi$ is so.
\end{proof}

\subsection{The case of a transversely K\" ahler foliation}
In the Kähler setting, rigidity of harmonic maps can be used to strengthen the conclusion of Theorem~\ref{T:harmonic-riemannian} and to produce transversely holomorphic maps.
Let us start with a Riemannian compact foliated manifold $(X,\F,g)$ where $( \F,g)$ is transversely K\" ahler with semi-simple structural Lie algebra without compact factors. As before (with a small shift), $m$ and $2n$ denote respectively the real dimension and codimension of $\F$.\footnote{For transversely K\"ahler foliation, the $\SO(2n)$-principal bundle $X^\sharp$ admits a reduction to a $U(n)$-principal bundle, but we will not use this property in the sequel.} The following two theorems, as well as their proofs, are strongly related to \cite[Proposition~4.3 and Corollary~4.4]{Frankel-annals}

\begin{thm}\label{T:harmonic-kahler}
Let $(X,\F,g)$ be a Riemannian compact foliated manifold where $(\F,g)$ is transversely K\" ahler, its structural Lie algebra $\mathfrak g$ being semi-simple without compact factors. Assume moreover that $(\F,g)$ is taut.
Let $\Psi: \wtX\to \mathcal S_\mathfrak{g}$ be the (unique) $\rho_\F$-equivariant harmonic map provided by Theorem~\ref{T:harmonic-riemannian}. Then $\mathcal S_\mathfrak{g}$ is Hermitian symmetric (of non-compact type).
\end{thm}

\begin{proof}
We follow closely Toledo's survey \cite{Toledo}. Here, we have to think that the relevant substitute for the complexification of the tangent bundle of the source manifold is the \textit{complexified normal bundle} together with its splitting into $(1,0)$ and $(0,1)$ parts provided by the transverse complex structure $J:=J_\F$:
\[N_\C\F= N^{1,0}\F \oplus N^{0,1}\F\]
and similarly for the lifted foliation $\wtF$.

Let us pick a bundle-like metric $g$ such that $(\F, g)$ is taut as in Theorem~\ref{T:minimalizableSS}. We begin with some useful observations. Recall that the tension field $\tau (\Psi)$ of $\Psi$ coincides with the basic one $\tau_b (\Psi)$. Consequently the harmonic equation $\tau (\Psi)=\tau_b (\Psi)=0$ reads  
\[ \tilde{\omega}^{n-1}\wedge d_\nabla d^c \Psi=0\] 
where $\tilde{\omega}$ is the lift of the transverse K\" ahler metric, $d_\nabla$ is the differentiation operator
\[A^k(\wtX, f^{*} T\sS_\mathfrak{g})\To A^{k+1}(\wtX, f^{*} T\sS_\mathfrak{g})\]
which extends the Levi-Civita connection on $f^{*} T\sS_\mathfrak{g}$ and which also acts (by restriction) on the complex of twisted basic forms $A_b^\bullet$. The term  $d^c\Psi$ is a basic twisted one form and stands for $J d\Psi$.  Consider the basic scalar-valued $3$-form
\[\eta_\Psi= { \|d^c \Psi\wedge d_\nabla d^c \Psi\|}^2\]
where the norm is taken with respect to the scalar product $f^{*}T\sS_\mathfrak{g}\otimes f^{*}T\sS_\mathfrak{g}\to \R$ induced by the Killing metric on $\sS_\mathfrak{g}$. Note also that $\eta_\Psi= \eta_{\varphi \circ \Psi}$ for any isometry $\varphi$ of $\sS_\mathfrak{g}$; in particular, $\eta=\eta_\Psi$ is actually well defined on $X$.

Up to passing to a double cover, we can assume that $\F$ is oriented.  The volume form defined by $g$ is thus $\nu=\omega^{ n}\wedge\chi$ where $\chi:=\chi_{ {\F}}$ is the characteristic form associated to the bundle-like metric $ g$ (\emph{cf.} \S\ref{SS:characteristicform}). By Stokes' Theorem , we have
\[
\int_X d (\eta\wedge \omega^{ n-2} \wedge \chi) =0.
\]
On the other hand, we know that $d (\eta\wedge \omega^{n-2})$ is basic and that the mean curvature $\kappa$ vanishes identically. From Lemma~\ref{L:Rummler/Stokes}, we obtain that 
\[ d (\eta\wedge \omega^{n-2} \wedge \chi) =d (\eta\wedge \omega^{n-2}) \wedge \chi. \]
We can then apply the punctual Hodge index theorem combined with the harmonic equation to deduce that Siu's vanishing theorem is still valid in our setting (see the proof of Theorem~3.1 in \cite{Toledo} and Generalization~1 in \textit{loc.~cit.}), namely
\begin{equation} \label{E:Siuvanishing}
d_\nabla d^c \Psi=0\quad\mathrm{and}\quad  R(d\Psi(V),d\Psi (W), d\Psi(\bar V), d\Psi (\bar W))=0
\end{equation}
for any (local) basic vector field $V$ and $W$ of type $(1,0)$. Here $R=-d_\nabla^2$ denotes the complexification of the curvature tensor on $T\sS_\mathfrak{g}\otimes \C$. 

Let $\mathfrak g=\mathfrak k\oplus\mathfrak p$ be a Cartan decomposition of the Lie algebra $\mathfrak g$. The following reasoning is again borrowed from \cite[\S4]{Toledo}. Pick a point $x\in \wtX$. From the previous vanishing properties and the fact that $\Psi$ is submersive, we deduce that $d\Psi (N_x^{1,0}\wtF)$ is an abelian subalgebra of $T_{\Psi (x)} \mathcal S_\mathfrak{g} \otimes\C\simeq \mathfrak p \otimes\C$. 
With this at hands and recalling that $\Psi$ is submersive and  $\sS_\mathfrak{g}$ is a symmetric space of non-compact type, we can conclude as in the discussion at the beginning of \S4 in \cite{Toledo} that $\sS_\mathfrak{g}$ is actually \textit{Hermitian symmetric}.  
\end{proof}
 
The fact that $\Psi$ has maximum rank also implies the following statement (\emph{cf.} \cite[Theorem~4.2]{Toledo}) and finally produces a transversely holomorphic map.
\begin{thm}\label{T:harmonic-holomorphic}
Up to replacing the complex structure by its conjugate on each irreducible factor of $\sS_\mathfrak{g}$, the map $\Psi$ is (transversely) holomorphic.
\end{thm}  
\begin{proof}
Up to replacing $X$ by a finite \'etale cover, we can suppose that the image of the representation $\rho_\F$ lies in the identity component $\Isom^0(\sS_\mathfrak{g})$ of the isometry group of $\sS_\mathfrak{g}$. We have an isometric splitting $\sS_\mathfrak{g}= { \sS}_1\times\cdots\times { \sS}_p$ into irreducible Hermitian symmetric spaces ordered in such a way that for some $p'$, ${ \sS}_{p'}\times\cdots\times { \sS}_{p}$ is the polydisk factor (maybe empty). The fundamental group $\pi_1(X)$ (viewed as the group of deck transformations) acts via the representation $\rho$ diagonally and isometrically on $\sS$. In particular, it inherits for every $i=1,\ldots,p$ a representation $\rho_i:\pi_1(X)\to \Isom^0(\sS_i)$ with dense image together with the foliated $\rho_i$-equivariant harmonic map $\Psi_i:\wtX\to \sS_i$ induced by $\Psi$ by projection. We need exactly to prove that for every factor $\sS_i$, $\Psi_* \circ J_{ \wtF}=J_i\circ \Psi_*$, where $J_i$ is one of the two 	$\Isom^0(\sS_i)$-invariant complex structures on $\sS_i$. Because $\Psi_i$ has maximal rank, this automatically holds in the case $i<p'$ according to Siu--Carlson--Toledo's rigidity results  \cite[Theorem~4.2]{Toledo} which can be derived from \eqref{E:Siuvanishing}. 

In the case where $i\geq p'$ and $\sS_i=\mathbb D$ is the Poincar\'e disk, we resort to the analysis developed in \cite{JostYau83}. Since the maximal rank condition and the property \eqref{E:Siuvanishing} are fulfilled, the local levels of $\Phi_i$ are given by $\{z_i= \text{const}\}$ where $z_i$ is a suitable holomorphic transverse coordinate (with respect to $\wtF$). Then, by connectedness of the fibers of $\Phi_i$, there exists on $\sS_i$ a well defined and necessarily unique complex structure $J$ such that $\Phi_i$ become (transversely) holomorphic with respect to $J$. Moreover, $J$ is invariant under the action of the image of $\rho_i$ by equivariance, hence by the whole action of $\Isom^0( \sS_i)$ by density of the representation. The proof of Theorem~\ref{T:harmonic-holomorphic} is thus complete.
\end{proof}

\section{de Rham decomposition for transversely K\" ahler foliation of quasi-negative transverse Ricci curvature}\label{S:DRdecomposition}
 
 Let $(X,\F,g)$ be a compact Riemannian foliated manifold satisfying the assumptions of Section~\ref{S:results}, that is 

	\begin{minipage}{0.8\textwidth}	
	\begin{itemize}
 	\item $\F$ is transversely K\" ahler and from now on the integer $n$ will denote its complex codimension 
 	\item The transverse Ricci form $\gamma=\textrm{Ric}(\overline g)$ is quasi-negative.
 	\item $\F$ is homologically orientable: the basic de Rham cohomology class of the basic volume form induced by $\overline g$ is non-trivial.
 \end{itemize}
	\end{minipage}
 \eqnum\label{P:propertiesofF}
\bigskip

We will denote by $\omega$ the fundamental form of the transverse K\" ahler metric $\overline g$. This is a basic $(1,1)$-form which is positive in the transverse direction. As in the classical setting of complex manifolds, the restriction of the differential $d$ to complex valued basic forms decomposes as the sum of two operators $\partial$ and $\bar{\partial}$ of respective bi-degrees $(1,0)$ and $(0,1)$.
\medskip

Also, as orbifolds enter into the picture at the end this section, we refer to \cite{Moerdijkbook2003} or \cite{caramelloorbifolds} (and references therein) for the related basic definitions/properties.
\subsection{Preparatory material}\label{SS:preparatory}
The following result was proven in \cite{Touzet-toulouse}.

\begin{thm}\label{T:Touzet-semisimple}
The structural Lie algebra $\mathfrak g$ of $\F$ is semi-simple without compact factors.
\end{thm}
In particular we can (and we will) assume that $(\F,g)$ is \textit{taut} by Nozawa's Theorem \ref{T:minimalizableSS}. Actually, tautness can be directly derived from the homological orientability assumption  according to Masa's criterion recalled in Theorem~\ref{TH:Masatautcrit}.
According to Theorem~\ref{T:harmonic-holomorphic}, $\F$ admits an extension $\G$ (a priori not Riemannian) which is defined on the universal cover $\wtX$ of $X$ by the levels of a $J_\F$-holomorphic map
\[\Psi:\wtX\to\scrH:=\sS_\mathfrak{g}\] to a bounded symmetric domain and whose additional properties are listed in Theorem \ref{T:harmonic-riemannian} (to stick with the notation used in the statements of the introduction, the we drop the subscript $\mathfrak{g}$ and use $\scrH$ to denote the Hermitian symmetric space). Let $h_{\scrH}$ be the Killing metric on~$\scrH$. This is a K\" ahler--Einstein metric whith negative Ricci form $\mathrm{Ric}(h_\scrH)$. After normalization, we can suppose that $\mathrm{Ric}(h_\scrH) =-\omega_\scrH$ where $\omega_\scrH$ is the fundamental form of~$h_\scrH$. The  pull-back of $\omega_\scrH$ by $\Psi$ is $\pi_1(X)$-invariant and then descends to $X$ as a basic (with respect to both $\F$ and $\G$) $(1,1)$-forms that we will denote by $\Omega$.\footnote{Unless otherwise specified,``basic'' means basic with respect to the original foliation $\F$; of course, any covariant tensor which is basic with respect to $\G$ is automatically basic with respect to $\F$}

Set $p=\textrm{rk}_\C\ \G/\F$.  The short exact sequence of complex vector bundles\footnote{After complexification and identification with their $(1,0)$ parts.}
$$0\rightarrow \G/\F \rightarrow N\F\rightarrow N\G\rightarrow 0$$
implies that the corresponding Chern classes are related by 
\[c_1(N\F)=c_1(\G/\F)+c_1(N\G)\in H^2(X,\mathbb R). \]
Consequently, $c_1(\G/\F)$ is represented by the basic closed $(1,1)$-form $\alpha= \frac{1}{2\pi}\left(\gamma+\Omega\right)$. Note that $\alpha$ coincides with $\frac{1}{2\pi}\gamma$ in restriction to the leaves of $\G$ and in particular is quasi-negative on $\G/\F$.

On the other hand, $c_1(\G/\F)$ is also represented by the Chern curvature of  $\omega^p$ seen as a metric on  $\bigwedge^p\G/\F$, that is $c_1(\G/\F)=[\frac{1}{2\pi}\alpha']$ with  
$$\alpha'=-\sqrt{ -1}\partial\bar{\partial}\log \left(\frac{\omega^{ p}\wedge\Omega^{n-p}}{{\left|dz_1\wedge\cdots \wedge dz_p\right|}^2 \wedge \Omega^{n-p}}\right)$$
and where the local tranverse holomorphic coordinates $(z_1,\ldots,z_{n})$ are chosen in such a way that  $\G$ is defined as the kernel of $dz_{p+1}\wedge\cdots \wedge dz_n$. Actually, both $\alpha$ and $\alpha'$ are closed basic $(1,1)$-forms and they are related by 
\begin{equation}\label{E:relationcherncurvature}
\alpha=\alpha' +\sqrt{ -1}\partial\bar{\partial}f_0
\end{equation}
where $f_0$ is the basic function such that
$$\omega^p\wedge \Omega^{n-p}=e^{f_0} \omega^{n}.$$
 
\subsection{Existence of a special $\F$-basic $\G$-leafwise K\"ahler metric}\label{SS:Gleafwisemetric}

The following definition concerns (holomorphic) extensions of foliations (\emph{cf.} \S\ref{SS:transverse Kähler}) and makes precise the notion of ``invariant Kähler metric along the leaves of $\G$''.

\begin{dfn} 
A basic (with respect to $\F$) and closed $(1,1)$-form  $\omega'$ is said to be a \textit{$\G$-leafwise  K\"ahler metric} whenever  it is positive in restriction to $\G/\F$.
\end{dfn}

 Actually,  the restriction  $\restr{\F}{{ \mathcal L}_\mathcal{\G}}$ to any leaf ${ \mathcal L}_\G$ of $\G$ is transversely holomorphic on the (non compact) manifold ${ \mathcal L}_\G$, so that an $\omega'$ satisfying the property above induces an invariant transverse K\"ahler metric with respect to $\restr{\F}{{\mathcal L}_\G}$, whence the termininology used. 
  
The following result guarantees the existence of a special $\F$-basic {$\G$-leafwise  K\"ahler metric} and can be thought as a variant of the famous Yau's existence Theorem of solutions to the complex Monge--Amp\` ere equation. It is motivated by the relationship between $\alpha$ and $\alpha'$ given in \eqref{E:relationcherncurvature}. 
\begin{TH}\label{L:relativemongeampere}
Let $\F$ be a foliation satisfying the assumptions listed in \eqref{P:propertiesofF}. Assume in addition that $\F$ is orientable. Let $\chi$ be the characteristic form associated to the bundle-like metric $g$. Let $f$ be a real basic smooth function such that 
	$$\int_X e^f\omega^p\wedge \Omega^{n-p}\wedge \chi=\int_X \omega^p\wedge \Omega^{n-p}\wedge \chi.$$
Then there exists a basic real smooth function $\varphi$ which solves 
\begin{equation}\label{E:relativemongeampere}
{(  \omega + \sqrt{-1}\partial\bar{\partial}\varphi)}^p\wedge \Omega^{n-p}=e^f\omega^p\wedge \Omega^{n-p}
\end{equation} 
and such that $ \omega + \sqrt{-1}\partial\bar{\partial}\varphi$ is a $\F$-basic $\G$-leafwise  K\"ahler metric. Moreover, such a $\varphi$ is unique up to an additive constant.
\end{TH}
The proof being quite involved (particularly the existence part), we postpone it to Appendix~\ref{S:proofofmaintechnlemma}. Here we simply formulate some useful remarks and conclude the proofs of the main results of this article.

\begin{remark}Observe that when $\F$ is minimal, \emph{i.e.} has dense leaves, $f$ is automatically constant, hence identically zero by the normalization condition, so that existence part of the lemma is obvious. 
\end{remark}

\begin{remark}When  $\F$ is still supposed minimal, $\overline{ \gamma_h}:=-\mathrm{Ric}(\overline g)$ is the fundamental form of a transverse K\" ahler metric $h$. Now, $\textrm{Ric}(h)$ represents the same cohomology class than $-\gamma_h$ (namely $c_1(N\F)$). By El-Kacimi's basic $\partial\overline{\partial}$-Lemma \cite[Proposition~3.5.1]{Elkacimi90}, we can conclude that $\mathrm{Ric}(h)=-\gamma_h$. That is $h$ is (transversely) K\"ahler--Einstein.
\end{remark}

By adding the observations made at the end of \S\ref{SS:preparatory} and more specially the equality~\eqref{E:relationcherncurvature}, we can formulate the following corollary.
\begin{cor}\label{C:ricciqnegative}
Let $\varphi_0$ a solution of \eqref{E:relativemongeampere} with $f=f_0 +c$ ($c$ a suitable normalizing constant) and set $\omega_{\varphi_0}=  \omega + \sqrt{-1}\partial\bar{\partial}\varphi_0$.
Let $\overline{\mathcal L}$ be a leaf of the holomorphic foliation $\overline{\G}$ induced by $\G$ on the local space of leaves $U/\F$. The the restriction of $\omega_{\varphi_0}$ to $\overline{\mathcal L}$ defines on $\overline{\mathcal L}$ a K\"ahler metric whose Ricci form coincide with $\alpha_{|\overline{\mathcal L}}$. 
\end{cor}

\subsection{Vanishing loci of isotropy subalgebras of the commuting sheaf }\label{SS:vanishingloci}
Recall that the differential $d\Psi$ provides a Lie algebra isomorphism between the  lift $r^* { \mathcal C}_\F$ of the commuting sheaf $r^* { \mathcal C}_\F$ to the universal cover $r:\wtX\to X$ and the Lie algebra $\Lieisom(\scrH)$ of infinitesimal isometries of $\scrH$. For any $s\in\scrH$, we will denote by $\Lieisom_s(\scrH)\subset \Lieisom(\scrH)$ the isotropy Lie subalgebra: 
$$\Lieisom_s(\scrH):=\{w\in \Lieisom(\scrH)|w(s)=0 \}.$$
\begin{thm}\label{TH:isotvanishingloci}
Let $s\in \scrH$ and $w\in \Lieisom_s(\scrH)$. Let $v$ be the unique element of $r^* { \mathcal C}_\F$ such that $d\Psi (v)=w$. Then $v$ vanishes identically on the fiber $F_s=\Psi^{-1}(s)$.
\end{thm}
\begin{proof}
Let $r^\sharp:\wtsX\to X^\sharp$ be the universal cover and $q:\wtsX\to \wtX$ be the natural projection. Set $\Psi^\sharp=\Psi\circ q$.  Let us also consider:
\[ r^*{ \mathcal C}_\F\supset r^*({ \mathcal C}_{\F,\,s}):= d\Psi^{-1}\left( \Lieisom_s(\scrH)\right)\] and similarly
\[ { (r^\sharp)}^*{ \mathcal C}_{ \F^\sharp}\supset  (r^\sharp)^*( \mathcal{C}_{ \F^\sharp,\, s})= d{ \Psi^\sharp}^{-1}\left( \Lieisom_s(\scrH)\right).\]
Let $\pi_{ \F^\sharp}:\wtsX\to P:= X^\sharp/\widetilde{ \F^\sharp} $ be the projection map onto the space of leaves. Recall that $P$ has the structure of a $G$-principal bundle (where $\Lie(G)=\mathfrak g$) over the basic manifold $W$ and such that the Lie algebra of fundamental vector fields coincides with  $(r^\sharp)^*{ \cC}_{ \F^\sharp}$ (identified with its projection via $\pi_{ \F^\sharp}$). Recall also that $G$ acts on the Hermitian symmetric space $\scrH$ via $\pi:G\to \Isom(\scrH)$.  Moreover, $\Psi^\sharp$ and its differential $d \Psi^\sharp$ are ``equivariant'' with respect to the representation $\rho_{\F^\sharp}$ and the adjoint action respectively; it means that the following identities hold for every $x\in \wtsX$, $\gamma\in\pi_1(X)$,  $h\in G$, and $v\in ( r^\sharp)^*( \mathcal{C}_{ \F^\sharp})$:
\begin{align*}
\Psi^\sharp (\gamma(x))&= \rho_{\F^\sharp} (\gamma)\left(\Psi^\sharp (x)\right),\\
d\Psi^\sharp\left(\textrm{Ad}(h)(v)\right)&=\textrm{Ad}\left(\pi( h)\right)\left(d\Psi^\sharp (v)\right),\quad\text{and}\\
d\gamma(v)&= \textrm{Ad}\left(\rho_{\F^\sharp}(\gamma)\right) (v).
\end{align*}
Let $\xi$ be a transverse symmetric $(0,2)$ tensor on $X$ (\emph{e.g.} $\xi=\overline g$, the transverse invariant metric of $\F$), and $\widetilde {\xi}$, ${{ {\xi}}^\sharp}$, $\widetilde {{ {\xi}}^\sharp}$ be their respective pull-backs on $\wtX$, $X^\sharp$ and $\widetilde{X^\sharp}$. Those are  basic transverse symmetric $(0,2)$ tensors (respectively with respect to $\wtF$, $\F^\sharp$ and $\widetilde{ \F^\sharp}$), ${{ {\xi}}^\sharp}$ and $\widetilde {{ {\xi}}^\sharp}$ being in addition $\SO(2n)$ and  $ { \mathscr S}_{ 2n}$-invariant respectively. Denote by $B$ the Killing form on $\mathfrak g\simeq ( r^\sharp)^* \mathcal{C}_{ \F^\sharp}\simeq  { r}^*{ \mathcal C}_{ \F}$ and by $B_s$ its restriction to the maximal compact subalgebra  ${ \mathfrak g}_s\simeq (r^\sharp)^*( \mathcal{C}_{\F^\sharp,\, s})\simeq  r^*(\mathcal{C}_{\F,\, s})$. In particular $B_s$ is negative definite on ${ \mathfrak g}_s$ and $\textrm{Ad}(h^{-1})$ provides an isometry between $({ \mathfrak g}_s,B_s)$ and $({ \mathfrak g}_{ h( s)},B_{ h( s)})$ (here, we identify $\mathfrak g$ with ${ (r^\sharp)^*{ \mathcal C}_{ \F^\sharp}} $). In particular, if we fix $s$ and an orthonormal basis $(v_i)$ of ${ \mathfrak g}_s$ (with respect to the scalar product $-B_s$), the function $\|{ \mathfrak g}_s\|:\widetilde{X^\sharp}\to \R$ defined by $\|{ \mathfrak g}_s\|(x):=\sum_i \widetilde {{ {\xi}}^\sharp}(v_i(x),v_i(x))$ does not depend on the choice of $(v_i)$ and is a basic function which satisfies
\begin{equation}\label{E:invg(s,x)}
\|{ \mathfrak g}_s\|= \|{ \mathfrak g}_{\rho^\sharp ( \gamma)( s)}\|\circ\gamma 
\end{equation}
 for any $\gamma\in\pi_1(X^\sharp)$.
 In addition $\|{ \mathfrak g}_s\|$ is $ { \mathscr S}_{ 2n}$-invariant (because $( r^\sharp)^*\mathcal{C}_{\F^\sharp}$ is so).
 
 Define ${ \mathscr D}^\sharp:\widetilde{ X^\sharp}\to\R$ by ${ \mathscr D}^\sharp (x)=\|{ \mathfrak g}_{ \Psi^\sharp(x)}\|(x)$. From the equation~\eqref{E:invg(s,x)} we can deduce that ${ \mathscr D}^\sharp$ is a basic function, that is moreover invariant under the actions of $\pi_1(X^\sharp)$ and ${\mathcal S}_{ 2n}$. Actually, the ${\mathcal S}_{ 2n}$-invariance is inherited from that of $\|{ \mathfrak g}_s\|$, taking into account that the ${\mathcal S}_{ 2n}$-action on $\widetilde{X^\sharp}$ makes the map $\Psi^\sharp$ equivariant. Now, as the projection $q: \widetilde{ X^\sharp}\to  \widetilde{ X}$ induces a surjective morphism from $\pi_1(X^\sharp)$ onto $\pi_1(X)$, we can conclude that there exists a basic $\pi_1(X)$-invariant function $ \mathscr D:\wtX\to \R$ such that ${ \mathscr D}^\sharp={ \mathscr D}\circ q$.
 
Note also that ${ \mathscr D}^\sharp$ (and consequently  ${ \mathscr D}$) is smooth.  Indeed, pick a  fiber $F_s= ( \Psi^\sharp)^{-1} (s)$ and remark that the restriction of $\|{ \mathfrak g}_s\|$ to $F_s$ is obviously smooth. Let $\Upsilon_s\subset G$ be a small transversal to the stabilizer of $G$ at $s\in\scrH$ passing through $e_G$. By the implicit function Theorem, we can fill out a small neighborhood $U$ of $s$ by $\{\pi ( h)(s)\mid h\in \Upsilon_s \}$ such that the correspondence $h\to \pi ( h)(s) $ induces a diffeomorphism between $\Upsilon_s$ and $U$. To any element $v\in  { \mathfrak g}_s$, we can thus associate on the neighborhood $( \Psi^\sharp)^{-1}(U)$ of $F_s$  a unique smooth basic vector field $\tau (v)$ characterized by 
\begin{itemize}
	\item for every $h\in \Upsilon_s$, $\tau (v)_{ |F_{ \pi( h)(s)}}$ is tangent to $F_{ \pi( h)(s)}=( \Psi^\sharp)^{-1}\left(\pi( h)(s)\right)$
	\item $\tau (v)_{ |F_{ \pi( h)(s)}}=\textrm{Ad}(h^{-1})(v)_{ |F_{ \pi( h)(s)}}.$ 
\end{itemize}

Roughly speaking, $\tau(v)$ is the ``transport'' of $v_{ |{F_s}}$ along $\Upsilon_s$. With this in mind, we thus observe that ${ \mathscr D}^\sharp$ coincides on $( \Psi^\sharp)^{-1} (U)$ with $\sum_i \widetilde { {\xi}^\sharp}(\tau( v_i),\tau( v_i))$ 
 where $v_i$ is any orthonormal basis of $( { \mathfrak g}_s,-B_s) $, thus proving the smoothness of ${ \mathscr D}^\sharp$ and $\mathscr D$. 

By construction, $\mathscr D$ only depends on the choice of the original tensor $\xi$. In the sequel, we will focus on basic symmetric tensors of the forms $\xi=\overline{g}_\varphi= \omega_\varphi (\cdot, J_\F(\cdot))$ where $\varphi:X\to \R$ is basic and $\omega_\varphi=\omega+\sqrt{-1}\partial\bar{\partial}\varphi$. We will restrict our attention for those  $\varphi$ for which the restriction of $\xi$ to $\G/\F $ is positive definite. Let $s\in\sS$ and denote by ${ \F}_s$ the restriction of $\wtF$ to $F_s$. We obtain in this way a foliated Riemannian manifold $(F_s,{ \F}_s)$ where the transverse Riemmanian structure is defined by the transverse K\" ahler metric $\xi_s$, setting $\xi_s={\tilde{\xi}}_{|F_s}$. Let $w$ a basic vector field on $F_s$, such that $v$ is the real part of a \textit{holomorphic} basic vector field (equivalently the Lie derivative ${ \mathcal L}_v J_\F$ vanishes). The transverse formulation of the classical Bochner--Weitzenb\"ock's formula (see for instance \cite[Proposition~3.1.8]{Kobayashi87})  reads 
\begin{equation}\label{E:trbochner}
-\dfrac{1}{2}\Delta_s\left(  { \| w\|}_{s}^2 \right)= { \|\nabla_s w\|}_{s}^2 -\textrm{Ric}_s(w,w)  
\end{equation}
where $\Delta_s$, $\nabla_s$, $\textrm{Ric}_s$, $\| \ \|$ are respectively the (transverse) Laplacian,  the $(1,0)$ part of the Levi-Civita connection,  the Ricci form and the norm with respect to $\xi_s$. In particular, this formula holds whenever $w$ is the restriction $v_s$ of an element of $ r^* ( \mathcal{C}_{\F,\, s})$ to $F_s$.

Now, Theorem~\ref{L:relativemongeampere} and Corollary~\ref{C:ricciqnegative} provide us with the existence of $\varphi_0$ such that for every $s\in\scrH$ and every $x\in F_s$, $\textrm{Ric}_s$ is negative semi-definite and is moreover negative somewhere. Fix $\varphi_0$ from now on. By semi-negativity of the Ricci form, \emph{i.e.} the semi-negativity of $\alpha$ on $\G/\F$, the right-hand side of \eqref{E:trbochner} is non-negative. Moreover, as $\mathscr D$ descends to $X$, it reaches its maximum at $x_0\in \wtX$. Let $s_0=\Psi(x_0)$. As the restriction of $\mathscr D$ to $F_{s_0}$ is given by ${ \mathscr D}_{ s_0}=\sum_i { \| { v_i}_{s_0}\|}_{s_0}^2$, $(v_i)$ an orthonormal basis of  $r^*( \mathcal{C}_{\F,\, s_0})$, we can derive from \eqref{E:trbochner} and the Hopf's maximum principle that ${ \mathscr D}_{ s_0}$ is constant. As ${\mathscr D}$ descends to $X$ and is continuous, this implies that ${ \mathscr D}$ is indeed constant on the whole of $\wtX$ by minimality of the foliation $\G$ and is finally identically zero by quasi-negativity of $\alpha$ on $\G/\F$. This concludes the proof of Theorem~\ref{TH:isotvanishingloci}.
\end{proof}
\begin{remark}
The use of the Weitzenb\"ock formula above and the resulting vanishing property given above is the substitute in our setting to \cite[Corollary~4.5(3)]{Frankel-annals}.
\end{remark}

\subsection{Two complementary parallel holomorphic foliations and proof of Theorem~\ref{T:A}}
By Theorem~\ref{TH:isotvanishingloci}, $X$ is now endowed with a foliation $\overline{\F}$ satisfying the following properties:
\begin{itemize}
	\item $\overline{\F}$ is an extension of $\F$.
	\item $\overline{\F}$ is locally generated by $\F$ and ${ \mathcal C}_\F$. In particular, the leaves of $\overline{\F}$ are the topological closures of the leaves of $\F$.
	\item Let $x\in X$. Locally, the vanishing locus of the set $ { \mathcal C}_{\F,\,x}$ of local sections of ${ \mathcal C}_\F$ vanishing at $x$ coincides with the leaf of $\G$ through $x$. In particular,  the rank of $\overline{\F}$ is equal to $\rg(\F)$ + $\dim(\sS)$ and $\overline{\F}$ intersects $\G$  transversely along $\F$: $$ \overline{\F}\cap \G=\F.$$
\end{itemize}

Using in addition that the local sections $v$ of  ${ \mathcal C}_\F$ are real part of \textit{ holomorphic} basic vector fields $v^{1,0}$, we can apply the statement of \cite[Lemma~12.1]{Frankel-acta} to conclude that we have a $\overline g$-orthogonal decomposition
\begin{equation}\label{E:parallsplitting}
 N\F= \G/\F \oplus\overline \F/\F
 \end{equation}
whence the following result.
\begin{lemma}
The foliation $\overline{\F}$ is $J_\F$-holomorphic and the aforementioned orthogonal splitting is parallel with respect to the transverse Levi-Civita connection of the transverse K\"ahler metric $\overline g$.
 \end{lemma}
 
 \begin{proof}
 Observe first that $\overline{\F}/\F$ is $J_\F$-invariant as $\G/\F$ is so. Now, $\overline{\F}/\F$ is locally spanned by sections $v$ of ${ \mathcal C}_\F$. Hence ${ ( \overline{\F}/\F)}^{1,0}$ is locally spanned by the \textit{holomorphic} vector fields $v^{1,0}=v-\sqrt{-1}J_\F (v)$. This proves the first claim. Concerning the second point, this amounts to showing that on the local space of leaves $U/\F$, the orthogonal and holomorphic distributions induced by $\overline{\F}$ and $\G$ are parallel. This last property is a general fact in K\" ahler geometry (see for instance \cite[Theorem~2.1]{Johnson80}).
 \end{proof}
 
\begin{proof}[Proof of Theorem~\ref{T:A}]
As $\overline{\F}$ is a Riemannian foliation with closed leaves and the transverse metric is K\" ahler with quasi-negative Ricci curvature, we conclude that the leaf space $X/\overline \F$ is a K\"ahler orbifold with quasi-negative  canonical bundle, hence of the general type according to \cite{Puchol}.  The statement of Theorem \ref{T:A} follows directly.
\end{proof}

\vskip\baselineskip
The decomposition into parallell subbundles given in \eqref{E:parallsplitting} provides us with the existence of two transverse invariant K\"ahler metrics with quasi-negative Ricci curvature  ${ \overline g}_1$, ${ \overline g}_2$ for the foliations $\G$ and $\bar\F$ and such that $\overline g={ \overline g}_1\oplus { \overline g}_2$.

\subsection{Proof of Theorem~\ref{L:BSpaceleaves}}
Consider the universal cover map $r^\sharp:\wtsX\to X^\sharp$ and the $\SO(n)$-principal bundle $Y:=r^{-1} (X^\sharp)$ over $\wtX$.  The latter coincides with the transverse orthonormal frame bundle associated to the complete Riemannian foliated manifold $(\wtX, \wtF,\widetilde{g})$. By virtue of Lemma~\ref{L:devcentercriterion}, the lift ${ \wtF}^\sharp$ of $\wtF$ to $Y$ is \textit{simple}. According to Molino's theory, this implies that the space of leaves $(\wtX/\wtF, \tilde{ \overline g}$) has a (unique) orbifold structure such that the projection map $\wtX\to \wtX/\wtF$ is a smooth orbifold map. Indeed, $\wtX/\wtF$ is canonically identified with the space of orbits of the locally free action of the compact Lie group $\SO(n)$ on the manifold $Y/{ \wtF}^\sharp$ and the points of  $\wtX/\wtF$ with non-trivial isotropy correspond to leaves with non-trivial holonomy. In our setting, the leaf space $ \wtX/\wtF$ inherits a K\"ahler metric from the transverse structure of $\wtF$. \hfill\qed
	
\subsection{Proof of Theorem~\ref{T:CdeRhamdec}}
Let $\widetilde{\G}$ and $\widetilde{\overline{\F}}$ be the respective pull-backs of ${\G}$ and ${\overline{\F}}$ to $\wtX$. Denote by $( \G_1,{ \overline g}_1)$, $( \G_2,{ \overline g}_2)$ the respective foliations induced on $\wtX/\wtF$ together with their transverse K\"ahler metric ${ \overline g}_i$ coming from the orthogonal and parallel splitting 
$$( N\wtF,\widetilde{ \overline g})= ( \tilde \G/\wtF,\widetilde{ \overline g}_2) \oplus( \widetilde{ \overline \F}/\wtF, \widetilde{ \overline g}_1)$$
where $\widetilde{-} $ stands for the pull-back of the transverse metrics under consideration on the universal cover. Equivalently
$$( T(\wtX/\wtF),\widetilde{ \overline g})= (  \G_1,{ \overline g}_2) \oplus(\G_2, { \overline g}_1)$$
in the orbifold setting.

Note that the orbifold fundamental group $\pi_1^\textrm{orb}(\wtX/\wtF)$ is nothing but the fundamental group of the holonomy pseudogroup of $\wtF$ (\emph{cf.} \cite[Appendix~D]{Molino-livre}). Because we have a natural \textit{surjective} morphism $\pi_1(\wtX)=\{1\}\to \pi_1^\textrm{orb}(\wtX/\wtF)$ (\textit{loc.cit}) , we conlude that $\wtX/\wtF$ is a \textit{simply connected} orbifold. The orbifold version of de Rham decomposition Theorem \cite[Lemma~2.19]{Kleinergeometrizationorb} then yields:
\begin{equation}\label{E:dRdec}
\left(\wtX/\wtF,\widetilde{ \overline g}\right)\simeq \left(  \mathcal{L}_1/\wtF,{ \overline g}_2\right) \times \left(\mathcal{L}_2/\wtF, { \overline g}_1\right)
\end{equation}
where ${ \mathcal L}_1$, ${ \mathcal L}_2$ are (arbitrary) leaves of $\widetilde \G$ and $\widetilde{ \overline \F}$. Of course, the (local) isotropy groups acts diagonally with respect to this splitting.   Note also that the map $\Psi$ descends on $\wtX/\wtF$ as an orbifold holomorphic map  $\overline\Psi:\wtX/\wtF\to\scrH$. Set $\mathscr K= { \mathcal L}_1/\wtF$ and $\scrH'= { \mathcal L}_2/\wtF$. The levels of $\overline\Psi$ are precisely the leaves of the horizontal foliation in the K\"ahlerian product~\eqref{E:dRdec}. As $\Psi$ is submersive, the restriction 
\[\restr{\overline\Psi}{\scrH'}\colon \scrH'\To \scrH\]
is a local diffeomorphism. In particular $\scrH'$ is smooth (the isotropy groups acts trivially on the second factor). Moreover $\restr{\overline\Psi}{\scrH'}$ is a local isometry between complete Riemannian manifolds (for a suitable invariant K\" ahler metric on $\scrH$). The target $\scrH$ being simply connected, we can conclude that  $\restr{\overline\Psi}{\scrH'}$ induces an isometric biholomorphism between $\scrH'$ and $\scrH$ . This proves the three first items of Theorem~\ref{T:CdeRhamdec}. The item (4) follows from the fact that $\overline \F$ has compact leaves and that the action of $\pi_1(X)$ on $\scrH$ is dense.\hfill\qed

\subsection{The case where $\F$ is minimal}

We can give a simple proof (without resorting to the technical material developed in this article) of Theorem~\ref{T:CdeRhamdec} in the situation where $\F$ is minimal, \emph{i.e.} in the situation where every leaf is dense.\footnote{In the case of Riemannian foliation, this is equivalent to saying that at least one leaf is dense.} In this case $\wtX/\wtF$ should be reduced to $\mathscr H$. Actually, according to \cite{Touzet-toulouse}, the structural Lie algebra $\mathfrak g$ (\emph{cf.} \S\ref{SS:commutingsheaf}) of $\F$ is semi-simple without compact factors. We can then derive from the works of Haefliger \cite[Theorem~6.4.1]{Haefligerleafclosures88} that $\F$ is transversely homogeneous. It precisely means that $\wtF$ is given by the fibers of a submersion $\Phi:\wtX\twoheadrightarrow H$ onto a homogeneous K\" ahler manifold $(H,h)$ such that the Ricci curvature of $h$ is negative and such that $\Phi$ is $\rho$-equivariant with respect to a representation 
$$\rho: \pi_1(X) \To \Isom(H) $$
whose image consists in a subgroup of holomorphic isometries acting densely on $H$. By a result of Borel \cite[Theorem~4]{Borelcoset54}, $H$ admits a structure of a homogeneous holomorphic fibre bundle whose base is a homogeneous bounded domain (\emph{i.e.} a Hermitian symmetric space) $\scrH$ and whose fiber is a flag manifold $F$. As $F$ is rational algebraic, it is reduced to a point, due to the quasi-negativity of the Ricci curvature. Finally we get that $H=\scrH$ as desired.   

\section{Automorphism group of the foliation in the K\" ahler case}\label{S:transversefinite}

In this section, $(X,\F)$ is a foliated compact K\"ahler manifold ($\F$ holomorphic) which satisfies the hypothesis of Theorem~\ref{T:transversefinite}, \emph{i.e.} $\F$ admits an invariant transverse Kähler metric having quasi-negative Ricci curvature. We will denote by $\knr$ the class of such foliations in the remaining part of this article.

Recall that $\Aut (X,\F)$ denote the group of analytic diffeomorphisms of $X$ preserving the foliation $\F$. Let $\overline g$ be the invariant transverse K\"ahler metric.
\medskip

The following preliminary and simple observations will be proved to be useful.
\begin{lemma}\label{L:orbitinclusion}
Let $V$ be an irreducible analytic complex space. Let $G$ a countable group of biholomorphisms of $V$. Let $f$ be a biholomorphism of $V$ such that for every $x\in V$, $f(x)\in G\cdot x$, the $G$-orbit of $x$. Then $f$ is an element of $G$.
\end{lemma}

\begin{proof} 
For every $g\in G$, set $Z_g:=\{x\in V\mid g(x)=f(x)\}$. We have $\bigcup_{g\in G} Z_g=V$ by assumptions. According, to Baire's lemma, there exists $h\in G$ such that $Z_h$ has non-empty interior. We then conclude by analytic continuation that $f=h$, as wanted.
\end{proof}

\begin{lemma}\label{L:CFpreserved}
If $f\in \Aut (X,\F)$, then $f_*({ \mathcal C}_\F)={ \mathcal C}_\F$.
\end{lemma}
\begin{proof}
Firstly, note that $\overline h=f_* \overline g$ is also a transverse K\"ahler metric. On the other hand, ${ \mathcal C}_\F$ is independant of the Riemannian transverse structure, as noticed in \cite[Proposition~5.1]{Molino-livre}, whence the result.
\end{proof}

\begin{cor}\label{C:bothpreserved}
The following inclusion holds true:
$$ \Aut (X,\F)\subset  \Aut (X,\overline\F)\cap  \Aut (X,\G).$$
\end{cor}
\begin{proof}
The first inclusion $ \Aut (X,\F)\subset  \Aut (X,\overline\F)$ is a consequence of the topological characterization of $\overline \F$. The second one $ \Aut (X,\F)\subset  \Aut (X,\G)$ is due to the fact that $\G$ is defined by the vanishing locus of the isotropy Lie subalgebra  of ${ \mathcal C}_\F$ (\emph{cf.} \S\ref{SS:vanishingloci}) combined with Lemma~\ref{L:CFpreserved}.
\end{proof}

During the proof of Theorem~\ref{T:transversefinite}, we will be led to replace the initial manifold $X$ with finite \'etale covers.  They can be chosen in such a way that the action of the group $G$ lifts, as shown in the next result.
\begin{lemma}\label{L:liftexistence}
Let $X_1$ be a compact complex manifold and $X_2$ be a finite \'etale cover of $X_1$. Let $G=\Aut(X_1)$ be the group of biholomorphisms of $X_1$. Then, there exists a finite \'etale cover $X_3$ of $X_2$ such that every $g\in G$ lifts to a biholomorphism of $X_3$. In particular, Theorem~\ref{T:transversefinite} holds on $X$ iff it holds on a finite \'etale cover.
\end{lemma}
\begin{proof}
It is clearly sufficient to find a finite index characteristic subgroup $H$ of $\pi_1(X_1)$ contained in $\pi_1(X_2)$. Because $\pi_1(X_1)$ is finitely generated, we can take $H$ to be the (finite) intersection of all subgroups of $\pi_1(X_1)$ of the given index $[\pi_1(X_1):\pi_1(X_2)]$.
\end{proof}
 
\begin{proof}[Proof of Theorem \ref{T:transversefinite}]
Let $G=\Aut(X,\F)$. Denote by $\widetilde{G}$ the group of biholomorphisms of $\wtX$ which descend to $X$ as an element of $\Aut(X,\F)$. Remark that $\widetilde{g}$ acts by biholomorphism on the complex orbifold $\wtX /\wtF=\mathscr K \times\mathscr H$. According to Corollary~\ref{C:bothpreserved}, $G$ both preserves $\overline\F$ and $\G$, so that $\widetilde{G}$ acts diagonally on $\wtX /\wtF$ with respect to the above decomposition. We can also suppose, up to taking a finite index subgroup that $G$ acts trivially on the general type orbifold $X/\overline\F$. Let us first show that the transverse action of $G$ on $X/\G$ is finite. For this, recall that up to replacing $X$ by a finite \'etale cover and according to Lemma~\ref{L:liftexistence}, we can assume that the projection of the diagonal action of $\pi_1(X)$ on $\wtX /\wtF=\mathscr K \times\mathscr H$ given by the morphism $(\rho_\mathscr K,\rho_\mathscr H):\pi_1(X)\to \Aut(\mathscr K)\times \Aut(\mathscr H)$ is a dense subgroup $H$ (in the Euclidean topology) of the connected semi-simple real algebraic group $\Aut^0 (\mathscr H)$.	 We can even assume (Selberg's Lemma) that $H$ does not contain any torsion elements. We can now exploit the results of \cite{zuo96} and \cite[Theorem~1]{campanaetal-rep} according to which there exists a dominant quasi-holomorphic map with connected fibers ${ \textrm{Sh}}_{ \rho_{ \mathscr H}}:X\to V$ factorizing the representation $\rho_{ \mathscr H}$ and such that $V$ is a manifold of general type. Here, ${ \textrm{Sh}}_{ \rho_{ \mathscr H}}$ is nothing but the Shafarevich map associated to the representation $ \rho_{ \mathscr H}$. Now from Lemma~\ref{L:CFpreserved} and uniqueness of the fibration defined by ${ \textrm{Sh}}_{ \rho_{ \mathscr H}}$, we can infer that every $f\in G$ descend to $V$ as a birational transformation of $V$. This immediately implies that $G$ acts trivially on $V$ (again up to extracting a finite index subgroup). It remains to observe that in a suitable neighborhood $U$ of a general fiber $F$ of ${ \textrm{Sh}}_{ \rho_{ \mathscr H}}$ (where the restriction of  $\rho_{ \mathscr H}$ is thus trivial),  $\G$ is defined by the levels of a holomorphic submersion $\varphi_U:U\to \mathscr H$ which takes its values in the bounded domain $\mathscr H$. This map is thus constant on $F$ and we can then conclude that $F$ is contained in a leaf of $\G$ and eventually that $G$ acts trivially on  $X/\G$. 

At this stage, we have thus showed that up to finite index, $G$ acts both trivially on the space of leaves $X/\overline\F$ and $X/\G$. However, as the leaves intersection of these two foliations may consist of infinitely many leaves of $\F$, one cannot conclude directly.

Let us consider the lift of this  action $\lambda_G:\widetilde{G}\to \Aut(\mathscr K)\times \Aut(\mathscr H)$. According to Lemma~\ref{L:orbitinclusion}, the image of $\lambda_G$ lies in $\rho_\mathscr K (\pi_1(X))\times\rho_\mathscr H( \pi_1(X))$. In particular, we can restrict $\lambda_G$ to ${ \widetilde{G}}_0$ where
\[ \widetilde{G}_0={ \lambda_G}^{-1}\left( \{\textrm{Id}_{\mathscr K}\}\times\rho_\mathscr H( \pi_1(X))\right).\]
Indeed, ${ \widetilde{G}}_0$ still projects to $G$. Let $o$ in ${\mathscr K}$ and let ${ \mathfrak L}_o$ be the leaf of $\widetilde{ \overline \F}$ through $o$. Let $q$ be the image of $o$ by the covering map $\mathscr K \to X/\overline\F$ and let ${ \mathcal L}_q$ be the leaf of ${ \overline \F}$ through~$q$. The covering map ${ \mathfrak L}_o\to { \mathcal L}_q$ is Galois with deck tranformations group the stabilizer ${ \pi_1(X)}_o$ of ${ \mathfrak L}_o$ with respect to the $\pi_1 (X)$-action on the leaves of $\widetilde{ \overline \F}$.  The pull-back on ${ \mathfrak L}_o$ of the foliation $\F_{|{ \mathcal L}_q}$  is thus given by the fibers of the restriction of the surjective submersion $\Phi:\wtX\to \mathscr H$ defining $\widetilde{\G}$. Let us denote it by $\Phi_o:{ \mathfrak L}_o\twoheadrightarrow \mathscr{H}$. Indeed, $\Phi_o$ is a topologically trivial fibration over the simply-connected space $\mathscr{H}$, so that the fibers are connected.    By construction  $\Phi_o$ is ${ \pi_1(X)}_o$-equivariant, that is 
$$\forall \gamma\in { \pi_1(X)}_o,\ \forall x\in { \mathfrak L}_o, \Phi_o (\gamma(x))= \rho_\mathscr H (\gamma)(\Phi_o (x)).$$

Denote by  $\rho_{ \mathscr H}^q$ the restriction of $\rho_{ \mathscr H}$ to $\pi_1({\mathcal L}_q)$. By definition of ${\mathcal L}_q$, this implies that the image of  $\rho_{ \mathscr H}^q$ (=$\rho_{ \mathscr H}({ \pi_1(X)}_o)$) intersects $\Aut^0 (\mathscr{H})$ as a dense subgroup. On the other hand, the space of leaves $X/\overline\F$ is an orbifold, so that we can choose $o$ such the corresponding point $q\in  X/\overline\F$ has trivial isotropy or, in other words, such that the leaf ${ \mathcal L}_q$ has trivial holonomy. This also amounts to saying that the ${ \pi_1(X)}_o$-action on $\wtX/\wtF=\mathscr{K}\times\mathscr{H}$ is trivial on the first factor.

It is thus sufficient to show that the action of $G$ on ${ \mathcal L}_q/\F$ is finite. Actually, thanks to the expression of ${ \widetilde{G}}_0$ and Lemma~\ref{L:orbitinclusion}, any subgroup of $G$ acting trivially on ${ \mathfrak L}_q/\F$ necessarily acts trivially on $X/\F$.
 
As before the latter fact can be established by considering the Shafarevich morphism  ${ \textrm{Sh}}_{ \rho_{ \mathscr H}^q}:{ \mathcal L}_q\to V_q$, following exactly the same line of ideas.
\end{proof}
\medskip

It is very likely that we can enlarge the setting of Theorem~\ref{T:transversefinite} by considering the group $\mathrm{Bim}(X,\F)$ of bimeromorphic transformations preserving the foliation $\F$ with essentially the same proof. %Here, $f\in \mathrm{Bim}(X,\F)$ is said to preserve the foliation leafwise if there exists a non-empty Zariski open subset of $X$ such that the restriction of $f$ to $U$ is a biholomorphism onto its image and such that for every $x\in U$, the leaf ${ \mathcal L}_x$ through $x$ coincides with $ { \mathcal L}_{ f(x)}$.

\section{Some final remarks/questions/partial results in the K\"ahler realm}\label{S:questions}

In this last section, the ambient manifold $X$ is K\"ahler and $\F$ is holomorphic and in the class \knr (see \S\ref{S:transversefinite}). Let us recal that the inclusion $N\F^*\subset\Omega^1_X$ induces $\bigwedge^pN\F^*\subset \Omega^p_X$. We would like to relate the fact that $\F\in\knr$ with the positivity properties of $L$. Let us first recall the following terminology introduced in~\cite{Campana2004} and in~\cite{Wu2020}.

\begin{dfn}\label{def:Bog sheaf}
Let $L\subset \Omega^p_X$ be a saturated rank-one subsheaf of $\Omega^p_X$. Then $L$ is said to be:
\begin{itemize}
\item[(BNS)] a \emph{numerical Bogomolov sheaf} if $\nd(L)=p$,
\item[(BS)] a \emph{Bogomolov sheaf} if $\kappa(L)=p$.
\end{itemize}
\end{dfn}

The Iitaka--Kodaira and the numerical dimension of a line bundle $L$ satisfying $\kappa(L)\le\nd(L)$, we see that $L$ is NBS if it is BS.

\begin{remark}\label{rem:max numerical dim}
According to \cite[Th\'eor\`eme~3.2.12]{Seb-Bouc}), $p$ is the maximum possible value for the numerical dimension (and thus for the Iitaka--Kodaira) of rank-one subsheaves of~$\Omega^p_X$. 
\end{remark}

\subsection{Geometry of foliations induces by NBS}

Let us recall that Bogomolov sheaves have a geometric interpretation, as proved in~\cite{Campana2004}.

\begin{prop}[\emph{cf.} \protect{\cite[Theorem~2.26]{Campana2004}}]\label{prop:BS general type}
Let $L\subset\Omega^p_X$ be BS. It induces a general type fibration\footnote{We refer the reader to \cite{Campana2004} for the notion of general type fibration and of special manifold.} (its Iitaka--Kodaira map)
\[\varphi_L\colon X\xdashrightarrow{\ \quad} Y\quad \text{with} \quad\dim(Y)=p.\]
In particular, the foliation $\F_L:=\Ker(L)$ is algebraically integrable.
\end{prop}

\begin{remark}\label{rem:BS+transverse Kähler}
Let $\F$ be a holomorphic foliation in the class \knr. Then if
\[L:=\left(\bigwedge^p N\F^*\right)^\mathrm{sat}\subset \Omega^p_X \quad ( \text{where}\ p=\codim(\F))\]
is BS, it implies that $\overline{\F}=\F$ and $\F$ is given by a fibration.
\end{remark}

We now explore the NBS case. If $L$ is NBS, a result of Demailly \cite{Demailly} asserts that $\F:=\Ker(L)$ is an integrable distribution and that moreover 
\begin{equation}\label{E:demaillyinteg}
\Theta\wedge \omega=0
\end{equation} where $\omega$ is a local generator of $L$ and  $\Theta$ is any closed positive current representing $c_1(L)$. The foliations we deal with include those for which we can choose $\Theta$ to be a smooth closed  positive $(1,1)$ form of rank $p$ everywhere, \emph{i.e.} those with negative transverse Ricci curvature. 

\begin{prop} \label{P:HNRclass}
Let $L\subset \Omega^p_X$ be NBS and assume that $\Theta$ can be chosen smooth and of rank $p$ everywhere in \eqref{E:demaillyinteg}. Then $\F=\Ker(L)$ is in the class \knr.
\end{prop}

\begin{proof}
The vanishing of $L\wedge\Theta=0$ Indeed implies that the kernel of $\Theta$ is exactly the tangent bundle to $\F$ and in particular that $\Theta$ is (the fundamental form of) a holonomy transverse invariant K\" ahler metric for $\F$.  We get thus another real basic $(1,1)$-form, namely the transverse Ricci form $\alpha=-\mathrm{Ricci}(\Theta)$. 

Note that $-\alpha$ also represents $c_1(N_\F^*)$, so that there exists, by the $\partial\overline{\partial}$-lemma, a smooth function $f:X\to \R$ such that
\[-\alpha=\Theta +\sqrt{-1}\partial\overline{\partial}f.\]
Moreover, $\partial\overline{\partial}f$ is basic as it is a sum of two basic forms.  This implies that $f$ is pluriharmonic along the leaves of $\F$. It turns out that $f$ is basic.

Let us justify the last claim by considering $\mathcal L$ a leaf of $\F$ and $\overline{\mathcal L}$ its topological closure. Let $x\in\overline{\mathcal L}$ such that $\restr{f}{\overline{\mathcal L}}$ reaches its maximum at $x$ and let ${\mathcal L}_x$ be the leaf passing through $x$. By the maximum principle for pluriharmonic functions, $f$ is constant on ${\mathcal L}_x$ hence on $ \overline{ {\mathcal L}_x}$. On the other hand, the leaves closure form a partition of $X$ (\emph{cf.}~\cite[Theorem~5.1]{ Molino-livre}) and we get $\overline{\mathcal L}=
\overline{ {\mathcal L}_x}$. As the original leaf $\mathcal L$ has been chosen arbitrarily,  we can conclude that $f$ is leafwise constant, as wanted. 

Then, $\alpha$ and $\eta$ are not only cohomologous in the ordinary $\partial\overline{\partial}$ cohomology, but also in the \textit{basic} $\partial\overline{\partial}$ cohomology. By the foliated version of Yau's solution to Calabi's conjecture \cite[Section~3.5]{Elkacimi90}, for $\F$ admits an invariant transverse K\"ahler metric whose Ricci form is equal to $-\Theta$.
\end{proof}

\medskip
In fact, these simple models could provide an insight of what is likely to happen for general NBS for which we expect a singular version of the decomposition theorem to hold. For codimension one foliations, this vague expectation can be turned into reality, and the picture is even more precise according to the
\begin{thm}[\emph{cf.} \protect{\cite{Touzetconormal} and \cite[Theorem~D]{Rousseau-nonspecial}}]\label{TH:codim1}
Let $(X,\F)$ be a foliated K\"ahler manifold such that $\F$ is a holomorphic codimension one  foliation given as the annihilator of a NBS $L\subset \Omega_X^1$. Assume  that $\F$ is not  algebraically integrable. Then, up to replacing $X$ by a non-singular K\"ahler modification, there exists a morphism  $\Psi: X\to {\mathbb D}^N/\Gamma$ whose image has dimension $p\geq 2$ such that $\F=\Psi^*\G$ where $\G$ is one of the tautological  foliation on $ {\mathbb D}^N/\Gamma$.
\end{thm}

It is worth mentioning that in this situation, the analogue of the  representation provided by the item~\eqref{I:denserepresentation} of Theorem~\ref{T:CdeRhamdec} is given by a morphism
\[\rho: \pi_1(X\setminus H)\longrightarrow \mathrm{pr}(\Gamma)\subset \Aut(\mathbb D)\]
with dense image and where $H$ is an $\F$-invariant hypersurface.  It is in particular of arithmetic nature (see the comments in \S\ref{SS:examples}).

In the general case of a singular foliation $\F$ defined by a BNS, it is then natural to expect that $\F$ should be obtained as  (the saturation of) the intersection of two foliations, says $\overline{\F}$ and $\G$ of lower codimension, where the leaves of $\overline{\F}$ are topological closure of  general leaves of $\F$ and $\G$ is transversely modelled on a bounded symmetric domain. 

\medskip
On a more realistic side one can also ask if the properties listed below hold for general NBS (note that the answer to the questions below if positive for foliations in \knr).

\begin{question}
Let $L\subset \Omega_X^p$ a NBS. Does the foliation $\F=\Ker(L)$ have codimension $p$?
\end{question}
Note that the answer is affirmative if we assume that $c_1 (L)$ can be represented by a positive current whose absolute continuous part has rank $p$ on a Lebesgue measurable set with non zero measure. Indeed, it can easily be deduced from \eqref{E:demaillyinteg} that $\Theta\wedge T_\textrm{ac}=0$ and by a pointwise calculation, we can conclude that the kernel of $\Theta$ has generically codimension $p$. We expect that such a current always exists.

\begin{question}\label{Q:koddim}
Let $L\subset \Omega_X^p$ be a NBS. Is this true that $\kappa(L)\in \{-\infty, p\}$.
\end{question}

Let us give some evidences supporting a positive answer to the last question. If $L$ is a NBS, any positive closed $(1,1)$ current $T$ representing $c_1(L)$ is invariant by the foliation $\F$ defined by $L$ by virtue of Demailly's identity~\eqref{E:demaillyinteg}. In particular, it holds when $T=[D]$ is the integration current along a $\mathbb Q$-effective divisor $D$. When $\F$ is in the class \knr, we can observe, according to the dynamical behavior of $\F$ as stated in \ref{T:A} that  the support $|D|$ of $D$ is saturated by  $\overline{\F}$. In particular, the general leaf of $\overline{\F}$ does not intersect $|D|$. On the other hand, $c_1(L)$ is a semi-positive class whose numerical dimension coincides with the codimension $p$ of $\F$. This prevents the existence of such a divisor $D$ (incompatibility of intersection properties), unless $\F= \overline{\F}$, in which case $\kappa (L)=p$.

\medskip
We can also ask for the analogue of  Theorem~\ref{T:transversefinite}.

\begin{question}\label{qt:non special}
Let $L\subset \Omega_X^p$ be a NBS and $\F=\mathrm{Ker}(L)$ the associated foliation. Is the action of the group of automorphisms (\emph{resp.} bimeromorphisms) of $X$ preserving $\F$ transversely finite?  
\end{question}
\noindent The case $p=1$ of Question~\ref{qt:non special} is settled in \cite{lobiancoetal2022}, at least when $X$ is projective.

\medskip
 The following question is already asked in \cite{Rousseau-nonspecial}. 
\begin{question} \label{Q:nonspecial}
Let $X$ be a compact K\" ahler manifold supporting a NBS $L\subset \Omega_X^p$. Is $X$ non-special? 
\end{question}
A positive answer to this question is provided in \cite{Rousseau-nonspecial} for $p=1$ and for $\F\in\knr$.

\subsection{Metric aspects and hyperbolicity properties}\label{SS:hyperbolicity prop}

We turn our attention to the existence of special metrics on NBS (and their tentative applications to Kobayashi hyperbolicity). Let $L$ be a NBS and $h$ be a singular metric on $L$ with psh local weight $\varphi$. 
On trivializing open sets, it reads $h(x,v)={|v|}^2e^{-\varphi(x)}$. As noticed before, the current $T=dd^c \varphi$ is $\F$-invariant where $\F$ is the foliation defined by $\Ker(\ L)$. When $p=\dim( X)$, it amounts to saying that the canonical bundle $K_X$ is big. According to \cite{BEGZ10}, this implies that $X$ admits a (singular K\"ahler--Einstein metric): one can choose the local potential $\varphi$ such that 
\begin{equation}\label{E:kahlereinstein}
\langle{ ( dd^c\varphi) }^p\rangle= {\sqrt{-1}}^{p^2}e^\varphi \sigma\wedge \overline{\sigma}
\end{equation}
where $\sigma$ is a local generator of $K_X$ suitably chosen and the brackets stand for the \textit{non-pluriplar} product (\emph{cf.} \textit{loc.cit}).

This motivates the following question.

\begin{question}
Let $\F$ be as above. Does $\F$ admit a transverse singular K\" ahler--Einstein metric? In other words, is it possible to choose  local potentials of the curvature current $T$ and some local generators $\sigma$ of $L$ such that the equation \ref{E:kahlereinstein} holds on $X$?
\end{question}

Once again the answer is affirmative
\begin{itemize}
\item[---] when $p=1$ (see \cite{Touzetconormal}),
\item[---] if $\F$ is in the class \knr (and any $p$).
\end{itemize}
This can be shown by using the techniques developped by El Kacimi in \cite{Elkacimi90} (although it is not explicit in\textit{loc.cit}) or the transverse version of the K\"ahler--Ricci flow, see \cite[Theorem~6.3]{Bedulli_2017}.

Finally let us consider the following question related to hyperbolicity properties of the ambient manifold. According to Campana\footnote{A non-special manifold shouldn't support a Zariski dense entire curve.}, a positive answer to Question~\ref{Q:nonspecial} should implies the same for the question below.
\begin{question} \label{Q:weak hyperbolicity conjecture}
Let $X$ be a compact K\" ahler manifold carrying $L\subset \Omega_X^p$ a NBS. Is it true that for any 
 entire curve $f: X \to\C$, the image of $f$ is never Zariski dense.
\end{question}

The statement of Theorem~\ref{TH:hyperbolicity} partially confirm this expectation in the case $p=1$.

\begin{proof}[Proof of Theorem~\ref{TH:hyperbolicity}]
As already noticed in the proof of Theorem \ref{T:harmonic-holomorphic} (from which we borrow notation),  up to replacing $X$ by a finite étale cover, we can assume that the representation $\rho$ considered in \eqref{E:repautg} takes value in the identity component $\Aut^0 (\mathfrak g)$ and then splits as a sum  of irreducible subrepresentations $\rho=\bigoplus_{i=1}^k \rho_i$ corresponding to the splitting $\mathfrak g= \bigoplus_{i=1}^k { \mathfrak g}_i$ as a sum of simple Lie algebras. By projecting onto the first factor, we can thus assume that $\rho$ takes its values in the real points of an absolutely simple linear algebraic group $G$ defined over $\Q$:
\[G(\R)=  \Aut^0 (\mathfrak g)=\Isom^0(\scrH)\]
where $\scrH$ is an irreducible Hermitian symmetric domain (non reduced to a point as a consequence of the assumption $\F\neq\overline{\F}$). We also have a $\rho$-equivariant holomorphic map (\emph{cf.} Section~\ref{S:equivariant-foliated-harmoni-maps})
\[\Psi:\wtX\To \scrH\]
and let us finally recall that $\rho$ has dense image in $G(\R)$ and consequently is Zariski dense as a representation in the complex algebraic group $G(\C)$. 
		 	
Up to replacing $X$ with a birational modification of an étale cover (it does not affect the answer to Question~\ref{Q:weak hyperbolicity conjecture}), the $\rho$-Shafarevich map is a surjective morphism with connected fibers $f\colon X\to V$ where $V$ is a projective manifold of general type and such that the representation $\rho$ factors through $f$, \emph{i.e.} $\rho=\rho_V\circ f_*$ for a representation $\rho_V\colon \pi_1(V)\to G$. This factorization property and the maximum principle implies that $\Psi$ descends to a $\rho_V$-equivariant holomorphic map $\Psi_V$  on the universal cover $\widetilde V$ of $V$:
\[\begin{tikzcd}
\wtX\ar[rr,"\tilde{f}"] \ar[dd, "r"']\ar[rd,"\Psi"'] & &\widetilde{V} \ar[dl,"\Psi_V"]\ar[dd,"r_V"]\\
& \scrH & \\
X\ar[rr,"f"] & & V.
\end{tikzcd}
\]
Note that the kernel of the differential $d\Psi_V$ defines a (possibly singular) foliation $\F_V$ which descends on $V$ as a foliation whose generic leaf is dense by construction.
		 	
\subsubsection*{Claim 1} The representation $\alpha_V$ is rigid as an element of $\textrm{Hom}(\pi_1(V), G(\C))$.

This is a consequence of the factorization's properties of non rigid representations due to Katzarkov and Zuo \cite{Katzarkov1994}, \cite[Theorem 3]{Zuo99}. Indeed, according to these properties,  a rigidity defect should imply the existence of a projective manifold $Y$ of positive dimension and a surjective morphism $\pi\colon V\to Y$ through which the representation $\rho_V$ factors. In particular the leaves of $\F_V$ should be contained  in the fibers of $\pi$ contradicting the fact that they are dense in $V$.
		 		 
\subsubsection*{Claim 2} Up to replacing $V$ by  a finite étale cover, $\rho_V$ takes its values in $G({ \mathcal O}_K)$ ($K$ a number field) and it underlies a direct factor of a $\mathbb Z$-variation of Hodge structures $\mathbb V$.  

This can deduced from factorization's results of $p$-adically unbouded representations by the same authors  \cite{Katzarkov1994},  \cite[Theorem 4]{Zuo99} together with the characterization of rigid representations in terms of VHS by Simpson  \cite[\S4]{Simpson92}.

\medskip
We can now easily conclude. Let $\Gamma$ be the monodromy group of the local system $\mathbb V$, $\cD$ be the period domain attached to the VHS, and $\varphi\colon V\to \cD/\Gamma$ be the period map. If $g:\mathbb C\to V$ is an entire curve, it is well known since the works of Griffiths and Schmid \cite{GS69} that $g^*\mathbb V$ is trivial. Consequently, the image of $g$ lies in the fibers of $\varphi$. The entire curves on $V$ are thus degenerate and the same is thus true for $X$: the proof of the proposition is now complete.
\end{proof}

\appendix

\section{A foliated Monge--Ampère equation: proof of Theorem~\ref{L:relativemongeampere}} \label{S:proofofmaintechnlemma}

Let us recall that we aim at proving that the equation~\eqref{E:relativemongeampere} has a unique solution $\varphi$ (up to adding a constant). The proof of existence is a non-trivial adaptation of the continuity method in the foliated setting and relies on subtle estimates. The uniqueness of the solution is easier to establish and we first explain this part of the proof.
\begin{proof}[Proof of the uniqueness in Theorem~\ref{L:relativemongeampere}]
Let $\varphi_1$, $\varphi_2$ two solutions of \eqref{E:relativemongeampere} such that  $ \omega + i\partial\bar{\partial}\varphi_i$ is positive in restriction to $\G/\F$ (for $i=1,\,2$).  Set $\omega_i= \omega +\sqrt{-1}\partial\bar{\partial}\varphi_i$. Let $\varphi=\varphi_1-\varphi_2$. Because forms of even degree commute, we obtain
\begin{equation} \label{E:vanishingwedge}
	\partial\bar{\partial}\varphi\wedge\sum_{k=0}^{p-1}\omega_1^k\wedge \omega_2^{ p-k-1}\wedge\Omega^{n-p}=0.
\end{equation}
On the other hand, and thanks to Rummler's formula~\eqref{Rummlerformula}, we have $\beta\wedge d\chi=0$ for every \textit{basic} $(n-1)$-form $\beta$. Combined with the closedness of $\omega_i$ and $\Omega$, Stokes' Theorem yields
$$\int_X d(\varphi d^c\varphi)\wedge\left(\sum_{k=0}^{p-1}\omega_1^k\wedge \omega_2^{ p-k-1}\right)\wedge\Omega^{n-p}\wedge\chi=0$$
When expanding the integrand and according to \eqref{E:vanishingwedge}, we get
	\begin{equation}\label{E:vanishingintegral}
		\int_X \sum_{k=0}^{p-1} d\varphi\wedge d^c\varphi\wedge\omega_1^k\wedge \omega_2^{ p-k-1}\wedge\Omega^{n-p}\wedge\chi=0
\end{equation}
	Let $x\in X$. On some distinguished neighborhood of $x$, we can find holomorphic transverse coordinates $(z_1,\ldots,z_{n})$ such the foliation $\G$ is defined by $\{dz_\alpha=0,\ \alpha>p\}$. Because $\omega_i$ ($i=1,\,2$) are positive in restriction to $\G/\F $, we can moreover choose the $p$ first coordinates $z_1,\ldots,z_p$ such that 
\begin{align*}
\omega_1(x)&= \sqrt{-1}\left(\sum_{\alpha=1}^p dz_\alpha\wedge d\bar{ z_\alpha} +\xi_1\right)(x)\quad \text{and}\\
\omega_2(x)&= \sqrt{-1}\left(\sum_{\alpha=1}^p \mu_\alpha dz_\alpha\wedge d\bar{ z_\alpha} +\xi_2\right)(x)
\end{align*}
where the $\xi_i (x)$'s vanish identically in restriction to ${  \G/\F}_x $ and the $\mu_\alpha$'s are positive real numbers. Thanks to this writing, we infer that there exist $p$ positive numbers $\nu_1,\ldots,\nu_p$ such that 
$$\left(\sum_{k=0}^{p-1} d\varphi\wedge d^c\varphi\wedge\omega_1^k\wedge \omega_2^{ p-k-1}\wedge\Omega^{n-p}\right) (x)= \sum_{\alpha=1}^p\nu_\alpha \left| \frac{\partial \varphi}{\partial z_\alpha}\right|^2 \left|dz_1\wedge\cdots \wedge dz_{n}\right|^2 (x).$$
Together with \eqref{E:vanishingintegral}, this implies that the real function $\varphi=\varphi_1-\varphi_2$ is constant on the leaves of $\G$, hence constant on the whole of $X$ as $\G$ is minimal. 
\end{proof}
\begin{remark}
Actually, by the same proof, the uniqueness holds as soon as $\varphi\in \mathcal C^2 (X)$. 
\end{remark}

\subsection{Preliminaries}\label{SS:prelproofoflemmaE}
We now turn to the proof of the existence part in Theorem~\ref{L:relativemongeampere}. It strongly relies on El-Kacimi's work \cite{Elkacimi90}, even if the context and presentation differ at multiple places.

The original foliation is equipped with the transverse volume form $\xi=\omega^p\wedge\Omega^{n-p}$ and $X$ is itself equipped with the volume form $dV=\xi\wedge\chi$ where $\chi$ is the characteristic form of $\F$. We can use them to endow the space ${ \mathcal C}^0 (X/\F)$ of \textit{continuous} basic functions with a scalar product $\langle\cdot,\cdot\rangle_X$, namely
$$\langle f_1,f_2\rangle_X=\int_X f_1f_2 dV$$
We can of course assume that $\mathrm{Vol}(X):=\int_X  dV=1$ (up to multiplying $\omega$ by a suitable constant).

The transverse invariant metric $\overline g$ induces by restriction a metric on $\G/\F\subset N\F$  that we denote by the same symbol. One can then consider the corresponding \textit{$\F$-basic $\G$-leafwise Laplacian $\Delta_\G$} as defined in \S\ref{SS:extension}.  Thanks to the choice of $dV$, we can derive, as in the classical setting, the following statement.
 \begin{lemma}
 	For $f,f_1,f_2\in { \mathcal C}^0 (X/\F)\cap { \mathcal C}^2(X)$, we have:
 \begin{enumerate}
 	\item $\int_X \Delta_\G (f)dV=0$. In particular, $f$ is constant if and only if $\Delta_\G (f)\geq 0$.
 	\item $\langle f_1,\Delta_\G (f_2)\rangle=\langle \Delta_\G (f_1),f_2\rangle$
 \end{enumerate}
 \end{lemma}
 \begin{proof}
Let $v$ be a basic section  of class ${ \mathcal C}^1$ of $\G/\F$ . The usual calculation yields $\textrm{div}_\G (v)\xi= d( i_{ v}\omega^p) \wedge \Omega^{n-p}=d( i_{v}\xi)$. According to Lemma~\ref{L:Rummler/Stokes}, this implies that 
\begin{equation}\label{E:intbasicdivergencenull}
\int_X \textrm{div}_\G (v)dV=0.
\end{equation}
The item (1) is then obtained by choosing $v=\Delta_\G(f)$.  The item (2) is a consequence of the standard identities $$\mathrm{div}_\G (f_i v)=\langle\nabla_\G (f_i),v \rangle_{\overline g} + f_i\mathrm{div}_\G (v),\ i=1,2$$ and the vanishing property \eqref{E:intbasicdivergencenull} where we take alternatively $v=\nabla_\G (f_2)$ and $v= \nabla_\G (f_1)$.
 \end{proof}
 
We now perform an analogous construction on the foliated transverse frame bundle $p:(X^\sharp,\F^\sharp)\to (X,\F)$, $\pi_{\F^\sharp}:X^\sharp\to W$ together with its projection map $p$, its basic fibration $\pi_{\F^\sharp}$ and its bundle-like metric $g^\sharp=p^*g\oplus_{ \mathcal H}\vartheta$ as defined in \S\ref{SS:asstrfrbundle} from which we maintain notation.  Let $\chi^\sharp:=p^*\chi$ be the induced characteristic form attached to $\F^\sharp$. Let $\nu$ be the characteristic form of the vertical fibration $\vartheta$. Note that, thanks to the transverse $\F^\sharp$-parallelism (see \S\ref{SS:asstrfrbundle})  $i_v d\nu=0$, $\nu$ is \textit{basic} with respect to $\F^\sharp$. Let us consider
 $\omega^\sharp=p^*\omega$ and $\Omega^\sharp=p^*\Omega$. They can be used to endow $\F^\sharp$ with a transverse volume form 
$\xi^\sharp= {\omega^\sharp}^p\wedge\nu\wedge { \Omega^\sharp}^{n-p }$ that can be completed to a volume form $dV^\sharp$ on $X^\sharp$ by setting $dV^\sharp=\xi^\sharp\wedge \chi^\sharp$. As before we can consider the $\F^\sharp$-basic $\G^\sharp$-leafwise  metric induced by the restriction to $\G^\sharp/\F^\sharp$ of the orthogonal sum ${\overline g}^\sharp:=p^*\overline{g}\oplus_{\overline{\mathcal H}}\vartheta$ (see \S\ref{SS:asstrfrbundle}). Let $\Delta_{ \G^\sharp}$ be the corresponding $\F^\sharp$-basic $\G^\sharp$-leafwise Laplacian.  
Every basic function of class ${ \mathcal C}^r$ ($r\in\llbracket0,\infty\rrbracket$) defined on $X^\sharp$ is of the form $f\circ\pi_{\F^\sharp}$, $f\in\mathcal{C}^r(W)$ and  the $\SO(2n)$-action on the principal bundle $X^\sharp$ projects to an $\SO(2n)$ action on $W$. We can identify the space of orbits $W/\SO(2n)$ with the space of leaves closure of $\F$ and the subspace $\mathcal{C}_{ \SO(2n)}^r(W)\subset \mathcal{C}^r(W)$ of $\SO(2n)$-invariant functions to the space ${ \mathcal C}^0 (X/\F)\cap \mathcal{C}^r(X)$. 
We will denote by $ D_W^2$ the linear differential operator induced on $W$ by $\Delta_{ \G^\sharp}$ (as in \cite[Proposition~2.7.7]{Elkacimi90}). It is given, for every $f\in\mathcal{C}^2(W)$, by the formula:
$$D_W^2 (f)=\Delta_{ \G^\sharp}( f\circ\pi).$$ 
 
Arguing as above, we can prove that the equality
 \begin{equation} \label{E:selfadjointnesssharp}
  \int_{X^\sharp}\Delta_{ \G^\sharp} (f_1\circ\pi_{\F^\sharp})( f_2 \circ\pi_{\F^\sharp})dV^\sharp=\int_{X^\sharp}\Delta_{ \G^\sharp} (f_2\circ\pi_{\F^\sharp})( f_1 \circ\pi_{\F^\sharp})dV^\sharp
 \end{equation}
 holds true for any $\ f_1,f_2\in \mathcal{C}^2(W)$.
Let $dw$ be the volume form\footnote{We assume here that $W$ is oriented; this can be achieved by replacing if necessary $X$ by a double cover. Note also that for any finite \'etale cover of $r:Y\to X$, the commuting sheaves are simply related by $\mathcal C_{r^*\F}=r^*{\mathcal C}_\F$. } on $W$ such that we have 
\begin{equation}\label{E:scalarsharp} \int_{X^\sharp}( f_1\circ\pi_{\F^\sharp})( f_2\circ\pi_{\F^\sharp})dV^\sharp=\int_{W}f_1f_2dw=:\langle f_1,f_2\rangle_W.
 \end{equation}
 for every $f_1,f_2\in \mathcal{C}^0(W)$. In particular, since $\nu$ assigns volume $1$ to the fibers of $X^\sharp\to X$, we have 
$$\langle f_1,f_2\rangle_X=\langle f_1,f_2\rangle_W$$ whenever $f_1$ and $f_2$ are $\SO(2n)$-invariant.

The following lemma establishes some ``basic'' properties of $D_W^2$.
\begin{lemma}\label{L:operatorproperties}
\leavevmode
\begin{enumerate}
	\item $D_W^2 (1)=0$.
	\item $D_W^2$ is a strongly elliptic operator of order $2$, in particular $D_W^2$ satisfies the Hopf's maximum principle and then $\textrm{Ker}( D_W^2)=\R$.
	\item $D_W^2$ is $\SO(2n)$-invariant and for every $f\in \mathcal{C}_{\SO (2n)}^2(W)( ={ \mathcal C}^0 (X/\F)\cap{ \mathcal C}^2 (X))$, $D_W^2(f)=\Delta_\G (f).$
	\item $D_W^2$ is self-adjoint with respect to the scalar product $\langle \cdot,\cdot\rangle_W$, i.e, for every $f_1,f_2\in \mathcal{C}^2(W)$, $\langle f_1,D_W^2( f_2)\rangle_W=\langle D_W^2( f_1),f_2\rangle_W.$
	\end{enumerate}
\end{lemma}
\begin{proof}
\leavevmode

	(1): Obvious.
	
	(2): The differential operator $D_W^2$ has obviously order $\leq 2$ ($=$ order of the basic operator $\Delta_\G$).
	
Let $w\in W$ and $u\in \pi_{\F^\sharp}^{-1}(w)$. Let $\xi\in T_x^* W -\{0\}$ and $f\in { \mathcal C}^2(W)$ such that $f(w)=0$ and $df_w=\xi$. Set $f^\sharp= f\circ \pi_{\F^\sharp}$. 	The differential $d\pi_{\F^\sharp}$ vanishes on $\F^\sharp$ and induces a surjective morphism of vector bundles: $\G^\sharp/ \F^\sharp\twoheadrightarrow T W$ (see Remark~\ref{R:defofG}); then the differential $df^\sharp_u$ induces a non-trivial linear form on $ { ( \G^\sharp/ \F^\sharp)}_u$. Consequently $D_W^2 (f^2)(w)=\Delta_{ \G^\sharp}( (f^\sharp)^2)(u)<0$ by strong ellipticity of the Laplacian. The terms involved in this equality are thus nothing but the principal symbol $\sigma( D_W^2)(w,\xi)$ (up to a factor $\frac{1}{2}$), showing that $D_W^2$ is strongly elliptic of order $2$. 

(3): Recall firstly that the structural group $\SO(2n)$ acts on $X^\sharp$ by basic transformation with respect to both $\F^\sharp$ and $\G^\sharp$ (more precisely the action is tangent to the latter). Moreover this action preserves the transverse metric $\overline{g}^\sharp$. This implies the $\SO(2n)$-invariance of $\Delta_{\G^\sharp}$ and consequently that of $D_W^2$. 

Concerning the last point, we maintain notation/observations from \S\ref{SS:asstrfrbundle}. Let us denote by $p_{U/\F}: (p^{-1}(U)/\F^\sharp,{ \overline{g}}^\sharp)\to ( U/\F,\overline g)$ the Riemannian submersion induced by $p$ and which allows us to identify the source as the orthonormal frame bundle of the target. Denote by $\overline{\G}$ (resp. by $\overline{\G^\sharp}$) the foliation induced by $\G$ (resp. by ${\G^\sharp}$) on $U/\F$ (resp. on $ p^{-1}(U)/\F^\sharp$). Both are related by $\overline{\G^\sharp}= p_{U/\F}^* \overline{\G}$. Now the $\SO(2n)$-action preserves the horizontal distribution $\overline { \mathcal H}$, so that the vertical fibers of $p_{U/\F}$ are totally geodesic (with respect to ${ \overline g}^\sharp$) and in particular minimal.  From Remark \ref{R:basiclaplacian} applied to the vertical foliation, we can deduce that $D_W^2(f)=\Delta_{\G^\sharp}(f\circ\pi_{\F^\sharp})=\Delta_{\G}(f)$, as desired.

(4): Straighforward consequence of the self-adjointness of \eqref{E:selfadjointnesssharp} and \eqref{E:scalarsharp}.
\end{proof}

In the sequel, we will also have to replace the transverse reference metric $\overline g$ by the ``metric'' $\overline{ g_\varphi}$ with fundamental form $\omega_\varphi=\omega+\sqrt{-1}\partial\bar{\partial}\varphi$ where $\varphi$ is a basic function of class $\mathcal C^k$, $k\gg 0$ such that $\omega_\varphi$ is positive in restriction to $\G/\F$, so that $\overline{ g_\varphi}$ is \textit{a priori} only a genuine metric of class ${ \mathcal C}^{k-2}$ in restriction  to $\G/\F$.\footnote{In fact, $\varphi$ is necessarily constant along the leaves of $\G$ (see the proof of Lemma~\ref{L:Ckalpharestr}), so that $\overline{ g_\varphi}$ really is a transverse invariant metric on $N\F$.} This is however enough to construct a $\SO(2n)$-invariant $\F^\sharp$-basic $\G^\sharp$-leafwise metric on $\G^\sharp/\F^\sharp$ by taking the restriction of the orthogonal sum $p^*{\overline g}_\varphi\oplus_{\overline{  \mathcal H}}\vartheta$.\footnote{Recall that the $\SO(2n)$-action on $X^\sharp$ preserve both $\F^\sharp$ and $\G^\sharp$ (and is actually tangent to the latter) and consequently induces an action on  $\G^\sharp/\F^\sharp$.} As before, we can attach to it the  $\F^\sharp$-basic $\G^\sharp$-leafwise (resp. $\F$-basic $\G$-leafwise ) Laplacian $\Delta_{ \G^\sharp}^\varphi$  (resp. $\Delta_{ \G}^\varphi$). In this setting, the relevant volume form on $X$ and $X^\sharp$ are respectively
$dV_\varphi:=\omega_\varphi^p\wedge\Omega^{n-p}\wedge \chi$ and $dV_\varphi^\sharp:=\xi_\varphi^\sharp\wedge \chi^\sharp$ where $\xi_\varphi^\sharp= {\omega_\varphi^\sharp}^p\wedge\nu\wedge { \Omega^\sharp}^{n-p }$ and $\omega_\varphi^\sharp=p^*\omega_\varphi$. Let us denote by  $\langle\cdot,\cdot\rangle_{\varphi,W}$ the corresponding scalar product on ${ \mathcal C}^0(W)$.
As before, $\Delta_{ \G^\sharp}^\varphi$ descends to $W$ as a second order differential $D_{ W}^{ 2,\varphi}$ fulfilling all the items of Lemma~\ref{L:operatorproperties}, replacing $\Delta_{ \G}$ by $\Delta_{ \G^\sharp}^\varphi$ and $\langle\cdot,\cdot\rangle_{W}$ by $\langle\cdot,\cdot\rangle_{\varphi,W}$ \footnote{In view of \S\ref{SS:openness}, one could also argue by considering the space of transverse orthonormal frames attached to the new metric $\bar{g}_\varphi$.}.
\medskip
  
As in the classical context of complex Monge--Amp\`ere equations, we want to apply the continuity method in order to prove Theorem~\ref{L:relativemongeampere}, that is to solve for every $t\in[0,1]$ the family of equations 
\begin{equation*}
(\mathrm{MA})_t:\left.\begin{array}{cl}{(  \omega + \sqrt{-1}\partial\bar{\partial}\varphi_t)}^p\wedge \Omega^{n-p}&=	e^{ f_t}\omega^p\wedge \Omega^{n-p}\\
&=e^{ tf}\dfrac{ \int_X\omega^p\wedge \Omega^{n-p} \wedge\chi}{\int_X e^{ tf} \omega^p\wedge \Omega^{n-p} \wedge\chi}\omega^p\wedge \Omega^{n-p}\end{array}\right.
\end{equation*}
which consists in replacing $f$ by $$f_t:=tf +\log\left({\int_X\omega^p\wedge \Omega^{n-p} \wedge\chi}\right)-\log\left({\int_X e^{ tf} \omega^p\wedge \Omega^{n-p} \wedge\chi}\right)$$ in the right-hand side of \eqref{E:relativemongeampere}. 

Let $k$ be a non-negative integer and $\alpha\in (0,1)$. Consider the Banach space ${ \mathcal C}^{k,\alpha}(W)$ of functions of class ${ \mathcal C}^{k}$ and H\"older exponent $\alpha$  and the closed subspace  ${ \mathcal C}_{\SO(2n)}^{k,\alpha}(W)$ of $\SO(2n)$-invariant functions. The latter can be identified with the space of basic functions on $X$ of class ${ \mathcal C}^{k,\alpha}$. On the local leaf space $U/\F$, observe  that $\varphi_t$ is a solution to $(\mathrm{MA})_t$ iff for every leaf $\overline{ \mathcal L}$ of the foliation ${ \overline\G}$ induced by $\G$ on $U/\F$, the following equation holds:
\begin{equation}\label{E:mongeampontheleafspace}
\restr{ {(  \omega + \sqrt{-1}\partial\bar{\partial}\varphi_t)}^p}{\overline{ \mathcal L}}=	\restr{ e^{ f_t}\omega^p}{\overline{ \mathcal L}}.
\end{equation}

Let $\mathcal{W}$ be a connected open subset of $\overline{ \mathcal L}$ relatively compact in $U/\F$. The following observation will turn out to be useful.
\begin{lemma}\label{L:Ckalpharestr}
Let $\varphi_t$ be a solution of $(\mathrm{MA})_t$. Assume that $\varphi_t$ belongs to the class ${ \mathcal C}^{k,\alpha}$ in restriction to $\mathcal{W}$. Then $\varphi_t\in { \mathcal C}^{k,\alpha} (\mathcal{V})$ where $\mathcal{V}$ is an open neighborhood of $\mathcal{W}$ in $U/\F$.
\end{lemma}
\begin{proof}
Consider the constant sheaf ${ { \mathcal C}_\F}_{|U}$, viewed as a Lie algebra $\mathfrak g$ of Killing vector fields on $U/\F$. These vector fields are also foliated with respect to $\overline\G$. Among those vector fields, those that vanish along $\overline{ \mathcal L}$ form a maximal compact subalgebra $\mathfrak k$ of $\mathfrak g$. Take a decomposition of $\mathfrak g$ as a direct sum of linear subspaces $\mathfrak g=\mathfrak k\oplus\mathfrak p$ (for instance a Cartan decomposition). Note that the dimension of $\mathfrak p$ is the real codimension $d=2(n-p)$ of $\overline\G$. Let $\mathscr B=\{v_1,\ldots,v_d\}$ be a basis of $\mathfrak p$. In particular, $\mathscr B$ provides along $\overline{ \mathcal L}$ a trivialization of the normal bundle $N\mathcal L$. For $(t_1,\ldots,t_d)\in{ ]-\varepsilon, \varepsilon[}^d$, the map $$\Psi:\left\{\begin{array}{rcl}
\mathcal{W}\times { ]-\varepsilon, \varepsilon[}^d	&\longrightarrow  & U/\F \\
\left(x,(t_1,\ldots,t_d)\right)	&\longmapsto  &e^{t_1 v_1}\circ e^{t_2 v_2}\circ\cdots\circ e^{t_d v_d}(x)
\end{array}\right.$$ is a smooth diffeomorphism onto an open subset $\mathcal{V}$ of $U/\F$ and $\mathcal{W}=\Psi (\mathcal{W}\times\{0\})$. The lemma is then a straightforward consequence of this trivialization and the fact $\varphi_t$ is constant along the orbit of any $v\in \mathfrak g$.
\end{proof}

We will also make use of the spaces:
\begin{align*}
E_{k,\alpha}&=\left\{\varphi\in { \mathcal C}_{\SO(2n)}^{k,\alpha}(W)\mid\langle\ \varphi,1\rangle_W=0\right\}\quad\text{and}\\
H_{k,\alpha}&=\left\{h\in { \mathcal C}_{\SO(2n)}^{k-2,\alpha}(W)\mid\langle h,1\rangle_W=\langle 1,1\rangle_W\right\}
\end{align*}
with $k\gg 0$. They are respectively a closed linear subspace of $ { \mathcal C}_{\SO(2n)}^{k,\alpha}(W)$ and a closed affine subspace of  ${ \mathcal C}_{\SO(2n)}^{k-2,\alpha}(W)$. Let us denote by $A_{k,\alpha}$ the subset of $t\in [0,1]$ such that $(\mathrm{MA})_t$ admits a solution $\varphi_t\in { \mathcal C}_{\SO(2n)}^{k,\alpha}(W)$ with the property that $\restr{ \omega_{\varphi_t}}{ \G/\F} >0$. We must prove that $A_{k,\alpha}$ contains $1$ for every $(k,\alpha)$. To this end, let us consider the open set subset of $E_{k,\alpha}$, $${ \mathscr U}_{k,\alpha} :=\left\{\varphi\in E_{k,\alpha} \}\mid\omega_\varphi=\omega+\sqrt{-1}\partial\bar{\partial}\varphi >0\ \textrm{in restriction to } \G/\F\right\}$$ and the map $\mathscr C :{ \mathscr U}_{k,\alpha}\to H_{k,\alpha}$ defined by  $$\mathscr C (\varphi)=\dfrac{ { ( \omega + \sqrt{-1}\partial\bar{\partial}\varphi)}^p\wedge \Omega^{n-p}}{ {\omega^p}\wedge \Omega^{n-p}}.$$
\subsection{Openness of $A_{k,\alpha}$.}\label{SS:openness} 
 A straightforward computation shows that $\mathscr C$ is differentiable and that the differential at $\varphi$ is given by:
 $$d \mathscr{C}_\varphi=-\mathscr C (\varphi)\Delta_\G^\varphi .$$
Now, recall that
$$\Delta_\G^\varphi:{ \mathcal C}_{\SO(2n)}^{k,\alpha}(W)\longrightarrow { \mathcal C}_{\SO(2n)}^{k-2,\alpha}(W)$$ is the restriction of $$D_W^{2,\varphi}:{ \mathcal C}^{k,\alpha}(W)\longrightarrow { \mathcal C}^{k-2,\alpha}(W).$$
By standard elliptic theory and Schauder estimates, the image of the latter coincides with 
$${ \mathscr E}_0 :=\left\{h\in {\mathcal C}^{k-2,\alpha}(W)\mid\langle h,1\rangle\}_{ \varphi,W}=0\right\}$$
(see \cite[Expos\'e~VI]{ConjectureCalabi78}).  We can infer that the image of the former is exactly ${ \mathscr E}_0\cap { \mathcal C}_{\SO(2n)}^{k-2,\alpha}(W)$. Indeed, the image of ${ \mathcal C}_{\SO(2n)}^{k,\alpha}(W)$ is contained in ${ \mathscr E}_0\cap { \mathcal C}_{\SO(2 n)}^{k-2,\alpha}(W)$ thanks to the $\SO(2n)$-invariance of $D_W^{2,\varphi}$. Conversely and still by  $\SO(2n)$-invariance, we can express every element of ${ \mathscr E}_0\cap { \mathcal C}_{\SO(2n)}^{k-2,\alpha}(W)$ as  $D_W^{2,\varphi} (f)$ where $f\in { \mathcal C}^{k,\alpha}(W)$ verifies 
 $$ D_W^{2,\varphi}(v(f))=0$$
 for every fundamental vector field $v$ of the $\SO(2n)$-action. This shows that $v(f)$ is constant, hence identically zero by compactness of $W$,  proving our assertion. 
We have thus established that $d \mathscr{C}_\varphi$ is an isomorphism between the tangent spaces $T_\varphi { \mathscr U}_{k,\alpha}$ and $T_{ \mathscr C (\varphi)}H_{k,\alpha}$. By the inverse function Theorem, we can conclude that $A_{k,\alpha}$ is an \textit{open} (non-empty) subset of $[0,1]$. 
\medskip

It remains to show that $A_{k,\alpha}$ is \textit{closed}. We can mimick without fundamental changes the classical proof given by Yau as explained below.

\subsection{Closedness of $A_{k,\alpha}$}

 Let $t\in A_{k,\alpha}$ and $\varphi_t\in { \mathcal C}_{\SO(2n)}^{k,\alpha}(W)$ such that $\restr{ \omega_{\varphi_t}}{ \G/\F} >0$   the solution of $(\mathrm{MA})_t$ together with the normalization   $\langle\varphi_t,1\rangle_W=0=\langle\varphi_t,1\rangle_X$, so that $\varphi_t$ is necessarily \textit{unique}. It is worth noticing that the properties of the operator $D_W^2$ are mainly used in the item (2) of Step~1 below and that the remaining part only needs to work directly on the original manifold $X$ without any reference to the normal frame bundle and the basic fibration on it.
 
\subsubsection{Step 1: $L^1$ estimate}
The solution $\varphi_t$ satisfies the following estimates.
\begin{enumerate}
 	\item $\Delta_\G (\varphi_t)<p$ (= $\dim_\C( \G/\F)$): this is a straightforward adaptation of the proof of \cite[Exposé VII, Lemme~2.1(i)]{ConjectureCalabi78}.
 	\item the upper-bound $\mathrm{sup}_{X\times A_{k,\alpha}}\varphi_t\leq C_1$ follows from the estimates for the Green function of $D_W^2$ (see \cite[Expos\'e VII]{ConjectureCalabi78} and for instance \cite[Appendix~A]{Alesker2013}).
 	\item We get the $L^1$ uniform bound $\int_X |\varphi_t|dV\leq 2C_1$ as a trivial consequence.
 \end{enumerate}
\subsubsection{Step 2: $L^q$ estimate} This is \cite[Expos\'e VII, Lemme 3.3]{ConjectureCalabi78} where the $L^q$ norm under consideration is computed with respect to the volume form $dV=\omega^p\wedge \Omega^{n-p}\wedge \chi$.
The proof is a direct adaptation of the strategy depicted in \textit{loc.~cit.}  (repeated use of $L^1$ estimates and Sobolev embeddings).

\subsubsection{Step 3: uniform ${ \mathcal C}^0$ estimate.}
This is again obtained as a direct adaptation of the classical situation \cite[Expos\'e VII]{ConjectureCalabi78}. Indeed it follows from $L^q$ estimates that we have:
$$\mathrm{Sup}_{A_{k,\alpha}}{ \|\varphi_t  \|  }_\infty\leq C_2  .  $$

\subsubsection{Step 4: Laplacian ${ \mathcal C}^0$ estimate.}
We follow \textit{verbatim} the strategy and computation made in \cite[Expos\'e VIII]{ConjectureCalabi78} and obtain
$$\mathrm{Sup}_{A_{k,\alpha}}{\|\Delta_{\G} (\varphi_t)\|}_\infty\leq C_3.$$ 

\subsubsection{Last step: higher order estimates}
This part follows closely the strategy\footnote{We could alternatively use \cite[Expos\'e XI]{ConjectureCalabi78}.} developed in \cite[\S14.3]{guedjzeriahi17}. Unlike the preceding steps, this is a purely local result to which we can reduce by considering the equation~\eqref{E:mongeampontheleafspace} and Lemma~\ref{L:Ckalpharestr} (see the comments at the end of \cite[\S14.1.2]{guedjzeriahi17}).

In the first place, we can derive from Evans--Krilov theory the upper bound:
\[\exists\, 0<\beta<\alpha,\ \mathrm{Sup}_{A_{k,\alpha}} { \|\varphi_t\|}_{ {\mathcal C }^{2,\beta}}\leq C_4.\]

As the injection $ {\mathcal C }^{2,\alpha}(X)\hookrightarrow  {\mathcal C }^{2,\beta}(X)$ is compact, we can find for any cluster point $t_0\in [0,1]$ a sequence $(t_n)\in { A_{k,\alpha}}$, $t_n\rightarrow t_0$ and $\varphi\in {\mathcal C }^{2,\beta}(X)$ such that $\varphi_{t_n}$ tends to $\varphi$ with respect to the ${ \mathcal C}^{2,\beta}$ norm. In particular we get that $t_0\in A_{2,\beta}$. 
We can then  deduce by applying inductively Schauder's estimates that    $\varphi\in { \mathcal C }^{k',\beta}(X)$ for every positive integer $k'$. In particular  $t_0$ belongs to $A_{k,\alpha}$. Eventually $A_{k,\alpha}=[0,1]$ and, as $k$ is arbitrarily large, this provides the sough solution. This concludes the proof of Theorem~\ref{L:relativemongeampere}.\hfill\qed

%%%%%%%%%%%%%%%%%%%%%%%
%%%%% BIBLIOGRAPHY
%%%%%%%%%%%%%%%%%%%%%%%

\providecommand{\bysame}{\leavevmode\hbox to3em{\hrulefill}\thinspace}
\providecommand{\MR}{\relax\ifhmode\unskip\space\fi MR }
\providecommand{\MRhref}[2]{%
  \href{http://www.ams.org/mathscinet-getitem?mr=#1}{#2}
}
\providecommand{\href}[2]{#2}

\end{document}